\documentclass[a4paper,11pt,openright,twoside]{article}
\usepackage{geometry}
\usepackage{etex}%
\usepackage[utf8]{inputenc}
\usepackage[english]{babel}
\usepackage[T1]{fontenc}
\usepackage{acronym}%
\usepackage{epstopdf}%
\usepackage{epigraph}%
\usepackage{indentfirst}%
\usepackage{amsmath,amssymb, amsthm, amsfonts}
\usepackage{braket}%
\usepackage{empheq}%
\usepackage{hyperref}
\usepackage{mathtools}
\usepackage{dcpic}%
\usepackage[all,cmtip,2cell]{xy} 
\UseAllTwocells
\input{diagxy}
\usepackage{wrapfig}%
\usepackage{subfig,caption}%
\usepackage{stmaryrd}%
\usepackage{textcomp}%
\usepackage{mathrsfs}%
\usepackage{lmodern}
\usepackage{tikz}

\usetikzlibrary{graphs,quotes,matrix}
\usepackage[all]{xy}

\newcommand{\cd}[2][]{\vcenter{\hbox{\xymatrix#1{#2}}}}
\newcommand\nbd\nobreakdash
\newcommand{\Cat}{{\mathcal{C}\mspace{-2.mu}\it{at}}}
\newcommand{\nCat}[1]{{#1}\hbox{\protect\nbd-}\kern1pt\Cat}	
\newcommand{\ulthreecell}[3][0.5]{\ar@{}[#2] \ar@3?(#1)+/dr  0.2cm/;?(#1)+/ul 0.2cm/_{#3}}

\newcommand{\gpd}{{\mathcal{G}\mspace{-2.mu}\it{pd}}}
\newcommand{\ngpd}[1]{{#1}\hbox{\protect\nbd-}\kern1pt\gpd}
\newcommand{\wgpd}{ \ngpd{\infty}}	
\newcommand{\Twocong}[2][0.5]{\ar@{}[#2] \save ?(#1)*{\cong}\restore}
\newcommand{\Twoeq}[2][0.5]{\ar@{}[#2] \save ?(#1)*{=}\restore}
\newcommand{\Ltwocell}[3][0.5]{\ar@{}[#2] \ar@{=>}?(#1)+/r 0.2cm/;?(#1)+/l 0.2cm/^{#3}}
\newcommand{\Rtwocell}[3][0.5]{\ar@{}[#2] \ar@{=>}?(#1)+/l 0.2cm/;?(#1)+/r 0.2cm/^{#3}}
\newcommand{\Utwocell}[3][0.5]{\ar@{}[#2] \ar@{=>}?(#1)+/d  0.2cm/;?(#1)+/u 0.2cm/_{#3}}
\newcommand{\Dtwocell}[3][0.5]{\ar@{}[#2] \ar@{=>}?(#1)+/u  0.2cm/;?(#1)+/d 0.2cm/^{#3}}
\newcommand{\Ultwocell}[3][0.5]{\ar@{}[#2] \ar@{=>}?(#1)+/dr  0.2cm/;?(#1)+/ul 0.2cm/^{#3}}
\newcommand{\Urtwocell}[3][0.5]{\ar@{}[#2] \ar@{=>}?(#1)+/dl  0.2cm/;?(#1)+/ur 0.2cm/^{#3}}
\newcommand{\Dltwocell}[3][0.5]{\ar@{}[#2] \ar@{=>}?(#1)+/ur  0.2cm/;?(#1)+/dl 0.2cm/^{#3}}
\newcommand{\Drtwocell}[3][0.5]{\ar@{}[#2] \ar@{=>}?(#1)+/ul  0.2cm/;?(#1)+/dr 0.2cm/^{#3}}
\newcommand{\Ulthreecell}[3][0.5]{\ar@{}[#2] \ar@3?(#1)+/dr  0.2cm/;?(#1)+/ul 0.2cm/_{#3}}

\newcommand{\D}{\mathcal{D}}

\newcommand{\M}{\mathbf{M}}

\newcommand{\R}{\mathcal{R}}

\newcommand{\G}{\mathbb{G}}
\newcommand{\p}{\mathbb{P}}

\newcommand{\C}{\mathscr{C}}

\newcommand{\cyl}{\mathbf{Cyl}}

\newcommand{\Tn}{\Theta^{\leq n}_0}

\newcommand{\plus}[1]{\mathop{\amalg}\limits_{#1}}
\renewcommand{\epsilon}{\varepsilon}
\renewcommand{\theta}{\vartheta}
\renewcommand{\rho}{\varrho}
\renewcommand{\phi}{\varphi}

\newcounter{ctr} \numberwithin{ctr}{section}

\theoremstyle{definition}

\theoremstyle{definition}

\theoremstyle{definition}
\newtheorem{defi}[ctr]{Definition}
\theoremstyle{definition}

\theoremstyle{definition}
\newtheorem{rmk}[ctr]{Remark}
\theoremstyle{definition}
\newtheorem{ex}[ctr]{Example}
\theoremstyle{plain}
\newtheorem{thm}[ctr]{Theorem}
\theoremstyle{plain}

\theoremstyle{plain}
\newtheorem{prop}[ctr]{Proposition}
\theoremstyle{plain}
\newtheorem{lemma}[ctr]{Lemma}
\theoremstyle{plain}
\newtheorem{cor}[ctr]{Corollary}

\DeclareMathOperator{\ob}{Ob}

\DeclareMathOperator{\colim}{colim}

\DeclareMathOperator{\height}{ht}
\begin{document}
	\title{A semi-model structure for Grothendieck weak 3-groupoids}
	\author{Edoardo Lanari\thanks{Supported by Macquarie University iMQRes PhD scholarship}}
	\maketitle
\begin{abstract}
In this paper we apply some tools developed in our previous work on Grothendieck $\infty$-groupoids to the finite-dimensional case of weak 3-groupoids. 

We obtain a semi-model structure on the category of Grothendieck 3-groupoids of suitable type, thanks to the construction of an endofunctor $\p$ that has enough structure to behave like a path object. This makes use of a recognition principle we prove here that characterizes globular theories whose models can be viewed as Grothendieck $n$-groupoids (for $0\leq n \leq \infty$). Finally, we prove that the obstruction in arbitrary dimension (possibly infinite) only resides in the construction of (slightly less than) a path object on a suitable category of Grothendieck (weak) $n$-categories with weak inverses. This also gives a sufficient condition for endowing an $n$-groupoid à la Batanin with the structure of a Grothendieck $n$-groupoid.
\end{abstract}
\tableofcontents
\section{Introduction}
Alexander Grothendieck proposed an algebraic definition of  weak $\infty$-groupoids in 1983, in a letter to Quillen, see \cite{Gr}. His idea was to have a completely algebraic model of these sofisticated objects, in contrast with the existing non-algebraic ones (i.e. Kan complexes). Moreover, as it was proven to be true for other models, he conjectured that these algebraic structures modeled all homotopy types. This goes under the name of ``homotopy hypothesis''.

In his recent paper (\cite{Hen}), Henry proved that if a technical condition on the category of Grothendieck $\infty$-groupoids is satisfied, namely the ``pushout lemma'' (which states that pushouts of certain maps induce isomorphisms on homotopy groups) then the homotopy hypothesis holds true. Moreover, this is the only non-trivial step in constructing a semi-model structure on the category of Grothendieck $\infty$-groupoids. 

In our previous work (\cite{EL}) we developed some tools to attack this problem. We constructed a globular set that ought to model a path object for the category $\ngpd{\infty}$, provided one can endow it with the required algebraic structure. We also initiated this construction by interpreting all the categorical operations in a non-functorial way, and we showed how to fix it in low dimensions. Here we prove in Theorem \ref{semi model str on C_W} that essentially this is the only obstruction to the construction of the semi-model structure, and we get the desired extension in the simpler case of a truncated 3-dimensional version of these highly-structured algebraic objects. In fact, we provide this path object with enough structure to make it into an object of $\ngpd{3}$, the category of Grothendieck 3-groupoids (of suitable type), and we use this path object to endow the abovementioned category with a semi-model structure. According to the (generalized) homotopy hypothesis, weak 3-groupoids should model all homotopy 3-types, and this will be the object of study of subsequent work, which will make use of the existence of this semi-model structure.
Unless otherwise stated, all the structures are weak, thus we use the term $n$-groupoids and $n$-categories to mean weak ones.

In Section 2 we recall the basic definitions available in the existing literature, with the necessary modifications needed to adapt them, when appropriate, to the category $\ngpd{n}$ of $n$-groupoids.
This section also contains the definition and some basic properties of a weak factorization system on the category $\ngpd{n}$ for $0\leq n \leq \infty$ that will be used throughout the entire paper.

The core of this work is in Section 3 and 4: in the former we prove a characterization of those globular theories for which the category of models bears a cofibrantly generated semi-model structure whose objects look like Grothendieck $n$-groupoids for $0\leq n \leq \infty$ (though they may possibly be strict $n$-groupoids): this is Theorem \ref{semi model str characterization}, which enables us to prove the main result of Section \ref{the semi model str section}, i.e. Theorem \ref{semi model str on C_W}. This is a very general result which states that it is enough to construct a path object on the category of weak $n$-categories with weak inverses, i.e. $\mathfrak{C}^{\mathbf{W}}$-models (see Definition \ref{D_W def}), in order to show that these are Grothendieck $n$-groupoids and obtain a semi-model structure on said category.

Section 5 is a recollection of all the main constructions performed in \cite{EL}, such as cylinders on globular sums, modifications, and, most importantly, the elementary interpretation $\hat{\rho}\colon \cyl(D_n) \rightarrow \cyl(A)$ of a given homogeneous (categorical) operation $\rho\colon D_n \rightarrow A$, all adapted to work in the case of $\mathfrak{C}^{\mathbf{W}}$-models. We also include the explicit description of a specific instance of the stack of cylinders that we defined abstractly in the previous paper, since this turns out to be useful in some calculations.

In the final section, given a $\mathfrak{C}_3^{\mathbf{W}}$-model $X$, where $\mathfrak{C}_3$ denotes a coherator for 3-categories, we endow the globular set $\p X$ with the structure of a $\mathfrak{C}_3^{\mathbf{W}}$-model: this makes use of results from Section 6 of \cite{EL} together with new content. Essentially, we inductively correct the boundary of $\hat{\rho}$ so as to make it into a functor $\cyl \colon \mathfrak{C}_3^{\mathbf{W}} \rightarrow \mathbf{Mod}(\mathfrak{C}_3^{\mathbf{W}})$. 
Therefore, we get an endofunctor $\p$ on $\mathbf{Mod}(\mathfrak{C}_3^{\mathbf{W}})$ (see Corollary \ref{lift to C_W}) as stated in Proposition \ref{path object for Batanin 3-groupoids}, which concludes this work according to Theorem \ref{semi model str on C_W}.
\section{Background}
The basic shapes that constitute the arities for globular theories are the so-called globular sums, i.e. a suitable notion of pasting of globes that will be introduced in the following section. We then recall the definition of models of a globular theory and their universal property, together with a class of (trivial) cofibrations on the category of such models.
\subsection{Globular theories and models}
\begin{defi}
	Let $\G$ be the category obtained as the quotient of the free category on the graph
	\[
	\bfig
	\morphism(0,0)|a|/@{>}@<2pt>/<300,0>[0`1;\sigma_0]
	\morphism(0,0)|b|/@{>}@<-2pt>/<300,0>[0`1;\tau_0]
	\morphism(300,0)|a|/@{>}@<2pt>/<300,0>[1`\ldots;\sigma_1]
	\morphism(300,0)|b|/@{>}@<-2pt>/<300,0>[1`\ldots;\tau_1]
	\morphism(720,0)|a|/@{>}@<2pt>/<400,0>[n`n+1;\sigma_n]
	\morphism(720,0)|b|/@{>}@<-2pt>/<400,0>[n`n+1;\tau_n]
	\morphism(1120,0)|a|/@{>}@<2pt>/<400,0>[n+1`\ldots;\sigma_{n+1}]
	\morphism(1120,0)|b|/@{>}@<-2pt>/<400,0>[n+1`\ldots;\tau_{n+1}]
	\efig
	\]
	by the set of relations $\sigma_k \circ \sigma_{k-1}=\tau_k \circ \sigma_{k-1}$, $\sigma_k \circ \tau_{k-1}=\tau_k \circ \tau_{k-1}$ for $k\geq 1$.
	
	Given integers $j>i$, define $\sigma^j_i=\sigma_{j-1}\circ \sigma^{j-1}_i$, where $\sigma^{i+1}_i=\sigma_i$. The maps $\tau^j_i$ are defined similarly, with the appropriate changes.
\end{defi}
The category of globular sets is the presheaf category $[\G^{op},\mathbf{Set}]$.
\begin{defi}
	A map $f\colon X \rightarrow Y$ of globular sets is said to be $m$-bijective if $f_k\colon X_k \rightarrow Y_k$ is a bijection of sets for every $k\leq m$, and it is $m$-fully faithful if the following square is cartesian for all $i\geq m$:
	\[
	\bfig 
	\morphism(0,0)|a|/@{>}@<0pt>/<800,0>[X_{i+1}`Y_{i+1};f_{i+1}]
	\morphism(0,0)|l|/@{>}@<0pt>/<0,-400>[X_{i+1}`X_i\times X_i;(s,t)]
	\morphism(800,0)|r|/@{>}@<0pt>/<0,-400>[Y_{i+1}`Y_i\times Y_i;(s,t)]
	\morphism(0,-400)|l|/@{>}@<0pt>/<800,0>[X_{i}\times X_i`Y_{i}\times Y_i;f_i\times f_i]
	\efig 
	\]
	We denote the class of $m$-bijective morphisms by $\mathbf{bij_m}$, and that of $m$-fully faithful ones by $\mathbf{ff_m}$.
\end{defi}
The following result holds true, and its proof is left as a simple exercise
\begin{prop}
	\label{fact syst glob set}
	The pair $(\mathbf{bij_m},\mathbf{ff_m})$ is an orthogonal factorization system on the category of globular sets $[\G^{op},\mathbf{Set}]$.
\end{prop}
 Globes are not enough to capture a meaningful theory of $n$-groupoids, for which we need more complex shapes, called globular sums, which are a special kind of pasting of globes.
\begin{defi}
	A table of dimensions is a sequence of integers of the form 
	\[\begin{pmatrix}
	i_1 &&i_2 & \ldots&i_{m-1} & &i_m\\
	& i'_1 & &\ldots&& i'_{m-1}
	\end{pmatrix}\]
	satisfying the following inequalities: $i'_k<i_k$ and $i'_k<i_{k+1}$ for every $1\leq k\leq m-1$.\\
	Given a category $\mathcal{C}$ and a functor $F\colon \G \rightarrow \mathcal{C}$, a table of dimensions as above induces a diagram of the form
	\[
	\bfig
	\morphism(0,0)|l|/@{>}@<0pt>/<-400,400>[F(i'_1)`F(i_1);F(\sigma_{i'_1}^{i_1})]
	\morphism(0,0)|a|/@{>}@<0pt>/<400,400>[F(i'_1)`F(i_2);F(\tau_{i'_1}^{i_2})]
	\morphism(800,0)|l|/@{>}@<0pt>/<-400,400>[F(i'_2)`F(i_2);F(\sigma_{i'_2}^{i_2})]
	\morphism(800,0)|a|/@{>}@<0pt>/<400,400>[F(i'_2)`F(i_3);F(\tau_{i'_2}^{i_3})]
	\morphism(1300,150)|r|/@{}@<0pt>/<100,0>[ `\ldots;]
	\morphism(2000,0)|l|/@{>}@<0pt>/<-400,400>[F(i'_{m-1})`F(i_{m-1});F(\sigma_{i'_{m-1}}^{i_{m-1}})]
	\morphism(2000,0)|a|/@{>}@<0pt>/<400,400>[F(i'_{m-1})`F(i_{m});F(\tau_{i'_{m-1}}^{i_{m}})]
	\efig 
	\]
	A globular sum of type $F$ (or simply globular sum) is the colimit in $\mathcal{C}$ (if it exists) of such a diagram.
	
	We also define the dimension of this globular sum to be $\dim(A)=\max \{i_k\}_{k \in \{1, \ldots, \ m\}}$. Given a globular sum $A$, we denote with $\iota_k^A$ the colimit inclusion $F({i_k})\rightarrow A$, dropping subscripts when there is no risk of confusion.
\end{defi}
We denote by $\Theta_0$ the full subcategory of globular sets spanned by the globular sums of type $y\colon \G \rightarrow [\G^{op},\mathbf{Set}]$, where $y$ is the Yoneda embedding. Moreover, we denote $y(i)$ by $D_i$ and  the globular sum corresponding to the table of dimensions 
\[\begin{pmatrix}
1 &&1 & \ldots&1 & &1\\
& 0 & &\ldots&& 0
\end{pmatrix}\] by $D_1^{\otimes k}$, where the integer $1$ appears exactly $k$ times.

 In dealing with Grothendieck $n$-groupoids, we will need a truncated version of the category $\G$, which we now introduce.
\begin{defi}
We denote with $\G_n$ the full subcategory of $\G$ generated by the set of objects $\{k\in\G \colon k\leq n\}$. Analogously to the infinite dimensional case, we consider the presheaf category $[\G_{n}^{op},\mathbf{Set}]$, called the category of $n$-truncated globular sets, or simply $n$-globular sets.
\end{defi}
We will always assume $n>0$, to avoid the trivial case of $0$-groupoids, i.e. sets.
Proposition \ref{fact syst glob set} can be extended to the case of $n$-globular sets, when (using the notation of the proposition) $m\leq n$.
If we consider the subcategory $\Tn \subset \Theta_0$ spanned by globular sums of dimension less or equal than $n$, we see that there is a fully faithful embedding functor $\Tn\rightarrow [\G_{n}^{op},\mathbf{Set}]$. The category $\Tn$ plays a similar role for $n$-groupoids as $\Theta_0$ does for $\infty$-groupoids.
\begin{defi}
An $n$-truncated globular theory is a pair $(\mathfrak{E},\mathbf{F})$, where $\mathfrak{E}$ is a category and $\mathbf{F}\colon \Tn \rightarrow \mathfrak{E}$ is a bijective on objects functor that preserves globular sums of dimension less than or equal to $n$.

We denote by $\mathbf{GlTh_n}$ the category of $n$-globular theories and $n$-globular sums preserving functors. More precisely, a morphism $H\colon(\mathfrak{E},\mathbf{F}) \rightarrow (\mathfrak{C},\mathbf{G})$ is a functor $H\colon \mathfrak{E} \rightarrow \mathfrak{C}$ such that $\mathbf{G}=H\circ \mathbf{F}$.
\end{defi}
If there is no risk of confusion we will omit the structural map $\mathbf{F}\colon \Tn \rightarrow \mathfrak{E}$ and simply denote the globular theory $(\mathfrak{E},\mathbf{F})$ by $\mathfrak{E}$.
\begin{defi}
	Given an $n$-globular theory $\mathfrak{E}$, we define the category of its models, denoted $\mathbf{Mod}(\mathfrak{E})$, to be the category of $n$-globular product preserving functors $G\colon \mathfrak{E}^{op} \rightarrow \mathbf{Set}$. Clearly, the Yoneda embedding $y\colon \mathfrak{E} \rightarrow [\mathfrak{E}^{op},\mathbf{Set}]$ factors through $\mathbf{Mod}(\mathfrak{E})$, and it will still be denoted by $y\colon \mathfrak{E} \rightarrow \mathbf{Mod}(\mathfrak{E})$. Also, notice that $\mathbf{Mod}(\Tn)\cong [\G_{n}^{op},\mathbf{Set}]$.
\end{defi}
Again, we denote the image of $i$ under $y$ by $D_i$.
\begin{prop}
	\label{UP of models}
	Given an $n$-globular theory $\mathfrak{E}$, its category of models $\mathbf{Mod}(\mathfrak{E})$ enjoys a universal property: given any cocomplete category $\D$, a cocontinuous functor $F\colon \mathbf{Mod}(\mathfrak{E}) \rightarrow \D$ is determined up to a unique isomorphism by a functor $\overline{F}\colon \mathfrak{E} \rightarrow \D$, corresponding to its restriction along the Yoneda embedding, that preserves $n$-globular sums.
\end{prop}
The 3-groupoids we are going to consider are presented as models of a certain class of 3-globular theories, namely the cellular and contractible ones.
\begin{defi}
	\label{contr glob th}
	Given $k\leq n$, two maps $f,g\colon D_k\rightarrow A$ in an $n$-globular theory are said to be parallel if either $k=0$ or $f\circ \epsilon= g\circ \epsilon$ for $\epsilon=\sigma,\tau$.
	A pair of parallel maps $(f,g)$ is said to be admissible if $\dim(A) \leq k+1$.
	A globular theory $(\mathfrak{C},F)$ is called contractible if for every admissible pair of maps $f,g\colon D_k\rightarrow A$ if $k=n$ then $f=g$, otherwise $k<n$ and there exists an extension $h\colon D_{k+1}\rightarrow A$ rendering the following diagram serially commutative
	\[
	\bfig 
	\morphism(0,0)|a|/@{>}@<3pt>/<500,0>[D_k`A;f]
	\morphism(0,0)|b|/@{>}@<-3pt>/<500,0>[D_k`A;g]
	
	\morphism(0,0)|r|/@{>}@<3pt>/<0,-400>[D_k`D_{k+1};\tau_k]
	\morphism(0,0)|l|/@{>}@<-3pt>/<0,-400>[D_k`D_{k+1};\sigma_k]
	
	\morphism(0,-400)|r|/@{>}@<0pt>/<500,400>[D_{k+1}`A;h]
	\efig 
	\]
\end{defi}
Contractibility ensures the existence of all the operations that ought to be part of the structure of an $n$-groupoid. However, it does not guarantee weakness of the models, and indeed there exists a contractible globular theory (which we denoted by $\tilde{\Theta}^{\leq n}$) whose models are strict $n$-groupoids.

To remedy this, we need the concept of cellularity, or freeness, to restrict the class of globular theories we consider. This notion is based on a slight variation of a construction explained in paragraph 4.1.3 of \cite{AR1},  which we record in the following proposition.
\begin{prop}
	\label{univ prop of glob th}
	Given an $n$-globular theory $\mathfrak{E}$ and set $X$ of admissible pairs in it, there exists another $n$-globular theory $\mathfrak{E}[X]$ equipped with a morphism $\phi\colon \mathfrak{E} \rightarrow \mathfrak{E}[X]$ in $\mathbf{GlTh_n}$ with the following universal property: given an $n$-globular theory $\mathfrak{C}$, a morphism $H\colon\mathfrak{E}[X] \rightarrow \mathfrak{C}$ is determined up to a unique isomorphism by its precomposition $F$ with $\phi$, a choice of an extension as in Definition \ref{contr glob th} for the image under $F$ of each admissible pair $f,g\colon D_k \rightarrow A$ in $X$ with $k<n$ and the requirement that $H(f)=H(g)$ if $k=n$.
\end{prop}
In words, $\mathfrak{E}[X]$ is obtained from $\mathfrak{E}$ by universally adding a lift for each pair in $X$ of non-maximal dimension and by equalizing parallel $n$-dimensional operations in $X$.
\begin{defi}
	\label{cell glob th}
	An $n$-globular theory $\mathfrak{E}$ is said to be cellular if there exists a functor $\mathfrak{E}_{\bullet} \colon \omega \rightarrow \mathbf{GlTh_n}$, where $\omega$ is the first infinite ordinal, such that:
	\begin{enumerate}
		\item $\mathfrak{E}_0 \cong \Tn$;
		\item for every $m \geq 0$, there exists a family $X$ of admissible pairs of arrows in $\mathfrak{E}_m$ (as in Definition \ref{contr glob th}) such that $\mathfrak{E}_{m+1}\cong \mathfrak{E}_m[X]$;
		\item $\colim_{m \in \omega}\mathfrak{E}_{m}\cong \mathfrak{E}$.
	\end{enumerate}
\end{defi}
As anticipated earlier, we now define the class of $n$-globular theories which are appropriate to develop a theory of $n$-groupoids.
\begin{defi}
An $n$-truncated (groupoidal) coherator, or, briefly, an $n$-coherator, is a cellular and contractible $n$-globular theory. Given an $n$-coherator $\mathfrak{G}$, the category of $n$-groupoids of type $\mathfrak{G}$ is the category $\mathbf{Mod}(\mathfrak{G})$ of models of $\mathfrak{G}$. In what follows, $\mathfrak{G}$ will always denote a coherator for $n$-groupoids, with $0\leq n \leq \infty$, and sometimes we will denote the category of its models by $\ngpd{n}$, with no reference to $\mathfrak{G}$.
\end{defi}
The restriction of an $n$-groupoid $X\colon\mathfrak{G}^{op} \rightarrow \mathbf{Set}$ to ${\Tn}^{op}$ gives an object of $\mathbf{Mod}(\Tn)\simeq [\G_n^{op},\mathbf{Set}]$, which we call the underlying $n$-globular set of $X$. The set $X_i$ represents the set of $i$-cells of $X$ for each $i\leq n$.

Let us now consider the algebraic structure acting on these sets of cells. Section 3 of \cite{AR2} shows how to endow the underlying globular set of an $\infty$-groupoid with all the sensible operations it ought to have to deserve to be called such. A completely analogous argument applies, mutatis mutandis, to the case of $n$-groupoids.

For example, we can build operations that represent binary composition of a pair of $1$-cells, codimension-$1$ inverses for $2$-cells and an associativity constraint for composition of $1$-cells by solving, respectively, the following extension problems:
\[
\bfig 
\morphism(0,0)|a|/@{>}@<3pt>/<600,0>[D_0`D_1\plus{ D_0}D_1;i_0\circ \sigma_0]
\morphism(0,0)|b|/@{>}@<-3pt>/<600,0>[D_0`D_1\plus{ D_0}D_1;i_1\circ \tau_0]

\morphism(0,0)|r|/@{>}@<3pt>/<0,-400>[D_0`D_{1};\tau_0]
\morphism(0,0)|l|/@{>}@<-3pt>/<0,-400>[D_0`D_{1};\sigma_0]

\morphism(0,-400)|r|/@{>}@<0pt>/<600,400>[D_{1}`D_1\plus{ D_0}D_1;\nabla^1_0]

\morphism(1000,0)|a|/@{>}@<3pt>/<500,0>[D_1`D_2;\tau_1]
\morphism(1000,0)|b|/@{>}@<-3pt>/<500,0>[D_1`D_2;\sigma_1]

\morphism(1000,0)|r|/@{>}@<3pt>/<0,-400>[D_1`D_{2};\tau_1]
\morphism(1000,0)|l|/@{>}@<-3pt>/<0,-400>[D_1`D_{2};\sigma_1]

\morphism(1000,-400)|r|/@{>}@<0pt>/<500,400>[D_{2}`D_2;\omega^2_1]

\morphism(2000,0)|a|/@{>}@<3pt>/<1400,0>[D_1`D_1 \plus{ D_0} D_1 \plus{ D_0} D_1;(\nabla^1_0\plus{ D_0} 1_{D_1})\circ \nabla^1_0]
\morphism(2000,0)|b|/@{>}@<-3pt>/<1400,0>[D_1`D_1 \plus{ D_0} D_1 \plus{ D_0} D_1;(1_{D_1} \plus{ D_0} \nabla^1_0 )\circ \nabla^1_0]

\morphism(2000,0)|r|/@{>}@<3pt>/<0,-400>[D_1`D_{2};\tau_1]
\morphism(2000,0)|l|/@{>}@<-3pt>/<0,-400>[D_1`D_{2};\sigma_1]

\morphism(2000,-400)|r|/@{>}@<0pt>/<1400,400>[D_{2}`D_1 \plus{ D_0} D_1 \plus{ D_0} D_1;\alpha]
\efig 
\]
In a similar fashion one can build every sensible operation an $n$-groupoid ought to be endowed with.

Whenever a choice of such operations is understood, at the level of models (i.e. $n$-groupoids) we denote with the familiar juxtaposition of cells the (unbiased) composition of them, and with the exponential notation $A^{-1}$ we denote the codimension-$1$ inverse of an $m$-cell $A$.

We will need to choose some operations once and for all, so we record here their definition.
Choose an operation $\nabla^1_0\colon D_1 \rightarrow D_1 \amalg_{D_0} D_1$ as above, and define $w=\nabla^1_0$.
Next, pick operations $D_2 \rightarrow D_2 \amalg_{D_0} D_1$ and $D_2 \rightarrow D_1 \amalg_{D_0} D_2$ whose source and target are given, respectively by $\left((\sigma\amalg_{ D_0}1 )\circ w,(\tau \amalg_{ D_0}1 )\circ w\right)$ and $\left((1\amalg_{ D_0} \sigma) \circ w,(1\amalg_{ D_0} \tau) \circ w\right)$. Proceeding in this way we get specified whiskering maps for every $k\leq n$ of the form:
\begin{equation}
\label{w's maps}
_{k}w\colon D_k \rightarrow D_k \plus{D_0} D_1 $$ 
$$w_k\colon D_k \rightarrow D_1 \plus{D_0} D_k
\end{equation}
We will often avoid writing down all the subscripts, when they are clear from the context.
\begin{defi}
	\label{whiskering w}
	Given a globular sum $A$, whose table of dimensions is
	\[\begin{pmatrix}
	i_1 &&i_2 & \ldots&i_{m-1} & &i_m\\
	& i'_1 & &\ldots&& i'_{m-1}
	\end{pmatrix}\]
	we define a map $_A w\colon A\rightarrow A \amalg_{ D_0} D_1$ by \[w_{i_1 +1 }\plus{w_{i'_1 +1} } \ldots \plus{w_{i'_{m-1} +1 }} w_{i_m +1}\colon D_{i_1 +1 }\plus{D_{i'_1 +1} } \ldots \plus{D_{i'_{m-1} +1 }} D_{i_m +1} \rightarrow (D_{i_1 +1 }\plus{D_{i'_1 +1} } \ldots \plus{D_{i'_{m-1} +1 }} D_{i_m +1})\plus{D_0}D_1  \] noting  that the target is isomorphic to \[(D_{i_1 +1 }\plus{D_0}D_1)\plus{D_{i'_1 +1} \plus{D_0} D_1 } \ldots \plus{D_{i'_{m-1} +1 } \plus{D_0} D_1} (D_{i_m +1} \plus{D_0} D_1)\]
	In a completely analogous manner we define a map $w_A\colon A \rightarrow D_1\amalg_{D_0}A$.
\end{defi}
Consider the forgetful functor \[\mathbf{U_n}\colon \mathbf{Mod}(\mathfrak{G}) \rightarrow \mathbf{Mod}(\Tn) \simeq [\G_n^{op},\mathbf{Set}] \] induced by the structural map $\Tn \rightarrow \mathfrak{G}$. Given a map of $n$-groupoids $f\colon X \rightarrow Y$ and a natural number $m\leq n$, we can factor the map $\mathbf{U_n}(f)$ as $\mathbf{U_n}(f)=g\circ h$, where $h$ is $m$-bijective and $g$ is $m$-fully faithful thanks to Proposition \ref{fact syst glob set}. It is not hard to see that the target of $h$ can be endowed with the structure of an $n$-groupoid so that $g$ and $h$ are maps of such. This fact, thanks to Proposition 2 of \cite{BG}, provides the following result that will be used in this paper.
\begin{prop}
	\label{fact of maps of gpds}
	Given $m\leq n$, the orthogonal factorization system $(\mathbf{bij_m},\mathbf{ff_m})$ on $n$-globular sets lifts to one on $\mathbf{Mod}(\mathfrak{G})$ via the forgetful functor $\mathbf{U_n}\colon \mathbf{Mod}(\mathfrak{G}) \rightarrow [\G_n^{op},\mathbf{Set}]$.
\end{prop}
This means, in particular, that every map in $\mathbf{Mod}(\mathfrak{G})$ admits a unique factorization $f=g\circ h$ where $\mathbf{U_n}(h)$ is $m$-bijective and $\mathbf{U_n}(g)$ is $m$-fully faithful, and that $m$-bijective maps are closed under colimits in $\mathbf{Mod}(\mathfrak{G})$.
\begin{ex}
	\label{source/targets are n-2 bij}
	The maps $\sigma_k,\tau_k\colon D_k \rightarrow D_{k+1}$ are $(k-2)$-bijective. Indeed, since the forgetful functor $\mathbf{U_n}$ preserves the right class of the factorization system $(\mathbf{bij_k},\mathbf{ff_k})$ on $\mathbf{Mod}(\mathfrak{G})$ for every $k\leq n$, its left adjoint $\mathbf{F_n}\colon [\G_n^{op},\mathbf{Set}] \rightarrow \mathbf{Mod}(\mathfrak{G})$ preserves the left class. Now it is enough to observe that $\mathbf{F_n}$ sends source and target maps of globular sets to source and target maps of $n$-groupoids, and for the former it is easy to check the statement on $(k-2)$-bijectivity.
\end{ex}
Given a globular sum $A$ such that $\dim(A)=m>0$, whose table of dimensions is \[\begin{pmatrix}
i_1 &&i_2 & \ldots&i_{q-1} & &i_q\\
& i'_1 & &\ldots&& i'_{q-1}
\end{pmatrix}\]
we define its boundary to be the globular sum whose table of dimensions is 
\[\begin{pmatrix}
\bar\imath_1 &&\bar\imath_2 & \ldots&\bar\imath_{q-1} & &\bar\imath_q\\
& i'_1 & &\ldots&& i'_{q-1}
\end{pmatrix}\]
where we set \[\bar\imath_k=\begin{cases}
i_k-1&\text{if} \ i_k= m\\
i_k&\text{otherwise}
\end{cases}\]
The maps $\sigma,\tau\colon D_{m-1} \rightarrow D_{m}$ induce maps 
\begin{equation}
\label{partial defi}
\partial_{\sigma},\partial_{\tau}\colon \partial A \rightarrow A
\end{equation}
Thanks to what we observed in Example \ref{source/targets are n-2 bij}, we have the following result.
\begin{prop}
	\label{partial are n-2 bij}
	Given a globular sum $A$, with $0<m=\dim(A)$, the maps $\partial_{\sigma},\partial_{\tau}\colon \partial A \rightarrow A$ are $(m$-$2)$-bijective.
\end{prop}
Let us now see how to adapt the main definitions to the case of $n$-categories, following \cite{AR1}.
The definition is essentially the same as that of $n$-groupoids, except we have to restrict the class of admissible maps.
\begin{defi}
	Given an $n$-globular theory $(\mathfrak{C},F)$, we say that a map $f$ in $\mathfrak{C}$ is globular if it is in the image of $\Tn$ under $F$.
	
	On the other hand, $f$ is called homogeneous if for every factorization $f=g\circ f'$ where $g$ is a globular map, $g$ must be the identity.
	
	$\mathfrak{C}$ is said to be homogeneous if it comes endowed with a globular sum preserving functor $H\colon \mathfrak{C} \rightarrow \Theta^{\leq n}$ that detects homogeneous maps, in the sense that a map $f$ in $\mathfrak{C}$ is homogeneous if and only if $H(f)$ is such, where $\Theta$ is the globular theory for strict $\infty$-categories, as defined in \cite{AR1}, and $\Theta^{\leq n}$ is its subcategory spanned by all globular sums of dimension less or equal to $n$. If this is the case, then given an homogeneous map $\rho\colon D_m\rightarrow A$ we have $m\geq \dim(A)$, and every map $f$ admits a unique factorization as a homogeneous map followed by a globular one.
\end{defi}
\begin{rmk}
\label{general homogeneous map}
A map $f\colon A \rightarrow B$ in a homogeneous globular theory $\mathfrak{C}$ is homogeneous if and only if, for every $D_{i_k}$ appearing in the globular decomposition of $A$, the homogeneous-globular factorizations of $D_{i_k}\rightarrow A \rightarrow B$ given by $D_{i_k} \rightarrow B_k \rightarrow B$ induce an isomorphism of the form: 
\[\colim_k B_k \cong B\]
\end{rmk}
\begin{defi}
	Let $(\mathfrak{C},F)$ be an $n$-globular theory. A pair of maps $(f,g)$ with $f,g\colon D_k \rightarrow A$ is said to be admissible for a theory of $n$-categories (or just admissible, in case there is no risk of confusion with the groupoidal case) if either $k=0$, or both of them are homogeneous maps or else if there exists homogeneous maps $f',g'\colon D_k \rightarrow \partial A$ such that the following diagrams commute
	\[
	\bfig 
	\morphism(0,0)|a|/@{>}@<0pt>/<400,0>[D_k`A;f]

	\morphism(0,0)|l|/@{>}@<0pt>/<0,-400>[D_k`\partial A;f']

	\morphism(0,-400)|r|/@{>}@<0pt>/<400,400>[\partial A`A;\partial_{\sigma}]

	\morphism(1000,0)|a|/@{>}@<0pt>/<400,0>[D_k`A;g]

	\morphism(1000,0)|l|/@{>}@<0pt>/<0,-400>[D_k`\partial A;g']

	\morphism(1000,-400)|r|/@{>}@<0pt>/<400,400>[\partial A`A;\partial_{\tau}]
	
	\efig 
	\]
\end{defi}
The definition of a coherator for $n$-categories is totally analogous to that for $n$-groupoids, i.e. it is a contractible and cellular globular theory, except the pair of maps that we consider in both cases have to be the admissible ones in the sense of the previous definition.

More precisely, the pairs appearing in Definition \ref{contr glob th} and in point 2 of Definition \ref{cell glob th} must be pairs of admissible maps.
\begin{defi}
	A (Grothendieck) $n$-category is a model of a coherator for $n$-categories.
\end{defi} 
Unless specified otherwise, $n$-category and $n$-groupoid will always mean weak ones, i.e. Grothendieck $n$-categories and Grothendieck $n$-groupoids. 
\subsection{Direct categories and cofibrations}
\begin{defi}(see also \cite{Ho}, Chapter 5)
	\label{direct category}
	A direct category is a pair $(\C,d)$, where $\C$ is a small category and $d:\ob(\C) \rightarrow \lambda$ is a function into an ordinal $\lambda$ , such that if there is a non-identity morphism $f:a \rightarrow b$ in $\C$, then $d(a)< d(b)$.
	
	Given a cocomplete category $\D$ and a functor $X\colon \C \rightarrow \D$, we define the latching object of $X$ at an object $c\in \C$ to be the object of $\D$ given by
	\[L_c(X)=\colim_{c' \in \C_{< d(c)}\downarrow c}X(c')\]
	This defines a functor $L_c$ from the functor category $[\C,\D]$ to the category $\D$, together with a natural transformation $\epsilon_c\colon L_c \Rightarrow ev_c$, with codomain the functor given by evaluation at $c$.
	We also define the latching map of a natural transformation $\alpha \colon  X\rightarrow Y$ in $\D^{\C}$ at an object $c\in \C$ to be the map of $\D$ 
	\[\hat{L}_c(\alpha)\colon X(c) \plus{L_c(X)} L_c(Y) \rightarrow Y(c)\]
	induced by $L_c(f)$ and $\epsilon_c$.
\end{defi}
The following results on direct categories are well known, therefore we omit their proofs. The notion of weak orthogonality is denoted with $\pitchfork$
\begin{lemma}
	\label{Reedy construction}
	Let $\D$ be a direct category and $\C$ a category equipped with two classes of arrows $(\mathscr{L},\R)$ such that $\mathscr{L}\pitchfork \R$.
	If we define

	\[	 \mathscr{L}^{\D}= \{ \alpha\colon X \rightarrow Y \ \text{in} \ \C^{\D} \ | \ \hat{L}_d(\alpha)\in \mathscr{L} \ \forall d \in \D \} \]
	and \[\R^{\D}=\{\alpha\colon X \rightarrow Y \ \text{in} \ \C^{\D} \ | \ \alpha_d\colon X(d) \rightarrow Y(d) \in \R  \ \forall d \in \D \}\] we have $\mathscr{L}^{\D} \pitchfork \R^{\D}$.
%
%
\end{lemma}
\begin{lemma}
	\label{preservation of cofs^D}
	Let $A,B$ be two cocomplete categories equipped, respectively, with two classes of arrows $(\mathscr{L}_A,\mathcal{R}_A)$ and $(\mathscr{L}_B,\mathcal{R}_B)$ such that $\mathscr{L}_A \pitchfork \mathcal{R}_A$ and $\mathscr{L}_B \pitchfork\mathcal{R}_B$.
	Given a cocontinuous functor $F\colon A \rightarrow B$ such that $F(\mathscr{L}_A)\subset \mathscr{L}_B$ and a direct category $\D$, the induced map $F^{\D}\colon A^{\D} \rightarrow B^{\D}$ preserves the direct cofibrations, i.e.\[F(\mathscr{L}^{\D}_A)\subset \mathscr{L}^{\D}_B\]
\end{lemma}
\begin{ex}
	\label{G is a direct cat}
	The category $\mathbb{G}_n$ has a natural structure of direct category, with degree function defined by
	\[\begin{matrix}
	\deg \colon \G_n \rightarrow \mathbb{N} \\
	\qquad m \mapsto m 
	\end{matrix}\]
	Every time we have an $n$-coglobular object $\mathbf{D}_{\bullet}\colon \G_n \rightarrow \C$ in a finitely cocomplete category, we can consider the latching map of $!\colon \emptyset \rightarrow \mathbf{D}_{\bullet}$ at $m$, i.e. the map \[\hat{L}_m(!)\colon L_m(\mathbf{D}_{\bullet}) \rightarrow \mathbf{D}_m\]
	Notice that \[\hat{L}_1(!)=(\mathbf{D}(\sigma_0),\mathbf{D}(\tau_0))\colon \mathbf{D}_0 \coprod \mathbf{D}_0\rightarrow \mathbf{D}_1\] and the other latching maps are obtained inductively from the following cocartesian square
	\[
	\bfig 
	
	\morphism(0,0)|a|/@{>}@<0pt>/<500,0>[L_m(\mathbf{D}_{\bullet})` \mathbf{D}_m;\hat{L}_m(!)]
	\morphism(0,0)|a|/@{>}@<0pt>/<0,-500>[L_m(\mathbf{D}_{\bullet})` \mathbf{D}_m;\hat{L}_m(!)]
	\morphism(500,0)|a|/@{>}@<0pt>/<0,-500>[ \mathbf{D}_m`L_{m+1}(\mathbf{D}_{\bullet});]
	\morphism(0,-500)|a|/@{>}@<0pt>/<500,0>[\mathbf{D}_m`L_{m+1}(\mathbf{D}_{\bullet}); ]
	
	\morphism(500,-500)|l|/@{-->}@<0pt>/<300,-300>[L_{m+1}(\mathbf{D}_{\bullet})` \mathbf{D}_{m+1};\exists ! \hat{L}_{m+1}(!)]
	
	\morphism(500,0)|a|/{@{>}@/^2em/}/<300,-800>[ \mathbf{D}_{m}` \mathbf{D}_{m+1}; \mathbf{D}(\sigma_m)]
	\morphism(0,-500)|l|/{@{>}@/_2em/}/<800,-300>[\mathbf{D}_{m}` \mathbf{D}_{m+1};\mathbf{D}(\tau_m) ]
	\efig 
	\]
	
\end{ex}
When $\mathbf{D}_{\bullet}\colon\G_n \rightarrow \mathfrak{G}\rightarrow \mathbf{Mod}(\mathfrak{G})$ is the canonical coglobular $n$-groupoid, we will also denote $L_{m}(\mathbf{D}_{\bullet})$ by $S^{m-1}$, borrowing  this notation from topology.
\begin{defi}
	\label{def cofi}
	Let $I_n$ (resp. $J_n$) be the set $\{S^{k-1} \rightarrow D_k\}_{0\leq k \leq n}\cup\{(1,1)\colon S^n\rightarrow D_n\}$ of boundary inclusions in $\mathbf{Mod}(\mathfrak{G})$, together with the map collapsing a pair of parallel $n$-cells to a single $n$-cells (resp. source maps $\{\sigma_k\colon D_{k} \rightarrow D_{k+1}\}_{0 \leq k \leq n-1}$), and $\mathbb{I}_n$ its saturation, i.e. the set $^{\pitchfork}(I_n^{\pitchfork})$ (resp. $\mathbb{J}_n= {}^{\pitchfork}(J_n^{\pitchfork})$). 
	
	We say that a map of $n$-groupoids $f\colon X \rightarrow Y$ is a cofibration (resp. trivial cofibration) of $n$-groupoids if it belongs to $\mathbb{I}_n$ (resp. $\mathbb{J}_n$).
	
	The maps in the class $J_n^{\pitchfork}$ (resp. $I_n^{\pitchfork}$) are called fibrations (resp. trivial fibrations).
\end{defi} 
The small object argument provides a factorization system on $n$-groupoids given by cofibrations and trivial fibrations. Lemma \ref{Reedy construction} will be applied to this factorization system and to the the direct category structure on $\G_n$ as defined in Example \ref{G is a direct cat}, to provide a way of inductively extending certain maps in $\mathbf{Mod}(\mathfrak{G})^{\G}$.

Let $*$ denote the terminal object in the category of $n$-groupoids. Since every map in $J$ admits a retraction, the following result is straightforward.
\begin{prop}
	\label{groupoids are contractible}
	Every $n$-groupoid is fibrant, i.e. the unique map $X \rightarrow *$ is a fibration for every $X \in \mathbf{Mod}(\mathfrak{G})$.
\end{prop}
\begin{defi}
	An $n$-groupoid $X$ is said to be contractible if the unique map $X \rightarrow *$ is a trivial fibration.
\end{defi}
The proof of the following fact is analogous to the one given for $\infty$-groupoids in Proposition 3.8 of \cite{EL}.
\begin{prop}
	\label{glob sums are contractible}
	Globular sums, seen as objects in the image of the Yoneda embedding functor $y\colon \mathfrak{G} \rightarrow \mathbf{Mod}(\mathfrak{G})$, are contractible $n$-groupoids.
%
%
\end{prop}
\section{Recognition principle for semi-model structures on categories of models of globular theories}
In this section we are going to characterize those globular theories $\mathfrak{C}$ for which the category of models $\mathbf{Mod}(\mathfrak{C})$ bears a cofibrantly generated semi-model structure that satisfies some natural conditions for objects of $\mathbf{Mod}(\mathfrak{C})$ to look like $\infty$-groupoids. The definition of a (cofibrantly generated) semi-model structure can be found in \cite{FR}, Section 12.1. It is clear that everything can be adapted, with the appropriate changes, to the case of $n$-globular theories for $n<\infty$.

To begin with, we define a class of maps $\mathbf{W}$ in $\mathbf{Mod}(\mathfrak{C})$ that consists of the maps $f\colon X\rightarrow Y$ such that every solid commutative square of the form:
\begin{equation}
\label{def of weak eq}
\bfig 
\morphism(0,0)|a|/@{>}@<0pt>/<500,0>[S^{k-1}` X;(\alpha,\beta)]
\morphism(0,0)|a|/@{>}@<0pt>/<0,-400>[S^{k-1}`D_k;j_k]
\morphism(0,-400)|a|/@{>}@<0pt>/<500,0>[D_k` Y ;\gamma]
\morphism(500,0)|r|/@{>}@<0pt>/<0,-400>[X` Y;f]

\morphism(0,-400)|a|/@{.>}@<0pt>/<500,400>[D_k` X;\Gamma]
\efig 
\end{equation}
admits an extension for the upper triangle which is a ``lift up to homotopy'' for the lower triangle. More precisely, there is a map $\Gamma$ such that $\Gamma \circ j_k = (\alpha,\beta)$ and there exists a $(k+1)$-cell $H$ in $Y$ with $H\colon f\circ \Gamma \rightarrow \gamma$.
\begin{thm}
	\label{semi model str characterization}
	Given a globular theory $\mathfrak{C}$, there exists a cofibrantly generated semi-model structure on the category of models $\mathbf{Mod}(\mathfrak{C})$ with weak equivalences given by the class $\mathbf{W}$, where every object is fibrant, globular sums are contractible and the set of generating cofibrations (resp. trivial cofibrations) consists of the boundary inclusions $\mathbf{I}\overset{def}{=}\{j_k \colon S^{k-1}\rightarrow D_k\}_{k\geq 0}$ (resp. source maps $\mathbf{J}\overset{def}{=}\{\sigma_k \colon D_k\rightarrow D_{k+1}\}_{k\geq 0}$) if and only if:
	\begin{itemize}
		\item each map $\sigma_k \colon D_k\rightarrow D_{k+1}$ admits a retraction;
		\item $D_0$ is contractible (i.e. the unique map $D_0 \rightarrow \ast$ is a trivial fibration);
		\item $\mathfrak{C}$ admits a system of composition and identities, as defined in Definition \ref{structure systems};
		\item for every cofibrant object $X$ in $\mathbf{Mod}(\mathfrak{C})$ there exists a fibration $\mathbf{ev}\colon \p X \rightarrow X \times X$ such that $\mathbf{ev}_i=\pi_i \circ \mathbf{ev}$ is a trivial fibration for $i=0,1$, where $\pi_i \colon X \times X \rightarrow X$ denote the product projections.
	\end{itemize}
\end{thm}
If such a semi-model structure exists on $\mathbf{Mod}(\mathfrak{C})$, then clearly the four conditions are satisfied. Let us check that the converse also holds true.

The proof is a matter of checking that the recognition principle for cofibrantly generated model structures applies (mutatis mutandis, since we want to produce a semi-model structure) to this situation. The first condition clearly implies that all objects are fibrant. Moreover, $\mathbf{Mod}(\mathfrak{C})$ is complete and cocomplete, and both the domains of $\mathbf{I}$ and $\mathbf{J}$ permit the small object argument.
\begin{lemma}
	\label{W closed under retracts}
	$\mathbf{W}$ is closed under retracts.
	\begin{proof}
		Assume $f$ is a retract of $g\in \mathbf{W}$, so that we have a commutative diagram:
		\[
		\bfig 
		\morphism(0,0)|a|/@{>}@<0pt>/<400,0>[X` W;a]
		\morphism(0,0)|a|/@{>}@<0pt>/<0,-400>[X`Y;f]
		\morphism(0,-400)|a|/@{>}@<0pt>/<400,0>[Y` Z;c]
		\morphism(400,0)|r|/@{>}@<0pt>/<0,-400>[W`Z;g]
		\morphism(400,0)|a|/@{>}@<0pt>/<400,0>[W`X;b]
		\morphism(800,0)|r|/@{>}@<0pt>/<0,-400>[X`Y;f]
		\morphism(400,-400)|a|/@{>}@<0pt>/<400,0>[Z`Y;d]
		\efig 
		\] with $b\circ a=1_X$ and $ d \circ c =1_Y$. Given a $(k-1)$-sphere $(\alpha,\beta)$ in $X$ and a $k$-cell $H\colon f(\alpha) \rightarrow f(\beta)$ in $Y$, we get a $k$-cell $\phi\colon a(\alpha)\rightarrow a(\beta)$ together with a $(k+1)$-cell $\Gamma\colon g (\phi )\rightarrow c (H)$ in $Z$. If we consider the $k$-cell $\overline{H}\overset{def}{=}b(\phi)$ in $X$, we see that $\overline{H}\colon \alpha\rightarrow \beta$ and $d(\Gamma)=\colon d(g (\phi))=f(\overline{H})\rightarrow d(c(H))=H$.
	\end{proof}
\end{lemma}
\begin{lemma}
	\label{closure prop of W}
	Let $f\colon X \rightarrow Y,g\colon Y \rightarrow Z$ be maps in $\mathbf{Mod}(\mathfrak{C})$. Then:
	\begin{enumerate}
		\item If $f$ and $g$ belong to $\mathbf{W}$ then so does $g\circ f$;
		\item  If $g$ and $g\circ f$ belong to $\mathbf{W}$, then so does $f$;
		\item If $g\circ f=1_X$ and $f\circ g$ belongs to $\mathbf{W}$, then both $f$ and $g$ belong to $\mathbf{W}$.
	\end{enumerate}
	\begin{proof}
		Firstly, assume $f$ and $g$ belong to $\mathbf{W}$, and assume given a $(k-1)$-sphere $(a,b)$ in $X$, together with a $k$-cell $\gamma \colon g\circ f (a) \rightarrow g\circ f(b)$ in $Z$. By assumption we get a $k$-cell $\beta \colon f(a)\rightarrow f(b)$ in $Y$ and a $(k+1)$-cell $H\colon g(\beta) \rightarrow \gamma$. Again by assumption we get a $k$-cell $\alpha\colon a \rightarrow b$ in $X$, together with a $(k+1)$-cell $H'\colon f(\alpha) \rightarrow \beta$. The composite $H\circ g(H')\colon g\circ f(a)\rightarrow \gamma$ (obtained using the system of composition on $\mathfrak{C}$) is the data we need to conclude the proof of the first statement.
		
		Turning to the second statement, assume $g$ and $g\circ f$ belong to $\mathbf{W}$ and consider a $(k-1)$-sphere $(a,b)$ in $X$, together with a $k$-cell $\alpha\colon f(a)\rightarrow f(b)$ in $Y$. We can lift the $(k-1)$-sphere $(g\circ f(a),g\circ f(b))$ in $Z$ along $g\circ f$ to get a $k$-cell in $X$ of the form $H\colon a \rightarrow b$, together with a $(k+1)$-cell $\Gamma \colon g\circ f(H) \rightarrow g(\alpha)$. We now have a $k$-sphere in $Y$ given by $(f(H),\alpha)$, and an extension to a $(k+1)$-cell in $Z$ between its image under $g$. By assumption, we get a lift to a $(k+1)$-cell $\overline{H}\colon f(H)\rightarrow \alpha$, which concludes the proof of the second statement.
		
		Finally, if $g\circ f=1_X$ then $g$ is a retract of $f\circ g$, and is thus a weak equivalence thanks to Lemma \ref{W closed under retracts}. Therefore, $f\in \mathbf{W}$ thanks to the second point of this lemma, since identities are weak equivalences thanks to Lemma \ref{tr fib are we}.
	\end{proof}
\end{lemma}
Since relative $\mathbf{J}$-cell complexes relative to $D_0$ include all globular sums, if we prove that such maps are weak equivalences we then obtain for free the contractibility of globular sums, since $D_0$ is contractible by hypothesis. We actually prove a little bit more, namely the following result.
\begin{lemma}
	Let $f\in cof(\mathbf{J})$ have a cofibrant domain. Then $f\in \mathbf{W}$.
	\begin{proof}
		Let $f\colon X \rightarrow Y$ be as in the statement. Pick a section $i_Y$ of the trivial fibration $\mathbf{ev}_0\colon \p Y \rightarrow Y$, which exists since $Y$ is cofibrant, and denote by $\alpha$ the endomorphism $\mathbf{ev}_1\circ i_Y$, which is a weak equivalence thanks to Lemma \ref{closure prop of W} and Lemma \ref{tr fib are we}. Consider the following commutative square:
		\[
		\bfig 
		\morphism(0,0)|a|/@{>}@<0pt>/<600,0>[X` \p Y;\alpha\circ f]
		\morphism(0,0)|a|/@{>}@<0pt>/<0,-500>[X`Y;f]
		\morphism(0,-500)|b|/@{>}@<0pt>/<600,0>[Y` Y\times Y ;(1,\alpha\circ f \circ r)]
		\morphism(600,0)|r|/@{>}@<0pt>/<0,-500>[\p Y` Y\times Y;\mathbf{ev}]
		
		\morphism(0,-500)|a|/@{.>}@<0pt>/<600,500>[Y` \p Y;\Gamma]
		\efig 
		\] where $r$ denotes the choice of a retraction of $f$, which exists since $X$ is fibrant, and the lift $\Gamma$ exists by assumption, since $f\in cof(\mathbf{J})$, $X$ is cofibrant and $\mathbf{ev}$ is a fibration. We have $\mathbf{ev}_0 \circ \Gamma=1_Y$ which implies that $\Gamma$ is a weak equivalence, thanks to Lemma \ref{closure prop of W}. Therefore, thanks to the same lemma and Lemma \ref{tr fib are we}, we see that $\alpha\circ f \circ r=\mathbf{ev}_1 \circ \Gamma$ is also a weak equivalence. A final application of the previous lemma yields that $f\circ r$ belongs to $\mathbf{W}$, which in turn implies that $f$ is a weak equivalence thanks to Lemma \ref{closure prop of W} again, since $r\circ f =1_X$.
	\end{proof}
\end{lemma}
\begin{lemma}
	\label{tr cof are cof}
	$cof(\mathbf{J})\subset cof (\mathbf{I})$.
	\begin{proof}
		Of course it is enough to check that $\mathbf{J}\subset cof (\mathbf{I})$. We thus have to prove that, for every $k\geq 0$, we have that $\sigma_k \colon D_k \rightarrow D_{k+1}$ belongs to $cof (\mathbf{I})$. We know by assumption that $S^{k-1} \rightarrow D_k$ is a cofibration, so that the colimit injection $i_0\colon D_k \rightarrow S^k\overset{def}{=}D_k\plus{S^{k-1}}D_k$ is also such, being a pushout of it. We can now compose that with the boundary inclusion $S^k \rightarrow D_{k+1}$ to conclude the proof. 
	\end{proof}
\end{lemma}
\begin{lemma}
	\label{tr fib are we}
	$inj(\mathbf{I}) = inj(\mathbf{J}) \cap \mathbf{W}$.
	\begin{proof}
		We start by proving $inj(\mathbf{I}) \subset inj(\mathbf{J}) \cap \mathbf{W}$. Thanks to Lemma \ref{tr cof are cof} we only have to prove that $inj(\mathbf{I}) \subset  \mathbf{W}$, which is obvious, since a cell $f\rightarrow f$ exists for every cell in $Y$ thanks to the system of identities in $\mathfrak{C}$. Conversely, assume $f$ is both a fibration and a weak equivalence, and consider a $(k-1)$-sphere $(a,b)$ in $X$ together with a $k$-cell $H\colon f(a) \rightarrow f(b)$ in $Y$. Since $f$ belongs to $\mathbf{W}$, we find a $k$-cell $\overline{H}\colon a \rightarrow b$ in $X$, together with a $(k+1)$-cell $\Gamma\colon f(\overline{H})\rightarrow H$ in $Y$. Because $f$ is a fibration, we can lift $\Gamma$ to a cell $\gamma\colon \overline{H} \rightarrow \beta$, so that $f(\beta) = H$ and $\beta\colon a \rightarrow b$, since it is parallel to $\overline{H}$, and this concludes the proof.
	\end{proof}
\end{lemma}
Since we have proven that globular sums are contractible in $\mathbf{Mod}(\mathfrak{C})$, we can endow models of $\mathfrak{C}$ with the structure of $\infty$-groupoids, and use the results in Section 4 of \cite{AR2} to obtain the missing piece: namely, the 2-out-of-3 property for $\mathbf{W}$. Indeed, the maps in $\mathbf{W}$ can be characterized as in Theorem 4.18 (ibid.) and we can use the invariance of basepoints (i.e. Corollary 4.14) to conclude that if $f$ and $g\circ f$ are weak equivalences then $g$ is also such. More precisely, let $y$ be a 0-cell of $Y$, we want to prove that $\pi_n(g)\colon \pi_n (Y,y)\rightarrow \pi_n (Z,g(y))$ is an isomorphism. Choose a 1-cell $f(x)\rightarrow y$, whose existence is ensured by the fact that $\pi_0(f)$ is bijective, so that we have an isomorphism $\pi_n (Y,y)\cong \pi_n (Y,f(x))$ as well as $\pi_n (Z,g(y))\cong \pi_n (Y,g(f(x)))$. Consider the following commutative diagram:
\[
\bfig
\morphism(0,0)|a|/@{>}@<0pt>/<800,0>[\pi_n (X,x)` \pi_n \left( Y,f(x) \right);\pi_n(f)]
\morphism(800,0)|a|/@{>}@<0pt>/<900,0>[\pi_n \left( Y,f(x) \right)` \pi_n \left( Z,g (f(x)) \right);\pi_n(g)]
\morphism(800,0)|a|/@{>}@<0pt>/<0,-400>[\pi_n \left( Y,f(x) \right)` \pi_n \left( Y, y \right);\cong]
\morphism(800,-400)|a|/@{>}@<0pt>/<900,0>[\pi_n \left( Y,y \right)` \pi_n \left( Z,g (y) \right);\pi_n(g)]
\morphism(1700,0)|r|/@{>}@<0pt>/<0,-400>[\pi_n \left( Z,g (f(x)) \right)` \pi_n \left( Z,g (y) \right);\cong]
\efig 
\]
By assumption, the upper horizontal arrow of the square is bijective, which implies that the bottom one is also such and this concludes the proof of Theorem \ref{semi model str characterization}.
\section{The semi-model structure}
\label{the semi model str section}
The part of Theorem \ref{semi model str characterization} that is hard to check in practice is that of the path object (trivial) fibration, i.e. the functorial construction of a fibration $\mathbf{ev}\colon\p X \rightarrow X \times X$ such that the composition with both projections is a trivial fibration. As we prove in Theorem \ref{semi model str on C_W}, it is enough to construct such map for a globular theory obtained from a coherator for $n$-categories (with $0\leq n \leq \infty$) by freely adjoining a left and a right inverse for each map. This appears to be quite easier than building one for a coherator for $n$-groupoids, since we can use the homogeneity property, and in the last section we will define such path object for the case $n=3$. 

The following definition is slightly different from the one given in \cite{EL}, namely we consider left and right inverses instead of two-sided inverses, and this is done in order to produce the correct homotopy type of globular sums in the corresponding category of models. More precisely, we are going to prove that the maps $\alpha_k$ defined in \ref{alpha_1} are trivial cofibrations of $\mathfrak{C}^{\mathbf{W}}$-models, which is false if one considers a theory with a two-sided inverse operation instead.
\begin{defi}
	\label{structure systems}
	A system of compositions in an $n$-globular theory $\mathfrak{C}$ consists of a family of maps $\{\mathbf{c}_k\colon D_k \rightarrow D_k \amalg_{D_{k-1}} D_k\}_{1\leq k \leq n}$ such that $\mathbf{c}_k \circ \sigma=i_1 \circ \sigma$ and $\mathbf{c}_k \circ \tau=i_2 \circ \tau$, where $i_1$ (resp. $i_2$) denotes the colimit inclusion onto the first (resp. second) factor.
	
	A system of identities (with respect to a chosen system of compositions) consists of a family of maps $\{\mathbf{id}_k\colon D_{k+1} \rightarrow D_k \}_{0\leq k \leq n-1}\cup \{\mathbf{l}_k,\mathbf{r}_k\colon D_k \rightarrow D_{k-1}\}_{2\leq k \leq n+1}$ such that $\mathbf{id}_k \circ \epsilon= 1_{D_k},$ for every $k\geq 0$ and $\epsilon=\sigma, \tau$, $\mathbf{l}_k \circ \sigma=1_{D_{k-1}}, \ \mathbf{l}_k\circ \tau = (1_{D_{k-1}},\tau \circ \mathbf{id}_{k-2} )\circ \mathbf{c}_{k-1}$ and $\mathbf{r}_k \circ \sigma=1_{D_{k-1}}, \ \mathbf{r}_k\circ \tau = (\sigma \circ \mathbf{id}_{k-2},1_{D_{k-1}})\circ \mathbf{c}_{k-1}$.
	
	A system of (left and right) inverses (with respect to chosen systems of compositions and identities) consists of a family of maps $\{\mathbf{i}^l_k,\mathbf{i}^r_k\colon D_k \rightarrow D_k\}_{1\leq k \leq n} \cup \{\mathbf{k}_k^l,\mathbf{k}_k^r\colon D_k \rightarrow D_{k-1} \}_{2\leq k \leq n+1}$ such that $\mathbf{i}^{\epsilon}_k\circ \sigma= \tau, \  \mathbf{i}_k\circ \tau = \sigma$ for $\epsilon=l,r$, $\mathbf{k}_k^l\circ \sigma =\sigma \circ \mathbf{id}_{k-2}, \ \mathbf{k}_k^l \circ \tau =(\mathbf{i}^l_{k-1},1_{D_{k-1}})\circ\mathbf{c}_{k-1}, \ \mathbf{k}_k^r\circ \sigma =\tau \circ \mathbf{id}_{k-2}$ and $\mathbf{k}_k^r\circ \tau =(\mathbf{i}^r_{k-1},1_{D_{k-1}}) \circ\mathbf{c}_{k-1} $.
	
	If $\mathfrak{C}$ admits a choice of such three systems, given a globular functor $\mathbf{F}\colon \mathfrak{C} \rightarrow \mathbf{Mod}(\mathfrak{G})$ we say that for every $\mathfrak{G}$-model $X$, the $\mathfrak{C}$-model $\mathfrak{G}(\mathbf{F},X)$ can be endowed with such systems.
\end{defi} 
\begin{rmk}
	\label{two-sided vs left and right}
In the presence of both left and right inverses for every cell, any of the two can be promoted to a two-sided one. For instance, assume $f$ is an $m$-cell with both a left inverse $k$ and a right inverse $g$, and let us show that $k$ is also a right inverse for $f$, the remaining case being similar. It is enough to provide a cell from $k$ to $g$, as follows (omitting bracketings for simplicity): \[\bfig \morphism(0,0)|a|/@{>}@<0pt>/<500,0>[k` kfg;k\mathbf{k}^r_m(f)] \morphism(500,0)|a|/@{>}@<0pt>/<700,0>[kfg` g;\mathbf{i}^r_{m+1}(\mathbf{k}^l_m(f))g]\efig\]
\end{rmk}
\begin{defi}
Given an $n$-coherator for categories $\mathfrak{C}$, we define a new globular theory $\mathfrak{C}^{\mathbf{W}}$ by means of the following pushout of globular theories:
\begin{equation}
\label{D_W def}
\bfig 
\morphism(0,0)|a|/@{>}@<0pt>/<-500,-500>[\Theta_0 [\mathbf{comp},\mathbf{id}]` \Theta_0 [\mathbf{comp},\mathbf{id}, \mathbf{inv}];]
\morphism(0,0)|a|/@{>}@<0pt>/<500,-500>[\Theta_0 [\mathbf{comp},\mathbf{id}]`\mathfrak{C};\mathbf{i}]
\morphism(500,-500)|a|/@{>}@<0pt>/<-500,-500>[\mathfrak{C}` \mathfrak{C}^{\mathbf{W}};]
\morphism(-500,-500)|r|/@{>}@<0pt>/<500,-500>[\Theta_0 [\mathbf{comp},\mathbf{id}, \mathbf{inv}]`\mathfrak{C}^{\mathbf{W}};]
\efig 
\end{equation}
\end{defi}
Here, we denote with $\Theta_0 [\mathbf{comp},\mathbf{id}]$ the free globular theory on a system of composition and identities, and with $\Theta_0 [\mathbf{comp},\mathbf{id}, \mathbf{inv}]$ the free globular theory on a system of composition, identities and inverses. There is a canonical map as depicted in the upper left of the square and the map denoted by $\mathbf{i}$ is defined as follows: first, we choose a binary composition of 1-cells operation (say, the $w$ defined in \eqref{w's maps}), and we set $\mathbf{i}(\mathbf{c}_k)=\Sigma^{k-1}(w)$. The action on identity operations is defined similarly.
We are going to prove that, given a coherator for $n$-categories $\mathfrak{C}$ (with $0\leq n \leq \infty$), one has that $\mathfrak{C}^{\mathbf{W}}$ is a coherator for $n$-groupoids and there is a semi-model structure on the category of $\mathfrak{C}^{\mathbf{W}}$-models $\mathbf{Mod}(\mathfrak{C}^{\mathbf{W}})$ (with (trivial) cofibrations and weak equivalences as in Theorem \ref{semi model str characterization}) provided there is a functor $\p \colon \mathbf{Mod}(\mathfrak{C}^{\mathbf{W}}) \rightarrow \mathbf{Mod}(\mathfrak{C}^{\mathbf{W}})$  together with a natural transformation $\mathbf{ev}\colon \p \Rightarrow \mathbf{Id}\times \mathbf{Id}$ which is a pointwise fibration with the property that $\mathbf{ev}_i\overset{def}{=}\pi_i \circ \mathbf{ev}$ is a pointwise trivial fibration.

In the case $n=3$, we are going to construct a globular functor $\cyl\colon \mathfrak{C}^{\mathbf{W}} \rightarrow \mathbf{Mod}(\mathfrak{C}^{\mathbf{W}})$ in Section \ref{S}, and by setting $\p X = \mathbf{Mod}(\mathfrak{C}^{\mathbf{W}})\left(\cyl(\bullet),X\right))$ we will obtain the endofunctor in the hypotheses of Theorem \ref{semi model str on C_W}. This, in turn, will produce Theorem \ref{model structure} as a corollary. We start with the general result.
\begin{thm}
	\label{semi model str on C_W}
	Let $\mathfrak{C}$ be a coherator for $n$-categories (with $0\leq n \leq \infty$), and suppose there is a functor $\p \colon \mathbf{Mod}(\mathfrak{C}^{\mathbf{W}}) \rightarrow \mathbf{Mod}(\mathfrak{C}^{\mathbf{W}})$ endowed with a natural transformation $\mathbf{ev}\colon \p \Rightarrow \mathbf{Id}\times \mathbf{Id}$ which is a pointwise fibration with the property that $\mathbf{ev}_i\overset{def}{=}\pi_i \circ \mathbf{ev}$ is a pointwise trivial fibration. Then $\mathfrak{C}^{\mathbf{W}}$ satisfies the hypotheses of Theorem \ref{semi model str characterization}, and therefore is a coherator for $n$-groupoids. Moreover, $\mathbf{Mod}(\mathfrak{C}^{\mathbf{W}})$ admits a semi-model structure as described in \ref{semi model str characterization}.
\end{thm}
\begin{proof}
	We denote by ${D_0}^{\mathfrak{C}^{\mathbf{W}}}$ the representable $\mathfrak{C}^{\mathbf{W}}$-model on $D_0$, and we adopt a similar convention for ${D_0}^{\mathfrak{C}}$.
	All the hypotheses of the theorem are trivially satisfied, except for the contractibility of ${D_0}^{\mathfrak{C}^{\mathbf{W}}}$. We know from Lemma \ref{vanishing of higher homotopy groups in globular sums} that ${D_0}^{\mathfrak{C}}$ is contractible, so it can be endowed with the structure of $\mathfrak{C}^{\mathbf{W}}$-model, that we still denote by ${D_0}^{\mathfrak{C}}$. The claim would then follow if we can prove that the counit of the adjunction \[
	\bfig
	\morphism(0,0)|a|/{@{>}@/^1em/}/<800,0>[\mathbf{Mod}(\mathfrak{C})`\mathbf{Mod}(\mathfrak{C}^{\mathbf{W}}) ;\mathbf{F}]
	\morphism(800,0)|b|/{@{>}@/^1em/}/<-800,0>[\mathbf{Mod}(\mathfrak{C}^{\mathbf{W}})`\mathbf{Mod}(\mathfrak{C}) ;\mathbf{U}]
	\morphism(350,-50)|a|/{@{}}/<100,0>[ `  ;\perp]
	\efig 
	\] is a weak equivalence at ${D_0}^{\mathfrak{C}}$, since $\mathbf{F}\mathbf{U} {D_0}^{\mathfrak{C}}={D_0}^{\mathfrak{C}^{\mathbf{W}}}$. This is a consequence of a more general result, proven in Proposition \ref{counit is a we}.
\end{proof}
\begin{prop}
	\label{counit is a we}
	Let $X$ be a $\mathfrak{C}^{\mathbf{W}}$-model such that $\mathbf{F}\mathbf{U} X$ is cofibrant. Then the counit $\epsilon$ of the adjunction 
	\[
	\bfig
	\morphism(0,0)|a|/{@{>}@/^1em/}/<800,0>[\mathbf{Mod}(\mathfrak{C})`\mathbf{Mod}(\mathfrak{C}^{\mathbf{W}}) ;\mathbf{F}]
	\morphism(800,0)|b|/{@{>}@/^1em/}/<-800,0>[\mathbf{Mod}(\mathfrak{C}^{\mathbf{W}})`\mathbf{Mod}(\mathfrak{C}) ;\mathbf{U}]
	\morphism(350,-50)|a|/{@{}}/<100,0>[ `  ;\perp]
	\efig 
	\] is a weak equivalence at $X$.
	\begin{proof}
		It is enough to show that $\mathbf{U}(\epsilon_X)$ is a weak equivalence of $\mathfrak{C}$-models. Let us consider the following commutative square in $\mathbf{Mod}(\mathfrak{C})$:
		\begin{equation}
		\label{homotopy for epsilon}
		\bfig 
		\morphism(0,0)|a|/@{>}@<0pt>/<1500,0>[\mathbf{U}X`\mathbf{U} \p  \mathbf{F}\mathbf{U}X ;\mathbf{U}(i)\circ \eta_{\mathbf{U}X}]
		\morphism(0,0)|a|/@{>}@<0pt>/<0,-500>[\mathbf{U}X`\mathbf{U} \mathbf{F}\mathbf{U}X;\eta_{\mathbf{U}X}]
		\morphism(0,-500)|a|/@{>}@<0pt>/<1500,0>[\mathbf{U} \mathbf{F}\mathbf{U}X` \mathbf{U} \mathbf{F}\mathbf{U}X \times  \mathbf{U} \mathbf{F}\mathbf{U}X;\left(\eta_{\mathbf{U}X} \circ \mathbf{U}(\epsilon_X),\mathbf{U}(\mathbf{ev}_1\circ i)\right)]
		\morphism(1500,0)|r|/@{>>}@<0pt>/<0,-500>[ \mathbf{U} \p  \mathbf{F}\mathbf{U}X`  \mathbf{U} \mathbf{F}\mathbf{U}X\times  \mathbf{U} \mathbf{F}\mathbf{U}X;\mathbf{U}(\mathbf{ev})]
		\efig 
		\end{equation}
		Here, $i$ denotes a choice of a section of the map $\mathbf{ev}_0\colon \p \mathbf{U} \mathbf{F}\mathbf{U}X \rightarrow \mathbf{U} \mathbf{F}\mathbf{U}X$, which is equal to $\mathbf{U}(\mathbf{ev}_0\colon \p \mathbf{F}\mathbf{U}X \rightarrow  \mathbf{F}\mathbf{U}X)$ and is therefore a trivial fibration. Hence, the existence of $i$ is ensured by the cofibrancy assumption on $\mathbf{F}\mathbf{U}X$. Suppose we manage to find a diagonal filler $\Gamma\colon \mathbf{U} \mathbf{F}\mathbf{U}X\rightarrow \p \mathbf{U} \mathbf{F}\mathbf{U}X$  for such square, we would then have that $\Gamma$ is a weak equivalence by the 2-out-of-3 property, since $\mathbf{ev}_1$ and $\mathbf{U}(\mathbf{ev}_1\circ i)$ both are, and by construction $\mathbf{ev}_1\circ \Gamma=\mathbf{U}(\mathbf{ev}_1\circ i) $. This, in turn, implies that $ \eta_{\mathbf{U}X} \circ \mathbf{U}(\epsilon_X)$ is also a weak equivalence, and moreover, thanks to the triangle identities, we have $ \mathbf{U}(\epsilon_X)\circ \eta_{\mathbf{U}X}=1_{\mathbf{U} \mathbf{F}\mathbf{U}X}$. This implies that $\mathbf{U}(\epsilon_X) $ is a weak equivalence and concludes the proof. Therefore, all is left to do is to find the filler $\Gamma$, and this is accomplished separately in the next two lemmas.
	\end{proof}
\end{prop} 
We define a set of maps $\alpha_k\colon D_k \rightarrow \mathbb{I}_k$, where  the codomain is obtained by freely adding a pair of $k$-cell going in the opposite direction as well as a pair of $(k+1)$-cells connecting the two possible composites with identities (with respect to the system of composition chosen to define \eqref{D_W def}). For example, if $k=1$, then $\mathbb{I}_k$ is the free $\mathfrak{C}$-model on the following pasting diagram:
\begin{equation}
\label{alpha_1}
\bfig 
\morphism(0,0)|a|/@{>}@<0pt>/<300,300>[0` 1;g]
\morphism(300,300)|a|/@{>}@<0pt>/<300,-300>[1` 0;f]
\morphism(0,0)|a|/@{=}@<0pt>/<600,0>[0` 0;]
\morphism(300,300)|a|/@{=}@<0pt>/<600,0>[1` 1;]
\morphism(600,0)|r|/@{>}@<0pt>/<300,300>[0` 1;k]
\morphism(300,50)|a|/@{=>}@<0pt>/<0,200>[` ;]
\morphism(600,300)|a|/@{=>}@<0pt>/<0,-200>[` ;]
\efig 
\end{equation} and $\alpha_1$ picks out $f$.
\begin{lemma}
	The map $\eta_{\mathbf{U}X}$ is obtained as a transfinite composite of pushouts of maps of the form $\alpha_k\colon D_k \rightarrow \mathbb{I}_k$ for $ k\geq 1$.
	\begin{proof}
		The claim follows from the same argument given in Proposition 2.2 in \cite{Nik}, which proves that the unit is a $\{\alpha_k\}_{k\geq 0}$-cell complex provided the maps $\alpha$'s are monomorphisms. These maps are cofibrations, so it suffices to show that cofibrations are monomorphisms. In the language of \cite{JB2}, we can view $\mathbf{Mod}(\mathfrak{C}^{\mathbf{W}})$ as the cofiltered limit of a tower of iterated injectives, starting from $\left([\G^{op},\mathbf{Set}],\mathbb{I}_0=\{S^{k-1}\rightarrow D_k\}_{k \geq 0}\right)$. Since maps in $\mathbb{I}_0$ are monomorphisms, we see that $F\mathbb{I}_0$-cell complexes in $\mathbf{Inj}(\mathbb{I}_1)$, where $F_0$ is the left adjoint to the forgetful functor into globular sets, are again monomorphisms since, by Proposition 2.18 (ibid.), these are $\mathbb{I}_0$-cell complexes. Therefore we can iterate this construction and get $\mathbb{I}\subset \mathbf{Mod}(\mathfrak{C}^{\mathbf{W}})$ as a filtered colimit of $F_i(\mathbb{I}_0)$, with $F_i$ being the left adjoint to the forgetful functor down to globular sets, where each set $F_i(\mathbb{I}_0)$ consists of monomorphisms by induction. It follows that $\mathbb{I}$ consists of monomorphisms.
	\end{proof}
\end{lemma}
\begin{lemma}
Let $p\colon E \rightarrow B$ be a fibration in $\mathbf{Mod}(\mathfrak{C}^{\mathbf{W}})$, then $p$ has the right lifting property with respect to the set of maps $\{\tau_k\colon D_k \rightarrow D_{k+1}\}_{k \geq 0}$.
\begin{proof}
Suppose given a $(k+1)$-cell $H\in B_{k+1}$ with $H\colon g \rightarrow p(f)$. By assumption there exists a $(k+1)$-cell $h_0 \in E_{k+1}$ with $h_0\colon f \rightarrow \bar{g}$ and $p(h_0)=H^{-1}$. We have $p({h_0}^{-1})=({H^{-1}})^{-1}$, so that there exists a $(k+2)$-cell $\gamma\colon p({h_0}^{-1}) \rightarrow H$ which we can lift to get a $(k+2)$-cell $\overline{\gamma}\colon{ h_0}^{-1} \rightarrow \overline{H}$. By construction, $p(\overline{H})=H$.
\end{proof}
\end{lemma}
\begin{lemma}
	The commutative square \eqref{homotopy for epsilon} admits a diagonal filler $\Gamma \colon \mathbf{U} \mathbf{F}\mathbf{U}X \rightarrow \mathbf{U} \p \mathbf{F}\mathbf{U}X$.
	\begin{proof}
	Thanks to the previous result and to the fact that $F(\alpha_k)=\alpha_k$ (where, with a minor abuse of language, we have denoted with the same expression the interpretation of $\alpha_k$ in $\mathbf{Mod}(\mathfrak{C})$ on the left and in $\mathbf{Mod}(\mathfrak{C}^{\mathbf{W}})$ on the right) with $F$ being the left adjoint to the forgetful functor $\mathbf{Mod}(\mathfrak{C}^{\mathbf{W}})\rightarrow \mathbf{Mod}(\mathfrak{C})$, it is enought to show that $\alpha_k$ is a trivial cofibration in $\mathbf{Mod}(\mathfrak{C}^{\mathbf{W}})$, i.e. it has the left lifting properties with respect to fibrations. Suppose given a fibration $p\colon E \rightarrow B$ and a diagram of $k$-cells and $(k+1)$-cells in $B$ of the form:
	\[
	\bfig 
	\morphism(0,0)|a|/@{>}@<0pt>/<300,300>[p(x)` p(y);g]
	\morphism(300,300)|a|/@{>}@<0pt>/<300,-300>[p(y)` p(x);p(f)]
	\morphism(0,0)|a|/@{=}@<0pt>/<600,0>[p(x)` p(x);]
	\morphism(300,300)|a|/@{=}@<0pt>/<600,0>[p(y)` p(y);]
	\morphism(600,0)|r|/@{>}@<0pt>/<300,300>[p(x)` p(y);k]
	\morphism(300,50)|a|/@{=>}@<0pt>/<0,200>[` ;\gamma]
	\morphism(675,275)|r|/@{=>}@<0pt>/<0,-150>[` ;\beta]
	\efig 
	\]
Since identities are preserved by any map, the domain of $\gamma$ is of the form $p(\mathbf{id}_{k-1}(x))$, therefore we can lift $\gamma$ to a $(k+1)$-cell $\gamma'\colon \mathbf{id}_{k-1}(x)  \rightarrow g_0$ in $E$. There exists a $(k+1)$-cell $\mathbf{k}^l_k(p(f)) g\colon p(f_l^{-1}g_0)=p(f)_l^{-1} p(f)g \simeq g$ in $B$ (where $(\cdot)_l^{-1}$ denotes the left inverse operation), which we can lift it to get a $(k+1)$-cell $\overline{\delta}\colon f^{-1}_lg_0\rightarrow \overline{g}$ in $E$, and in particular $p(\overline{g})=g$. Thanks to Remark \ref{two-sided vs left and right}, there exists a cell $\chi\colon g_0 \rightarrow f f_l^{-1} g_0$, and the composite $f \overline{\delta}\circ \chi \circ \gamma'\colon \mathbf{id}_{k-1}(x) \rightarrow f\overline{g} $ lives over a cell of the form $\mathbf{id}_{k-1}(p(x))  \rightarrow p(f)g$ which is homotopic to $\gamma$. Therefore, by lifting this homotopy and taking its target, we get a cell $\overline{\gamma}\colon f \overline{g} \rightarrow  \mathbf{id}_{k-1}(x) $ which is the lift we were looking for.
The case of $\beta$ is similar to the one we have just considered thanks to the previous lemma.
	\end{proof}
\end{lemma}	
As anticipated earlier, when $n=3$ we can use the results of Section \ref{S} in conjunction with those of Section 6 in \cite{EL} to obtain an endofunctor $\p$ on $\mathbf{Mod}(\mathfrak{C}^{\mathbf{W}})$ with the desired properties.
\begin{prop}
	\label{path object for Batanin 3-groupoids}
	Given a coherator $\mathfrak{C}$ for $\nCat{3}$, there exists a functor
	\[\p\colon \mathbf{Mod}(\mathfrak{C}^{\mathbf{W}}) \rightarrow \mathbf{Mod}(\mathfrak{C}^{\mathbf{W}})\] equipped with a natural transformation $\mathbf{ev}\colon \p \Rightarrow \mathbf{Id}\times \mathbf{Id}$ which is a pointwise fibration.
	
	Moreover the composites with the product projections $\mathbf{ev}_i\overset{def}{=}\pi_1\circ\mathbf{ev}$ are trivial fibrations for $i=0,1$.
\end{prop}
It follows that, in the situation of the previous proposition, $\mathfrak{C}^{\mathbf{W}}$ is a coherator for 3-groupoids, and we can now present the central result of this work.
\begin{thm}
	\label{model structure}
	There exists a cofibrantly generated semi-model structure on the category $\ngpd{3}\cong \mathbf{Mod}(\mathfrak{C}^{\mathbf{W}})$ of Grothendieck 3-groupoids of type $\mathfrak{C}^{\mathbf{W}}$, whose set of generating cofibrations (resp. trivial cofibrations) consists of boundary inclusions $\{S^{k-1}\rightarrow D_k\}_{0\leq k \leq 4}$ (resp. source maps $\{\sigma_k\colon D_k \rightarrow D_{k+1}\}_{0\leq k \leq 2})$, where by definition we set $S^3\rightarrow D_4$ equal to $(1,1)\colon S^3\rightarrow D_3$. The weak equivalences coincide with the class $\mathbf{W}$ defined in \eqref{def of weak eq}, and all the objects are fibrant.
	\begin{proof}
		Thanks to the previous corollary, we have that $\mathfrak{C}^{\mathbf{W}}$ satisfies all the hypotheses of Theorem \ref{semi model str characterization}, and this concludes the proof.
	\end{proof}
\end{thm}
\begin{rmk}
	In the following section we will define a functor \[\p\colon\mathbf{Mod}(\mathfrak{C}^{\mathbf{W}})\rightarrow [\G^{op},\mathbf{Set}]\] where $\mathfrak{C}$ is a coherator for $\infty$-categories.
	It follows from the results of this and the previous section that if one proves the division lemma holds for $\mathfrak{C}^{\mathbf{W}}$-models then it is enough to extend the functor above to one of the form \[\p\colon\mathbf{Mod}(\mathfrak{C}^{\mathbf{W}})\rightarrow \mathbf{Mod}(\mathfrak{C})\simeq \nCat{\infty}\] (as we do in this paper in the 3-dimensional case) to prove that globular sums are contractible in $\mathfrak{C}^{\mathbf{W}}$, i.e. the latter is a coherator for $\infty$-groupoids, thus getting a semi-model structure on Grothendieck $\infty$-groupoids using the same strategy outlined in this section. This would also solve the open problem of making an $\infty$-groupoid à la Batanin, i.e. a $\mathfrak{C}^{\mathbf{W}}$-model (see \cite{BAT}), into a Grothendieck one. Moreover, this would also prove the homotopy hypothesis thanks to the main results in \cite{Hen}.
\end{rmk}
\section{Main constructions (revisited)}
\label{main constructions}
In this section we are going to adapt all the main constructions on the category of $\infty$-groupoids made in our previous work (\cite{EL}) to the context of $\mathfrak{C}$-models and $\mathfrak{C}^{\mathbf{W}}$-models. In what follows, $\mathfrak{C}_n$ will denote a fixed coherator for $n$-categories, sometimes denoted with just $\mathfrak{C}$ when there is no risk of ambiguity. 
\subsection{Relative lifting properties of $\mathbf{Mod}(\mathfrak{C})$}
To obtain the desired results, we need some preliminary lemmas on relative liting properties of $\mathfrak{C}$-models with respect to $\Theta$-models, i.e. strict $\infty$-categories. These are needed since we used contractibility of globular sums in various steps of those constructions, and in this context globular sums are not going to be contractible in general. Recall that the structural functor $F\colon\mathfrak{C}\rightarrow \Theta$ of the homogeneous coherator $\mathfrak{C}$ gives rise to a cocontinuous functor $F\colon\mathbf{Mod}(\mathfrak{C})\rightarrow \mathbf{Mod}(\Theta)\simeq \nCat{\omega}$ by considering the following Kan extension:
\[
\bfig
\morphism(0,0)|a|/@{>}@<0pt>/<600,0>[\mathfrak{C}`\nCat{\omega};y\circ F]

\morphism(0,0)|l|/@{>}@<3pt>/<0,-400>[\mathfrak{C}`\mathbf{Mod}(\mathfrak{C});y]

\morphism(0,-400)|r|/@{-->}@<0pt>/<600,400>[\mathbf{Mod}(\mathfrak{C})`\nCat{\omega};F]
\efig 
\] where the $y$'s denote two (different) instances of the Yoneda embedding.
\begin{lemma}
	\label{filler of spheres}
	An extension problem in $\mathfrak{C}$ of the form:
	\[
	\bfig
	\morphism(0,0)|a|/@{>}@<0pt>/<600,0>[S^{n-1}` A;(f,g)]

	\morphism(0,0)|l|/@{>}@<3pt>/<0,-400>[S^{n-1}`D_{n};j_{n}]

	\morphism(0,-400)|r|/@{-->}@<0pt>/<600,400>[D_{n}` A;]
	\efig 
	\] admits solution if and only its image under $F\colon \mathfrak{C}\rightarrow \Theta$ does so, and moreover such extension can be chosen so as to live over the one in $\Theta$.
	\begin{proof}
		Let's prove the non-trivial implication. Suppose we have a map $H\colon D_n \rightarrow A$ in $\Theta$, with boundary $(F(f),F(g))$. By factoring $H$ into a homogeneous map $p\colon D_n \rightarrow A'$ followed by a globular map $i\colon A' \rightarrow A$, we see by inspection that the pair $(f',g')\overset{def}{=}p\circ j_n\colon S^{n-1}\rightarrow A'$ is admissible: indeed, such are the boundaries of homogeneous maps in $\Theta$. By uniqueness of homogeneous-globular factorizations in $\mathfrak{C}$ we see that $f$ and $g$ have to factor through $A'$ via an admissible pair $(\overline{f},\overline{g})\colon S^{n-1}\rightarrow A'$ that lives over $(f',g')$. It follows that there exists an extension of $(\overline{f},\overline{g})$ to a map $\overline{p}\colon D_n \rightarrow A'$, and therefore the composite $i\circ \overline{p}$ is the extension we are looking for, and lives over $H$ by construction.
	\end{proof}
\end{lemma}
\begin{lemma}
	\label{extension of cellular maps}
	Let $i\colon X \rightarrow Y$ be an $\mathbf{I}$-cellular map in $\mathfrak{C}$ (i.e. the transfinite composite of pushouts of maps in $\mathbf{I}$), and consider the following extension problem, where $A$ is a globular sum:
	\[
	\bfig
	\morphism(0,0)|a|/@{>}@<0pt>/<600,0>[X` A;f]

	\morphism(0,0)|l|/@{>}@<3pt>/<0,-400>[X`Y;i]

	\morphism(0,-400)|r|/@{-->}@<0pt>/<600,400>[Y` A;]
	\efig 
	\]
	Then such an extension exists if and only if $F(f)$ admits an extension along $F(i)$. Moreover, if we fix an extension in $\nCat{\omega}$ then the one in $\mathfrak{C}$ can be chosen to live over that under $F$.
	\begin{proof}
		There is only one non-trivial implication, which follows from Lemma \ref{filler of spheres} and cocontinuity of $F$ by constructing the extension cell by cell.
	\end{proof}
\end{lemma}
\begin{lemma}
	\label{extension of coglobular objects}
	Let $i\colon X \rightarrow Y$ be a map in $\mathbf{Mod}(\mathfrak{C})^{\G}$, such that its latching maps $\hat{L}_n(i)$ are $\mathbf{I}$-cellular maps for every $n\geq 0$. Then an extension problem of the form:
	\[
	\bfig
	\morphism(0,0)|a|/@{>}@<0pt>/<600,0>[X` A;f]

	\morphism(0,0)|l|/@{>}@<3pt>/<0,-400>[X`Y;i]

	\morphism(0,-400)|r|/@{-->}@<0pt>/<600,400>[Y` A;]
	\efig 
	\] where $A$ factors through ${\mathfrak{C}}^{\G}$ (i.e. is pointwise a globular sum), admits a solution if and only if its image under $F^{\G}$ does so in $\nCat{\omega}^{\G}$.
	\begin{proof}
		The non-trivial implication follows from the observation that $F(\hat{L}_n(i))\cong \hat{L}_n(F(i))$ by cocontinuity of $F$, so that one can construct an extension using the usual inductive argument for Reedy categories and the previous lemmas.
	\end{proof}
\end{lemma}
Let us conclude this section with a very useful lemma on fillers of spheres in globular sums.
\begin{lemma}
	\label{vanishing of higher homotopy groups in globular sums}
	Let $A$ be a globular sum in $\mathfrak{C}$ with $n=\dim(A)$. Then every $k$-sphere in $A$ with $k\geq n$ admits a filler. In particular, $D_0$ is contractible, i.e. the unique map $D_0\rightarrow \ast$ has the right lifting property with respect to all boundary inclusions $S^{k-1}\rightarrow D_k$.
	\begin{proof}
		Thanks to Lemma \ref{filler of spheres}, it is enough to prove the statement in $\nCat{\omega}$. If $k>n$ then the only sphere $S^k\rightarrow A$ is given by a pair of identities on the same cell, and therefore it surely admits a filler. If $k=n$ the restriction along one of the inclusions $D_k \rightarrow S^k$ is an identity cell, then the other must be as well, since globular sums in $\Theta$ admits no non-trivial endomorphisms of cells. In this case too, a filler exists. Finally, if we have an $n$-sphere in $A$ consisting of a pair of parallel $n$-cells none of which is an identity, then the claim follows from the fact that an $n$-cell in an $n$-dimensional globular sum in $\Theta$ is uniquely determined by its boundary, as can easily be proven using the combinatorial description of $\Theta$ in terms of trees given in Section 3.3 of \cite{AR1}.
	\end{proof}
\end{lemma}
\subsection{Suspension-loop space adjunction}
We recall the construction of the suspension-loop space adjunction performed in Section 4 of \cite{EL}. In this case, given $X\in \mathbf{Mod}({\mathfrak{C}^{\mathbf{W}}_n})$ and two $0$-cells $a,b \in X_0$ we produce the $\mathfrak{C}^{\mathbf{W}}_{n-1}$-model of morphisms from $a$ to $b$, denoted by $\Omega (X,a,b)$. For sake of simplicity we only consider the case $n=\infty$ and omit the subscript, leaving the task of modifying this to fit into the finite-dimensional case to the interested reader.

This functor will be then extended to an adjunction of the form
\[\xymatrixcolsep{1pc}
\vcenter{\hbox{\xymatrix{
			**[l]\mathbf{Mod}(\mathfrak{C}^{\mathbf{W}}_{n-1}) \xtwocell[r]{}_{\Omega}^{\Sigma}{'\perp}& **[r] S^0 \downarrow \mathbf{Mod}(\mathfrak{C}^{\mathbf{W}}_{n})
}}}
\] where the category on the right is the slice category under $S^0$.

Using the language of trees it is straightforward to construct a functor \[\Sigma \colon \nCat{\omega} \rightarrow S^0\downarrow \nCat{\omega}\] where $\nCat{\omega}$ denotes the category $\mathbf{Mod}(\Theta)$ of strict $\infty$-categories. As previously done, we will construct $\Sigma \colon \mathbf{Mod}(\mathfrak{C})\rightarrow S^0\downarrow \mathbf{Mod}(\mathfrak{C})$ as the cocontinuous globular extension of a functor $\Sigma \colon \mathfrak{C}\rightarrow S^0 \downarrow \mathbf{Mod}(\mathfrak{C})$ by induction on the defining tower of $\mathfrak{C}$, assuming that at each step the following square commutes:
\[
\bfig 
\morphism(0,0)|a|/@{>}@<0pt>/<400,0>[\mathfrak{C}_{\alpha}`\mathfrak{A} ;\Sigma]
\morphism(0,0)|a|/@{>}@<0pt>/<0,-400>[\mathfrak{C}_{\alpha}`\Theta;F]
\morphism(0,-400)|a|/@{>}@<0pt>/<400,0>[\Theta` \Theta ;\Sigma]
\morphism(400,0)|r|/@{>}@<0pt>/<0,-400>[\mathfrak{A}` \Theta;]
\efig 
\] Here, $\mathfrak{A}$ is obtained by taking a bijective on objects-fully faithful factorization of the map $\Sigma\colon\Theta_0 \rightarrow S^0\downarrow \mathbf{Mod}(\mathfrak{C})$, which is defined as before, and so it clearly comes endowed with a functor down to $\Theta$. Implicitly, we are assuming that $\Sigma$ factors through $\mathfrak{C}$. The case $\Theta_0=\mathfrak{C}_0$ has already been discussed, and the limit ordinal case is trivial. Let us then suppose we have the construction on $\mathfrak{C}_{\alpha}$, and that $\mathfrak{C}_{\alpha+1}$ is obtained by adding an operation $\rho \colon D_n \rightarrow A$ with boundary an admissible pair $(f,g)$. It is easy to see that $\Sigma\colon \Theta \rightarrow \Theta$ preserves admissible pairs, so that we can define $\Sigma(\rho)$ as the choice of an extension to the following diagram in $\mathfrak{A}\subset S^0\downarrow \mathbf{Mod}(\mathfrak{C})$ (which exists in $\mathfrak{C}$ and it is automatically under $S^0$):
\[
\bfig
\morphism(0,0)|a|/@{>}@<0pt>/<600,0>[S^{n}`\Sigma A;(\Sigma(f),\Sigma(g))]

\morphism(0,0)|l|/@{>}@<3pt>/<0,-400>[S^n`D_{n+1};j_{n+1}]

\morphism(0,-400)|r|/@{-->}@<0pt>/<600,400>[D_{n+1}`\Sigma A;\Sigma (\rho)]
\efig 
\] which again satisfies the inductive hypothesis.

If we want to adapt this construction to the case of $\mathbf{Mod}(\mathfrak{C}^{\mathbf{W}})$, we simply consider the pushout \eqref{D_W def} that defines this globular theory. The definition of $\mathfrak{A}$ is the same as above, and now constructing a functor $\Sigma\colon \mathfrak{C}^{\mathbf{W}} \rightarrow \mathfrak{A}$ amounts to define an action on $\mathfrak{C}$ and on inverses in a compatible way. More precisely, we define a map $\Sigma\colon\mathfrak{C} \rightarrow S^0\downarrow \mathbf{Mod}(\mathfrak{C}^{\mathbf{W}})$ by composing the functor $\Sigma\colon\mathfrak{C} \rightarrow S^0\downarrow \mathbf{Mod}(\mathfrak{C})$ defined above with the natural map $S^0\downarrow \mathbf{Mod}(\mathfrak{C}) \rightarrow S^0\downarrow \mathbf{Mod}(\mathfrak{C}^{\mathbf{W}})$. In addition, we define a map $\Sigma\colon\Theta_0 [\mathbf{comp},\mathbf{id}, \mathbf{inv}] \rightarrow S^0\downarrow \mathbf{Mod}(\mathfrak{C}^{\mathbf{W}})$ by setting $\Sigma(\mathbf{i}^{\epsilon}_k)=\Sigma^{k}(\mathbf{i}_1)$ for $\epsilon=l,r$ and similarly for the maps ${\mathbf{k}_m}^{\epsilon}$. The universal property of pushouts yields the desired map $\Sigma\colon \mathfrak{C}^{\mathbf{W}} \rightarrow \mathfrak{A}\subset \mathbf{Mod}\left(\mathfrak{C}^{\mathbf{W}}\right) $, and consequently a functor \[\Sigma\colon \mathbf{Mod}(\mathfrak{C}^{\mathbf{W}}) \rightarrow \mathbf{Mod}(\mathfrak{C}^{\mathbf{W}}) \] by left Kan-extension.
By adjunction, the underlying globular set of $\Omega (X,a,b)$ is given by 
\[\Omega (X,a,b)_k\colon =\{x \in X_{k+1}| \ s_0^{k+1}(x)=a, \ t_0^{k+1}(x)=b \} \]
We will often denote $\Omega(X,a,b)$ simply by $X(a,b)$.
\begin{rmk}
	\label{cocontinuity of UoSIgma}
	If we compose $\Sigma$ with the forgetful functor $U\colon S^0 \downarrow \mathbf{Mod}(\mathfrak{C}^{\mathbf{W}}_n) \rightarrow \mathbf{Mod}(\mathfrak{C}^{\mathbf{W}}_n) $, we get a functor which is no longer cocontinuous.
	Nevertheless, it is well known that $U$ creates connected colimits, therefore $U\circ \Sigma$ preserves all such.
	Because $\Sigma(I_{n-1})\subset I_n$, where $I_{k}$ is the set of maps defined in Definition \ref{def cofi}, we therefore have that $U\circ \Sigma$ preserves cofibrations (i.e. it sends maps in $\mathbb{I}_{n-1}$ to maps in $\mathbb{I}_{n}$, the respective saturations of $I_{n-1}$ and $I_n$). A similar situation is treated in Lemma 1.3.52 of \cite{CIS}.
\end{rmk}
Given a map $(\alpha,\beta)\colon S^{k} \rightarrow X$, seen as a map $(\hat{\alpha}, \hat{\beta})\colon S^{k-1} \rightarrow  X(a,b)$, where $a=s_0^{k} (\alpha)$ and $b=t_0^{k} (\beta)$, then it holds true that
\[\pi_{k-1}(X(a,b),\hat{\alpha}, \hat{\beta}) \cong \pi_{k}(X,\alpha,\beta)\] (see also \cite{AR2}, Definition 4.11).
%
%

The following lemma will be used quite frequently in the forthcoming sections. Its proof is straightforward and it is thus left to the reader.
\begin{lemma}
	\label{decomposition of glob sum}
	For every globular sum $A$ there exist unique globular sums $\alpha_1,\ldots , \alpha_q$ such that 
	\begin{equation}
	A\cong \Sigma \alpha_1 \plus{ D_0} \Sigma \alpha_2\plus{ D_0} \ldots \plus{ D_0} \Sigma \alpha_q
	\end{equation} the colimit being taken over the maps
	\[
	\bfig 
	\morphism(0,0)|a|/@{>}@<0pt>/<-400,-400>[D_0`\Sigma \alpha_{i};\top]
	\morphism(0,0)|a|/@{>}@<0pt>/<400,-400>[D_0`\Sigma \alpha_{i+1};\perp]
	\efig 
	\] where we denote by $(\Sigma B, \perp,\top)$ the image under the functor $\Sigma\colon \ngpd{(n-1)} \rightarrow S^0\downarrow \ngpd{n}$ of any globular sum $B$ .
\end{lemma}
\subsection{Cylinders}
We now recall the important features of cylinders defined in \cite{EL}, this time in the context of $\mathfrak{C}^{\mathbf{W}}_n$-models. As before, we only provide details for the case $n=\infty$ and we drop subscripts.

 Cylinders should be thought as homotopies between cells that are not parallel, so that one needs to provide first homotopies between the 0-dimensional boundary, then between the 1-dimensional boundary adjusted using those homotopies, and so on. This is the right notion of (pseudo)-natural transformation in this context.

\begin{ex}
	Let $n\geq 2$. By definition, $\cyl (D_0)$ is the free model on a $1$-cell. Therefore, giving a $0$-cylinder in a $\mathfrak{C}^{\mathbf{W}}$-model $X$ is equivalent to specifying one of its $1$-cells.
	
	If we go one dimension up, we have that a $1$-cylinder $C\colon \cyl(D_1) \rightarrow X$ consists of the following data
	\[
	\bfig
	\morphism(0,0)|a|/@{>}@<0pt>/<300,0>[a`b;\alpha]
	\morphism(0,0)|l|/@{>}@<0pt>/<0,-300>[a`c;f]
	\morphism(300,0)|r|/@{>}@<0pt>/<0,-300>[b`d;g]
	\morphism(0,-300)|b|/@{>}@<0pt>/<300,0>[c`d; \beta]
	\morphism(250,-75)|a|/@{=>}@<0pt>/<-150,-150>[`; C]
	\efig\]
	Following the notation in the next section, we have that $f=C \circ \cyl (\sigma)$ and  $g=C \circ \cyl (\tau)$. Moreover, $\alpha=C \circ \iota_0$ and $\beta=C \circ \iota_1$.
	
	This cylinder represents the fact that to give a ``homotopy'' from $\alpha$ to $\beta$ we first have to give one between $a$ and $c$, and one from $b$ to $d$. Only then can we compose these with the cells we want to compare, and consider the 2-cells such as $C$ that fill the resulting square, thus giving us the homotopy we are looking for.
\end{ex}
\begin{defi}
		\label{cyl defi}
	Given a $\mathfrak{C}^{\mathbf{W}}$-model $X$, we inductively define $k$-cylinders $F\colon A \curvearrowright B$ in $X$ for a pair of $k$-cells $A,B$ in $X$, with a simultaneous induction on $k$.
	
	A $0$-cylinder $F\colon A \curvearrowright B$ in $X$ is just a $1$-cell $F\colon A \curvearrowright B$. Given $k>0$, a $k$-cylinder $F$ in $X$ consists of a pair of 1-cells $F_s,F_t$ together with a $(k-1)$-cylinder $\overline{F}\colon F_t A \curvearrowright BF_s$ in the $\mathfrak{C}^{\mathbf{W}}$-model  $ X\left(s(F_s),t(F_t) \right)$, where juxtaposition of cells stands for the result of composing them using the maps $w$ defined in Definition \ref{whiskering w}. Given a $k$-cylinder $C$ in $X$, with $k>0$, we get by induction source and target cylinders denoted, respectively, by $s(C)$ and $t(C)$. In fact, if $k=1$ we set $s(C)=C_s, t(C)=C_t$, and the inductive step is straightforward.
\end{defi}
It is possible to construct a coglobular object $\cyl(D_{\bullet})\colon \G_n \rightarrow \mathbf{Mod}(\mathfrak{C}^{\mathbf{W}})$ that corepresents these cylinders, and their coglobular structure, in the sense that the set of $k$-cylinders in a $\mathfrak{C}^{\mathbf{W}}$-model $X$ is given by $\mathbf{Mod}(\mathfrak{C}^{\mathbf{W}})(\cyl(D_k),X)$, and the source and target cylinders are induced by precomposition with the structural map of this $n$-coglobular object.

Moreover, this comes endowed with a map $(\iota_0,\iota_1)\colon D_{\bullet}\coprod D_{\bullet}\rightarrow \cyl(D_{\bullet})$ in $\mathbf{Mod}(\mathfrak{C}^{\mathbf{W}})^{\G_n}$, where $D_{\bullet}$ is the canonical $n$-coglobular $n$-groupoid of globes, which is a direct cofibration in the sense of \ref{def cofi} (i.e. it belongs to the class $\mathbf{I}^{\G}$). If $F\colon A \curvearrowright B$ is a $k$-cylinder in $X$, represented by a map $F\colon \cyl(D_k)\rightarrow X$, then precomposition with $\iota=(\iota_0,\iota_1)$ gives back the pair of $k$-cells $(A,B)$.

Finally, the source of the latching map $\hat{L}_k(\iota_0,\iota_1)$ will be denoted by $\partial \cyl (D_k)$. This can be constructed as the following pushout 
\begin{equation}
\label{2}
\bfig 
\morphism(0,0)|a|/@{>}@<0pt>/<750,0>[S^{k-1}\coprod S^{k-1}` \cyl(S^{k-1});(\iota_0,\iota_1)]
\morphism(0,0)|a|/@{>}@<0pt>/<0,-500>[S^{k-1}\coprod S^{k-1}`D_{k} \coprod D_{k};]
\morphism(0,-500)|a|/@{>}@<0pt>/<750,0>[D_{k} \coprod D_{k}` \partial \cyl(D_{k});]
\morphism(750,0)|a|/@{>}@<0pt>/<0,-500>[\cyl(S^{k-1})`  \partial \cyl(D_{k});]
\efig
\end{equation}
Therefore, we get a cofibration of $\mathfrak{C}^{\mathbf{W}}$-models
 \begin{equation}
 \label{boundary cyl cofi}
\hat{L}_n(\iota_0,\iota_1)\colon\partial(\cyl(D_k))\rightarrow \cyl_n(D_k)
\end{equation}
\begin{ex}
	Let $X$ be an $\mathfrak{C}^{\mathbf{W}}$-model for $n\geq 3$. A $2$-cylinder $C\colon A \curvearrowright B$ in $X$ consists of a pair of $1$-cells $f=C_{s_0}, g=C_{t_0}$ and a $1$-cylinder $\bar{C}\colon gA\curvearrowright Bf$ in $  X\left( s(f),t(g)  \right) $. It can also be represented as the following data in $X$
	\[
	\xymatrix{s(f) \ar[dd]_{f} \rrtwocell{A}  \Dtwocell{ddrr}{} &&s(g) \ar[dd]^{g}  \\ & \\
		t(f)\ar[rr]^{t(B)} && t(g) }
	\quad \xymatrix{ \\ \Rrightarrow}\quad
	\xymatrix{s(f) \ar[rr]^{s(A)} \ar[dd]_{f}  \Dtwocell{ddrr}{} &&s(g) \ar[dd]^{g} \\ \\
		t(f) \rrtwocell{B} &&t(g)}
	\]
	Or, in a way that better justifies its name, as 
	\begin{equation}
	\cd[@+3em]{
		A \ar@/^1em/[r]^-{} \ar@/_1em/[r]_-{} \ar[d]_-{f}
		\Dtwocell{r}{A} \Dtwocell[0.33]{dr}{}
		\Dtwocell[0.67]{dr}{} \ulthreecell{dr}{} &
		B \ar[d]^-{g} \\
		A' \ar@/^1em/[r]^-{} \ar@/_1em/[r]_-{} \Dtwocell{r}{B}  & B'
	}
	\end{equation}
	where the front face is the square (i.e. $1$-cylinder) given by $t(C)$, and the back one is $s(C)$.
\end{ex}
\begin{defi}
	We call $\partial \cyl (D_k)$ the boundary of the $k$-cylinder.
	Given a $k$-cylinder in $X$ $C\colon\cyl (D_k) \rightarrow X$, we call the boundary of $C$, denoted by $\partial C$, the following composite
	\[
	\bfig
	\morphism(0,0)|a|/@{>->}@<0pt>/<600,0>[ \partial \cyl_n (D_k)`    \cyl_n (D_k);] \morphism(600,0)|a|/@{>}@<0pt>/<500,0>[ \cyl_n (D_k)`   X;C]
	\efig
	\]
\end{defi}
Thanks to \eqref{2}, we know that specifying the boundary of an $k$-cylinder in an $n$-groupoid $X$ is equivalent to providing the following data:
\begin{itemize}
	\item a pair of parallel $(k-1)$-cylinders $C\colon A\curvearrowright B,D\colon	A'\curvearrowright B'$ in $X$;
	\item a pair of $k$-cells $\alpha\colon A \rightarrow A', \beta \colon B \rightarrow B'$ in $X$.
\end{itemize}
We can define a map of $n$-coglobular $n$-groupoids $\mathbf{C}_{\bullet}\colon \cyl(D_{\bullet})\rightarrow D_{\bullet}$ that fits into the following factorization of the codiagonal map
\begin{equation}
\label{codiag fact}
\bfig
\morphism(0,0)|a|/@{>->}@<0pt>/<750,0>[ D_{\bullet} \coprod D_{\bullet}`\cyl_n (D_{\bullet});(\iota_0,\iota_1)]
\morphism(750,0)|a|/@{>}@<0pt>/<500,0>[ \cyl_n (D_{\bullet})`D_{\bullet};\mathbf{C}_{\bullet}]
\efig
\end{equation}
by solving the following lifting problem in $\mathbf{Mod}(\mathfrak{C}^{\mathbf{W}})^{\G}$:
\[\bfig
\morphism(0,0)|a|/@{>}@<0pt>/<500,0>[ D_{\bullet} \coprod D_{\bullet}`D_{\bullet};\nabla]
\morphism(0,0)|a|/@{>->}@<0pt>/<0,-400>[D_{\bullet} \coprod D_{\bullet}`\cyl_n(D_{\bullet});(\iota_0,\iota_1)]
\morphism(0,-400)|r|/@{-->}@<0pt>/<500,400>[\cyl_n(D_{\bullet})`D_{\bullet};\mathbf{C}_{\bullet}]
\efig
\] This is done using \ref{extension of coglobular objects} and the fact that such an extension exists in $\nCat{\omega}$ to construct a solution in $\mathbf{Mod}(\mathfrak{C})$, and then applying the left adjoint to the forgetful functor $\mathbf{U}\colon \mathbf{Mod}(\mathfrak{C}^{\mathbf{W}}) \rightarrow \mathbf{Mod}(\mathfrak{C})$ to get the the extension in $\mathbf{Mod}(\mathfrak{C}^{\mathbf{W}})$.

A degenerate $k$-cylinder $F\colon \cyl^p_q(D_k)\rightarrow X$ is a $k$-cylinder in $X$ whose $p$-iteration of the source and $q$-iteration of the target are collapsed. For instance, this is the picture of a 1-cylinder with degenerate source:
 	\[ 
 \bfig
 \morphism(0,400)|a|/@{>}@<0pt>/<400,0>[a` b;\alpha]
 \morphism(0,400)|a|/@{=}@<0pt>/<0,-400>[a` a;]
 \morphism(0,0)|b|/@{>}@<0pt>/<400,0>[a` c;\beta]
 \morphism(400,400)|r|/@{>}@<0pt>/<0,-400>[b` c;g]
 
 \morphism(200,300)|a|/@{=>}@<0pt>/<0,-200>[`;C]
 \efig
 \] In this case, a map $\cyl^0 (D_1) \rightarrow X$ corresponds to 1-cells $g,\alpha, \beta$ and a 2-cell $C\colon g \alpha \rightarrow \beta$ in $X$. This will also be denoted by $F\colon \alpha \curvearrowright^0 \beta$. See Definition 9.1 in \cite{EL} for a detailed description of the general case.
 
In \cite{EL} we defined the vertical composition of a compatible stack of an $m$-tuple of (possibly degenerate) $k$-cylinders in an $\infty$-groupoid.
This operation takes as input a sequence of $k$-cylinders $F_i\colon A_i \curvearrowright^{p_i}_{q_i} A_{i+1}$ in an $\infty$-groupoid $X$, and produces a $k$-cylinder denoted by \[F_m \otimes F_{m-1} \otimes \ldots \otimes F_1 \colon A_1 \curvearrowright^p_q A_{m+1}\] with $p=\min \{p_i\}_{1\leq i \leq m}$ and $q=\min \{q_j\}_{1\leq j \leq m}$. Moreover, it has the property that \[\epsilon(F_m \otimes F_{m-1} \otimes \ldots \otimes F_1 )=\epsilon(F_m) \otimes \epsilon(F_{m-1}) \otimes \ldots \otimes \epsilon(F_1)\] for $\epsilon=s,t$. This definition can be easily adapted to the case of $\mathfrak{C}$-models, and to get one for $\mathfrak{C}^{\mathbf{W}}$-models we simply apply the free functor to the map of $\mathfrak{C}$-models we are about to construct.
The construction makes use of coherence cylinders which are defined using the contractibility of globular sums in $\ngpd{\infty}$, and so we need to examine this construction more carefully in the context of a categorical coherator. This is really the only issue that needs to be addressed to get such operation.

In detail, we need to solve the following extension problems, as defined in Definition 8.3 of \cite{EL}:
\begin{equation}
\label{coherence cyls}
\bfig
\morphism(0,0)|a|/@{>}@<0pt>/<1200,0>[\Sigma D_{\bullet} \coprod \Sigma D_{\bullet} `D_1^{\otimes m}\plus{D_0}\Sigma D_{\bullet}\plus{D_0} D_1^{\otimes k} ;\psi^{m,k}]
\morphism(0,0)|a|/@{>}@<0pt>/<0,-500>[\Sigma D_{\bullet} \coprod \Sigma D_{\bullet} `\Sigma \cyl (D_{\bullet});\Sigma(\iota)]
\morphism(0,-500)|r|/@{-->}@<0pt>/<1200,500>[\Sigma \cyl (D_{\bullet}) `D_1^{\otimes m}\plus{D_0}\Sigma D_{\bullet}\plus{D_0} D_1^{\otimes k};\Psi^{m,k}]

\morphism(0,-1000)|a|/@{>}@<0pt>/<1400,0>[\Sigma^m D_{\bullet} \coprod \Sigma^m D_{\bullet} `D_1^{\otimes q}\plus{D_0} D_{m+1}\plus{D_{m}}\Sigma^{m}D_{\bullet}\plus{D_0}D_1^{\otimes k};\phi^{q,m,k}]
\morphism(0,-1000)|a|/@{>}@<0pt>/<0,-500>[\Sigma^m D_{\bullet} \coprod \Sigma^m D_{\bullet} `\Sigma^m \cyl (D_{\bullet});\Sigma^m(\iota)]
\morphism(0,-1500)|r|/@{-->}@<0pt>/<1400,500>[\Sigma^m \cyl (D_{\bullet}) `D_1^{\otimes q}\plus{D_0} D_{m+1}\plus{D_{m}}\Sigma^{m}D_{\bullet}\plus{D_0}D_1^{\otimes k};\Phi^{m,k}]

\morphism(0,-2000)|a|/@{>}@<0pt>/<1400,0>[\Sigma^m D_{\bullet} \coprod \Sigma^m D_{\bullet} ` D_1^{\otimes q}\plus{D_0}\Sigma^{m}D_{\bullet} \plus{D_{m}}D_{m+1}\plus{D_0}D_1^{\otimes k};\theta^{q,m,k}]
\morphism(0,-2000)|a|/@{>}@<0pt>/<0,-500>[\Sigma^m D_{\bullet} \coprod \Sigma^m D_{\bullet} `\Sigma^m \cyl (D_{\bullet});\Sigma^m(\iota)]
\morphism(0,-2500)|r|/@{-->}@<0pt>/<1400,500>[\Sigma^m \cyl (D_{\bullet}) ` D_1^{\otimes q}\plus{D_0}\Sigma^{m}D_{\bullet} \plus{D_{m}}D_{m+1}\plus{D_0}D_1^{\otimes k};\Theta^{q,m,k}]
\efig
\end{equation}
Let us explain in detail how $\psi^{m,k}$ is defined and how to get the extension $\Psi^{m,k}$: the other cases can be treated similarly and are thus left to the interested reader. In what follows, we assume we have chosen composition operations $\gamma\colon D_n \rightarrow D_1^{\otimes m}\plus{D_0} D_{n}\plus{D_0} D_1^{\otimes k}$ for $k,m,n>0$, which are compatible with the source and target maps, i.e. \[\gamma \circ \epsilon=\left( 1_{D_1^{\otimes m}}\plus{D_0}\epsilon\plus{D_0}1_{D_1^{\otimes k}}\right) \circ \gamma\] There is no risk of confusion in referring to all such maps as $\gamma$, because the codomain uniquely determines such $\gamma$. Let us restrict, for sake of conciseness, to the case $m,k \neq 0$, and define maps 
\[ \bfig
\morphism(0,0)|a|/@{>}@<0pt>/<1000,0>[\Sigma D_{\bullet} \coprod \Sigma D_{\bullet} `D_1^{\otimes m}\plus{D_0}\Sigma D_{\bullet}\plus{D_0} D_1^{\otimes k} ;\psi^{m,k}]
\efig \]
by setting the first component in dimension $n$ to be given by the composite
\[
\bfig 
\morphism(0,0)|a|/@{>}@<0pt>/<1000,0>[D_{n+1}`D_1^{\otimes m-1}\plus{D_0} D_{n+1}\plus{D_0} D_1^{\otimes k} ;\gamma]
\morphism(1000,0)|a|/@{>}@<0pt>/<2000,0>[D_1^{\otimes m-1}\plus{D_0} D_{n+1}\plus{D_0} D_1^{\otimes k} `D_1^{\otimes m}\plus{D_0} D_{n+1}\plus{D_0} D_1^{\otimes k} ;1_{D_1^{\otimes m-1}}\coprod w \coprod 1_{D_1^{\otimes k}}]	
\efig 
\] 
and the second one to be
\[
\bfig 
\morphism(0,0)|a|/@{>}@<0pt>/<1000,0>[D_{n+1}`D_1^{\otimes m}\plus{D_0} D_{n+1}\plus{D_0} D_1^{\otimes k-1} ;\gamma]
\morphism(1000,0)|a|/@{>}@<0pt>/<2000,0>[D_1^{\otimes m}\plus{D_0} D_{n+1}\plus{D_0} D_1^{\otimes k-1} `D_1^{\otimes m}\plus{D_0} D_{n+1}\plus{D_0} D_1^{\otimes k} ;1_{D_1^{\otimes m}}\coprod w \coprod 1_{D_1^{\otimes k-1}}]	
\efig 
\] 
This means that given an $\infty$-groupoid $X$ and a map \[(f_1,\ldots,f_m,\alpha,g_1,\ldots,g_k)	\colon D_1^{\otimes m}\plus{D_0} D_{n+1}\plus{D_0} D_1^{\otimes k} \rightarrow X\] we get a pair of $(n+1)$-cells in $X$ of the form $g_k\ldots g_1 (\alpha f_m)f_{m-1}\ldots f_1$ and $g_k\ldots g_2 (g_1 \alpha) f_mf_{m-1}\ldots f_1$, where juxtaposition is the result of composition using the appropriate $\gamma$.

The extension problem at hand satisfies the hypotheses of Lemma \ref{extension of coglobular objects}, and thus such an extension exists if and only if the corresponding extension problem in $\nCat{\omega}$ admits a solution. The latching map $\hat{L}_n\left( \Sigma(\iota) \right)\cong \Sigma \left( \hat{L}_n(\iota)\right) \colon \Sigma \partial \cyl(D_n)\rightarrow \Sigma \cyl (D_{n})$ is a pushout of the sphere inclusion $\Sigma(S^{n}\rightarrow D_{n+1})\cong S^{n+1}\rightarrow D_{n+2}$, and therefore the extension problem admits a solution thanks to Lemma \ref{vanishing of higher homotopy groups in globular sums}, since $\dim(D_1^{\otimes m}\plus{D_0} D_{n+1}\plus{D_0} D_1^{\otimes k})=n+1$.

After having solved the other extension problems in a similar fashion, we have all the tools we need to define the operation of vertical composition of cylinders inside $\mathbf{Mod}(\mathfrak{C})$.
\subsection{Modifications}
In what follows, we define modifications in $\mathbf{Mod}(\mathfrak{C})$, which induce modifications in $\mathbf{Mod}(\mathfrak{C}^{\mathbf{W}})$ upon applying the free functor $F\colon \mathbf{Mod}(\mathfrak{C}) \rightarrow \mathbf{Mod}(\mathfrak{C}^{\mathbf{W}})$ (i.e. the left adjoint to the obvious forgetful functor) to the coglobular object representing them.

Given $k$-cells $A,B$ and 2-cells $b,c$ in a given $\mathfrak{C}$-model $X$ with $t^2(b)=s^k(A)$ and $t^k(B)=s^2(c)$, we make use of $(k-1)$-cylinders in $  X\left(s^2(b),t^2(c) \right)$ of the form $\Gamma(b,A)\colon As(b)\curvearrowright At(b)$ or $\Upsilon(B,c)\colon s(c)B\curvearrowright t(c)B$. These were obtained in \cite{EL} by using contractibility of globular sums and the fact that $(\iota_0,\iota_1)\colon D_{\bullet} \coprod D_{\bullet}  \rightarrow \cyl_n(D_{\bullet})$ is a direct cofibration, therefore this needs some extra justification in this context. We only construct $\Gamma$, the other case being entirely analogous. 

We have to find an extension in the following diagram:
\[
\bfig 
\morphism(0,0)|a|/@{>}@<0pt>/<800,0>[\Sigma D_{\bullet} \coprod \Sigma D_{\bullet} `D_2\plus{D_0}\Sigma D_{\bullet};(a,b)]
\morphism(0,0)|a|/@{>}@<0pt>/<0,-400>[\Sigma D_{\bullet} \coprod \Sigma D_{\bullet} `\Sigma \cyl (D_{\bullet});\Sigma(\iota)]
\morphism(0,-400)|r|/@{-->}@<0pt>/<800,400>[\Sigma \cyl (D_{\bullet}) `D_2\plus{D_0}\Sigma D_{\bullet};\Gamma]
\efig 
\] where $a=(\sigma_1 \amalg_{ D_0}1)\circ w, \ b=(\tau_1 \amalg_{ D_0}1)\circ w$ (see Definition \ref{whiskering w}). At this point it is enough to observe that $\dim(D_2\plus{D_0}\Sigma D_{n})=n+1$ and that the latching map of the vertical arrow at $n$ is a pushout of the boundary inclusion $S^{n+1}\rightarrow D_{n+2}$ to conclude that an extension exists thanks to Lemma \ref{vanishing of higher homotopy groups in globular sums}.
\begin{defi}
Given a $\mathfrak{C}$-model $X$, a modification in $X$ between $k$-cylinders $\Theta\colon C \Rightarrow D$ amounts to a pair of $2$-cells $\Theta_s \colon s^k(C)\rightarrow s^k (D)$, $\Theta_t \colon t^k (D) \rightarrow t^k(C)$ together with a modification of $(k-1)$-cylinders in $X(x,y)$  \[\bar{\Theta}\colon\Upsilon(\iota_0 C,\Theta_t)\otimes \bar{C}\otimes\Gamma(\Theta_s,\iota_1 C)\Rightarrow \bar{D}\]
where $x=s^k(C)\circ \sigma, \ y=t^k(C)\circ \tau$.
\end{defi}

If $k=1$ then we can depict $C$ and $D$ as, respectively
\[
\bfig
\morphism(0,0)|a|/@{>}@<0pt>/<300,0>[a`b;\alpha]
\morphism(0,0)|l|/@{>}@<0pt>/<0,-300>[a`c;f]
\morphism(300,0)|r|/@{>}@<0pt>/<0,-300>[b`d;g]
\morphism(0,-300)|b|/@{>}@<0pt>/<300,0>[c`d; \beta]
\morphism(250,-75)|a|/@{=>}@<0pt>/<-150,-150>[`; \Gamma]

\morphism(600,0)|a|/@{>}@<0pt>/<300,0>[a`b;\alpha]
\morphism(600,0)|l|/@{>}@<0pt>/<0,-300>[a`c;f']
\morphism(900,0)|r|/@{>}@<0pt>/<0,-300>[b`d;g']
\morphism(600,-300)|b|/@{>}@<0pt>/<300,0>[c`d; \beta]
\morphism(850,-75)|a|/@{=>}@<0pt>/<-150,-150>[`; \Delta]
\efig\] 
Therefore, a modification $\Theta\colon C \Rightarrow D$ corresponds to the data of a pair of $2$-cells in $X$ $S\colon f \rightarrow f'$, $T\colon g' \rightarrow g$ and a modification $\bar{\Theta}\colon\Upsilon(\iota_0 C,\Theta_t)\otimes \bar{C}\otimes\Gamma(\Theta_s,\iota_1 C)\Rightarrow \bar{D}$, which is easily seen to correspond to a $3$-cell in $X$ $\tilde{\Theta}\colon (\beta \Theta_s)(\Gamma(\Theta_t \alpha)) \rightarrow \Delta$, where we denote by juxtaposition the result of the appropriate operations $w$ involved in the definition. Notice that if $f=f'$ and $g=g'$, then a modification $\Theta\colon C \Rightarrow D$ such that $\Theta_s$ and $\Theta_t$ are identities can be equivalently thought of as a $3$-cell between the $2$-cells $\Gamma$ and $\Delta$.

Similarly to the case of cylinders, it turns out that modifications are representable by a coglobular object which we denote with $\mathbf{M}_{\bullet}\colon \G_n \rightarrow \mathbf{Mod}(\mathfrak{C})$.  This can be endowed with a direct cofibration of the form 
\begin{equation}
\label{cyl into mod cof}
\cyl_n(D_{\bullet}) \ast \cyl_n(D_{\bullet}) \rightarrow \mathbf{M}_{\bullet}
\end{equation}
where $\cyl_n(D_{\bullet}) \ast \cyl_n(D_{\bullet})$ denotes the colimit of the diagram below:
\[\bfig
\morphism(0,0)|a|/@{>}@<0pt>/<500,0>[D_{\bullet}` \cyl(D_{\bullet});\iota_0 ]
\morphism(0,0)|a|/@{>}@<0pt>/<0,-500>[D_{\bullet}`\cyl(D_{\bullet});\iota_0]
\morphism(500,-500)|a|/@{>}@<0pt>/<0,500>[D_{\bullet}`\cyl(D_{\bullet});\iota_1]
\morphism(500,-500)|b|/@{>}@<0pt>/<-500,0>[D_{\bullet}`\cyl(D_{\bullet});\iota_1]
\efig\]
Note that, by extending $ \mathbf{M}_{\bullet}$ to $\Theta_0$, we can make sense of modifications of the form $\Theta \colon C \Rightarrow D$ for $C,D\colon \cyl(A) \rightarrow X$ in an $\mathfrak{C}$-model $X$.

As an immediate consequence of \eqref{cyl into mod cof} we get the following result.
\begin{lemma}
	\label{modifications in contractible groupoids}
	Let $X$ be a contractible $\mathfrak{D}$-model, with $\mathfrak{D}=\mathfrak{C}$ or $\mathfrak{D}=\mathfrak{C}^{\mathbf{W}}$. Given a pair of $n$-cylinders $C,D\colon A \curvearrowright B$ in $X$, there exists a modification $\Theta\colon C \Rightarrow D$ in $X$. Moreover, the statement remains true if we also require the boundary of $\Theta$ to agree with a prescribed parallel pair of $n$-modifications $\left(\Sigma\colon s(C)\Rightarrow s(D),\Delta\colon t(C) \Rightarrow t(D)\right)$
\end{lemma}
\subsection{The elementary interpretation $\hat{\rho}$}
We now want to refine a result proven in \cite{EL}, namely the construction that takes a homogeneous map $\rho\colon D_k \rightarrow A$ in a homogeneous coherator for $\infty$-categories $\mathfrak{C}$ (i.e. equipped with a globular map $F\colon \mathfrak{C} \rightarrow \Theta$ that detects homogeneous maps) and gives back its ``elementary interpretation'' $\hat{\rho}\colon \cyl(D_k) \rightarrow \cyl(A)$ in $\ngpd{\infty}$. Recall that this map satisfies the following two important properties:
\begin{equation}
\label{hat properties}
\bfig
\morphism(-1000,0)|a|/@{>}@<0pt>/<750,0>[D_k \coprod D_k`A \coprod A;\rho \coprod \rho]
\morphism(-1000,-500)|a|/@{>}@<0pt>/<750,0>[\cyl(D_k)`\cyl(A);\hat{\rho}]
\morphism(-250,0)|r|/@{>}@<0pt>/<0,-500>[A \coprod A`\cyl(A);\iota_0 \coprod \iota_1]
\morphism(-1000,0)|l|/@{>}@<0pt>/<0,-500>[D_k \coprod D_k`\cyl(D_k);\iota_0 \coprod \iota_1]

\morphism(500,0)|a|/@{>}@<0pt>/<1000,0>[\cyl(S^{k-1})`\cyl(A);(\widehat{\rho \circ \sigma},\widehat{\rho \circ \tau})]
\morphism(500,0)|a|/@{>}@<0pt>/<0,-500>[\cyl(S^{k-1})`\cyl(D_k);]
\morphism(500,-500)|r|/@{>}@<0pt>/<1000,500>[\cyl(D_k)`\cyl(A);\hat{\rho}]
\efig
\end{equation}
 We will show that it is actually possible to construct $\hat{\rho}\colon \cyl(D_k) \rightarrow \cyl(A)$ in $\mathbf{Mod}(\mathfrak{C})$. Moreover, there is a (non-canonical) globular map $\mathfrak{C} \rightarrow \mathfrak{G}$, where $\mathfrak{G}$ is the coherator for $\infty$-groupoids we are using to model $\ngpd{\infty}$, which then induces a cocontinuous functor $J \colon \mathbf{Mod}(\mathfrak{C}) \rightarrow \mathbf{Mod}(\mathfrak{G})$ by left Kan extension: the previous construction of $\hat{\rho}$ is then nothing but the image of this refined version of $\hat{\rho}$ under the functor $J$, as will be clear from what follows. The version for $\mathfrak{C}^{\mathbf{W}}$ is obtained in a similar way. Observe that having this definition for homogeneous maps implies its extension to the whole category $\mathfrak{C}$, thanks to the homogeneous-globular factorization system on it.
 
The construction was performed inductively, i.e. assuming that we already have a construction for $\rho_{\epsilon}\colon D_{k-1}\rightarrow A'$, where we denote by $\rho_{\epsilon}$ the homogeneous part of the factorization of $\rho \circ \epsilon$ into a homogeneous map followed by a globular one, for $\epsilon=\sigma,\tau$. The map $\hat{\rho}\colon \cyl(D_k)\rightarrow \cyl(A)$ was obtained by vertically composing a stack of (possibly degenerate) $(k-1)$-cylinders in $\cyl(A)(x,y)$ for $x=\iota_0\circ \partial_{\sigma}^{\dim(A)},y=\iota_1\circ \partial_{\tau}^{\dim(A)}$, and each cylinder in the stack was the transpose of a map of the form $\Sigma(\cyl^{r_i}_{q_i}(D_{k-1}))\rightarrow B$, where $B$ is a globular sum endowed with a map $i_B\colon B \rightarrow \cyl(A)$, as we are going to describe later in this section. This map was obtained by solving an extension problem of the following form, exploiting the contractibiliy of the globular sum $B$:
\[ \bfig
\morphism(0,0)|a|/@{>}@<0pt>/<700,0>[\Sigma\partial \cyl^{r_B}_{q_B}(D_{k-1})`B;]
\morphism(0,0)|a|/@{>}@<0pt>/<0,-400>[\Sigma\partial \cyl^{r_B}_{q_B}(D_{k-1})`\Sigma\cyl^{r_B}_{q_B}(D_{k-1});\Sigma\partial]
\morphism(0,-400)|r|/@{-->}@<0pt>/<700,400>[\Sigma\cyl^{r_B}_{q_B}(D_{k-1})`B;]
\efig \] where the horizontal map is defined by induction. Thanks to Lemma \ref{extension of cellular maps}, it is enough to verify the existence of an extension in $\nCat{\omega}$. The vertical map is a pushout of the boundary inclusion $S^{k}\rightarrow D_{k+1}$, so that such an extension always exist if $\dim(B)\leq n$, thanks to Lemma \ref{vanishing of higher homotopy groups in globular sums}. By construction, the dimension of any $B\in \mathcal{L}(A)$ is at most $n+1$, where $n=\dim(A)\leq k$ (as $\rho$ is homogeneous). Therefore the only case left out is when $n=k$ and $\dim(B)=n+1$, i.e. the new vertex has been added at maximal height. In this case we have $r_B=q_B=k-2$, so that the extension problem is of the form:
\[ \bfig
\morphism(0,0)|a|/@{>}@<0pt>/<400,0>[S^k`B;(f,g)]
\morphism(0,0)|a|/@{>}@<0pt>/<0,-300>[S^k`D_{k+1};j_{k+1}]
\morphism(0,-300)|r|/@{-->}@<0pt>/<400,300>[D_{k+1}`B;]
\efig \]
From the explicit description of the stack of cylinders given in Section 9.3 of \cite{EL}, it is clear that $A=\partial B$ and $f=\partial_{\sigma}\circ p, \ g =\partial_{\tau}\circ p$ for $p$ equal the unique homogeneous map $D_k \rightarrow A$ in $\Theta$. Therefore, the existence of the extension is granted by the fact that $\Theta$ admits extensions of admissible pairs.

As anticipated earlier, we will need some more details of how the cylinders are obtained besides the mere existence of this map. To begin with, we associate an ordered set $\mathscr{L}(A)$ to the globular sum $A$, which is defined by considering all the possible globular sums obtained from $A$ by adjoining a new vertex to the tree associated to $A$ (see Chapter 2 of \cite{AR1} for the correspondence between globular sums and trees). The order is obtained by letting the new edge traverse the tree counterclockwise, starting from the bottom right corner. For instance, this is the (ordered) set $\mathscr{L}(A)$ for $A=D_2 \plus{D_1}D_2 \plus{D_0} D_1$, whose associated tree is $
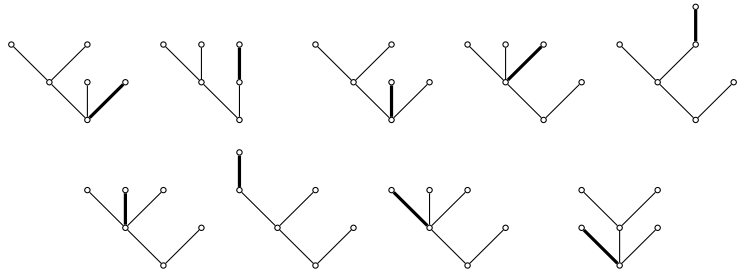
\begin{tikzpicture}
\tikzstyle{every node}=[circle, draw,
inner sep=0pt, minimum width=2pt]
\node[] (A) at (0,-3) {};

\node[] (B) at (-0.5,-2.5) {};
\node[] (C) at (0,-2.5) {};

\node[] (E) at (-1,-2) {};
\node[] (F) at (0,-2) {};

\path [] (A) edge (B);
\path [] (A) edge (C);

\path [] (B) edge (E);
\path [] (B) edge (F);
\end{tikzpicture}$
\[
\begin{tikzpicture}
\label{list of trees}

\tikzstyle{every node}=[circle, draw,
inner sep=0pt, minimum width=2pt]
\node[] (A) at (0,-3) {};

\node[] (B) at (-0.5,-2.5) {};
\node[] (C) at (0.5,-2.5) {};
\node[] (D) at (0,-2.5) {};

\node[] (E) at (-1,-2) {};
\node[] (F) at (0,-2) {};

\path [] (A) edge (B);
\path [draw=red, very thick] (A) edge (C);
\path [] (A) edge (D);
\path [] (B) edge (E);
\path [] (B) edge (F);

\node[] (A') at (2,-3) {};

\node[] (B') at (1.5,-2.5) {};
\node[] (C') at (2,-2) {};
\node[] (D') at (2,-2.5) {};

\node[] (E') at (1,-2) {};
\node[] (F') at (1.5,-2) {};

\path [] (A') edge (B');
\path [draw=red, very thick] (D') edge (C');
\path [] (A') edge (D');
\path [] (B') edge (E');
\path [] (B') edge (F');

\node[] (A'') at (4,-3) {};

\node[] (B'') at (3.5,-2.5) {};
\node[] (C'') at (4,-2.5) {};
\node[] (D'') at (4.5,-2.5) {};

\node[] (E'') at (3,-2) {};
\node[] (F'') at (4,-2) {};

\path [] (A'') edge (B'');
\path [draw=red, very thick] (A'') edge (C'');
\path [] (A'') edge (D'');
\path [] (B'') edge (E'');
\path [] (B'') edge (F'');

\node[] (a) at (6,-3) {};

\node[] (b) at (5.5,-2.5) {};
\node[] (c) at (6,-2) {};
\node[] (d) at (6.5,-2.5) {};

\node[] (e) at (5,-2) {};
\node[] (f) at (5.5,-2) {};

\path [] (a) edge (b);
\path [draw=red, very thick] (b) edge (c);
\path [] (a) edge (d);
\path [] (b) edge (e);
\path [] (b) edge (f);

\node[] (a'') at (8,-3) {};

\node[] (b'') at (7.5,-2.5) {};
\node[] (c'') at (8,-2) {};
\node[] (d'') at (8.5,-2.5) {};
\node[] (f'') at (7,-2) {};
\node[] (e'') at (8,-1.5) {};

\path [] (a'') edge (b'');
\path [draw=red, very thick] (c'') edge (e'');
\path [] (a'') edge (d'');
\path [] (b'') edge (c'');
\path [] (b'') edge (f'');
\end{tikzpicture}\]
\[\begin{tikzpicture}
\tikzstyle{every node}=[circle, draw,
inner sep=0pt, minimum width=2pt]
\node[] (a) at (10,-3) {};

\node[] (b) at (9.5,-2.5) {};
\node[] (c) at (10,-2) {};
\node[] (d) at (10.5,-2.5) {};

\node[] (e) at (9,-2) {};
\node[] (f) at (9.5,-2) {};

\path [] (a) edge (b);
\path [draw=red, very thick] (b) edge (f);
\path [] (a) edge (d);
\path [] (b) edge (e);
\path [] (b) edge (c);

\node[] (a'') at (12,-3) {};

\node[] (b'') at (11.5,-2.5) {};
\node[] (c'') at (12,-2) {};
\node[] (d'') at (12.5,-2.5) {};
\node[] (f'') at (11,-2) {};
\node[] (e'') at (11,-1.5) {};

\path [] (a'') edge (b'');
\path [draw=red, very thick] (f'') edge (e'');
\path [] (a'') edge (d'');
\path [] (b'') edge (c'');
\path [] (b'') edge (f'');

\node[] (a) at (14,-3) {};

\node[] (b) at (13.5,-2.5) {};
\node[] (c) at (14,-2) {};
\node[] (d) at (14.5,-2.5) {};

\node[] (e) at (13,-2) {};
\node[] (f) at (13.5,-2) {};

\path [] (a) edge (b);
\path [draw=red, very thick] (b) edge (e);
\path [] (a) edge (d);
\path [] (b) edge (f);
\path [] (b) edge (c);

\node[] (a''') at (16,-3) {};

\node[] (b''') at (16,-2.5) {};
\node[] (c''') at (16.5,-2) {};
\node[] (f''') at (15.5,-2) {};
\node[] (d''') at (16.5,-2.5) {};

\node[] (e''') at (15.5,-2.5) {};

\path [] (a''') edge (b''');
\path [draw=red, very thick] (a''') edge (e''');
\path [] (a''') edge (d''');
\path [] (b''') edge (c''');
\path [] (b''') edge (f''');
\end{tikzpicture}\]
Then, we construct a zig-zag diagram involving these globular sums $B\in \mathcal{L}(A)$, whose colimit is precisely $\cyl(A)$. This endows each globular sum $B$ in the list with a structural map $i_B\colon B \rightarrow \cyl(A)$ given by a colimit inclusion.

We will now describe the stack of cylinders we get in the case $k=2$, and we assume (without loss of generality) that $\rho$ is a homogeneous operation, which forces $\dim (A)\leq 2$. Also, we describe this stack representably, i.e. we assume given a map $C\colon \cyl(A) \rightarrow X$, with $C\colon U \curvearrowright V$, and we describe the stack of 1-cylinders in $X(x,y)$ that we get out of that.

To each of the globular sums in $\mathscr{L}(A)$ we associate a (possibly degenerate) 1-cylinder in $X(x,y)$, where $x=s^2 \left( X(\rho)(U)\right), \ y=t^2 \left( X(\rho)(V)\right)$. These 1-cylinders will be vertically composable in the order induced by that of $\mathscr{L}(A)$, and the composite will produce a 2-cylinder $\hat{\rho}$ upon transposing along the adjunction $\Sigma\dashv \Omega$ (here, we make use of the fact that an $n$-cylinder is defined to be an $(n-1)$-cylinder in the hom-$\mathfrak{C}$-model between two objects, with the appropriate top and bottom $(n-1)$-cells, see Definition \ref{cyl defi}). The first square, i.e. the one associated with the globular sum $B\in \mathscr{L}(A)$ where the new vertex $*_B$ has been added at height 1 as the maximal element over the root (i.e. the right most one) is given by:
\begin{equation}
\label{cyl 1}
\begin{tikzpicture}
\tikzstyle{every node}=[circle, draw,
inner sep=0pt, minimum width=2pt]
\node[] (A) at (-4,-3) {};

\node[] (B) at (-4.5,-2.5) {};
\node[] (C) at (-3.5,-2.5) {};
\node[] (D) at (-4,-2.5) {};

\node[] (E) at (-5,-2) {};
\node[] (F) at (-4,-2) {};

\path [] (A) edge (B);
\path [draw=red, very thick] (A) edge (C);
\path [] (A) edge (D);
\path [] (B) edge (E);
\path [] (B) edge (F);
\end{tikzpicture} \ \ \ \ \ \ \ \ \ 
\bfig
\morphism(-1000,0)|a|/@{>}@<0pt>/<600,0>[\cdot`\cdot;C_t\rho (U)]
\morphism(-1000,-400)|b|/@{>}@<0pt>/<600,0>[\cdot`\cdot;\rho\left(U_{< p}, C_t U_p\right)]
\morphism(-400,0)|r|/@{>}@<0pt>/<0,-400>[\cdot`\cdot;\simeq]
\morphism(-1000,0)|l|/@{>}@<0pt>/<0,-400>[\cdot `\cdot;\simeq]

\morphism(-650,-150)|l|/@{=>}@<0pt>/<-100,-100>[  ` ; ]
\efig 
\end{equation}

Here, $U_{< p}$ and $U_p$ respectively denote the restriction of $U$ to $\Sigma \alpha_1\plus{D_0}\ldots\plus{D_0} \Sigma \alpha_{p-1}$ and $\Sigma \alpha_p$, where we have considered the decomposition of $A$ as $\Sigma \alpha_1 \plus{D_0}\ldots \plus{D_0} \Sigma \alpha_p$ as in Lemma \ref{decomposition of glob sum}. Furthermore, juxtaposition is the result of composing using the maps introduced in Definition \eqref{whiskering w}, and given a map $W:A\rightarrow X$, which we think as an $A$-shaped pasting diagram in $X$, we denote $X(\rho)(W)$ with $\rho(W)$. Finally, we denoted $C\circ \cyl(\partial^2_{\tau})$ with $C_t$. Both the sides and the interior of the square are obtained by solving extension problems in the globular sum $B$, using the tools developed in this section, and the same holds for all the other cases to follow.

Dually, the last square in the stack is associated to the globular sum $B'$ obtained from $A$ by adjoining a new vertex at height 1 as the minimal element over the root (i.e. the left most one). This time, the square is given by:
\begin{equation}
\label{cyl 2}
\begin{tikzpicture}
\tikzstyle{every node}=[circle, draw,
inner sep=0pt, minimum width=2pt]
\node[] (a''') at (16,-3) {};

\node[] (b''') at (16,-2.5) {};
\node[] (c''') at (16.5,-2) {};
\node[] (f''') at (15.5,-2) {};
\node[] (d''') at (16.5,-2.5) {};

\node[] (e''') at (15.5,-2.5) {};

\path [] (a''') edge (b''');
\path [draw=red, very thick] (a''') edge (e''');
\path [] (a''') edge (d''');
\path [] (b''') edge (c''');
\path [] (b''') edge (f''');
\end{tikzpicture}\ \ \ \ \ \ \ \ \
\bfig
\morphism(-1000,0)|a|/@{>}@<0pt>/<600,0>[\cdot`\cdot;\rho\left(V_1 C_s, V_{>1}\right)]
\morphism(-1000,-400)|b|/@{>}@<0pt>/<600,0>[\cdot`\cdot;\rho(V) C_s]
\morphism(-400,0)|r|/@{>}@<0pt>/<0,-400>[\cdot`\cdot;\simeq]
\morphism(-1000,0)|l|/@{>}@<0pt>/<0,-400>[\cdot `\cdot;\simeq]

\morphism(-650,-150)|l|/@{=>}@<0pt>/<-100,-100>[  ` ; ]
\efig
\end{equation}

Suppose now the new vertex in $B\in \mathscr{L}(A)$ is adjoined at height 1 over the $q$-th 0-cell of $A$, with $q\neq 0,p$. 
The associated square then looks like the one depicted here below:
\begin{equation}
\label{cyl 3}
\begin{tikzpicture}
\tikzstyle{every node}=[circle, draw,
inner sep=0pt, minimum width=2pt]
\node[] (A'') at (4,-3) {};

\node[] (B'') at (3.5,-2.5) {};
\node[] (C'') at (4,-2.5) {};
\node[] (D'') at (4.5,-2.5) {};

\node[] (E'') at (3,-2) {};
\node[] (F'') at (4,-2) {};

\path [] (A'') edge (B'');
\path [draw=red, very thick] (A'') edge (C'');
\path [] (A'') edge (D'');
\path [] (B'') edge (E'');
\path [] (B'') edge (F'');
\end{tikzpicture} \ \ \ \ \ \ \ \ \
\bfig
\morphism(-1000,0)|a|/@{>}@<0pt>/<600,0>[\cdot`\cdot;\rho\left(U_{\leq q},V_q a, V_{> q}\right)]
\morphism(-1000,-400)|b|/@{>}@<0pt>/<600,0>[\cdot`\cdot;\rho\left(U_{< q},a U_q , V_{\geq q} \right)]
\morphism(-400,0)|r|/@{>}@<0pt>/<0,-400>[\cdot`\cdot;\simeq]
\morphism(-1000,0)|l|/@{>}@<0pt>/<0,-400>[\cdot `\cdot;\simeq]

\morphism(-650,-150)|l|/@{=>}@<0pt>/<-100,-100>[  ` ; ]
\efig
\end{equation}

 Here, we have used $a$ to denote the 1-cell in $X$ corresponding to the restriction of the composite map $C\circ i_B\colon B\rightarrow \cyl(A) \rightarrow X$ to the 1-cell in $B$ associated with the newly added vertex.

Suppose now the vertex has been added to $A$ at height 2, to get a globular sum $B \in \mathscr{L}(A)$. We need to consider all the vertices over the one at which the new edge has been adjoined, and again we distinguish according to the position of the newly added vertex.
Firstly, let's consider the case in which it has been added over a copy of $D_1$ (i.e. it is the only vertex above the one to which the new edge is attached). This determines two sub-globular sums of $A$, $A_<$ and $A_>$, which can be informally described as being obtained by removing the 1-cell over which we have attached the new vertex ($A_<$ being the one on the left). Precomposing $U$ (resp. $V$) with the inclusion of $A_<$ (resp. $A_>$) we get an $A_<$-shaped (resp. $A_>$-shaped) pasting diagram in $X$ which we call $U_<$ (resp.$V_>$). The square is then given by:
\begin{equation}
\label{cyl 4}
\begin{tikzpicture}
\tikzstyle{every node}=[circle, draw,
inner sep=0pt, minimum width=2pt]
\node[] (A') at (2,-3) {};

\node[] (B') at (1.5,-2.5) {};
\node[] (C') at (2,-2) {};
\node[] (D') at (2,-2.5) {};

\node[] (E') at (1,-2) {};
\node[] (F') at (1.5,-2) {};

\path [] (A') edge (B');
\path [draw=red, very thick] (D') edge (C');
\path [] (A') edge (D');
\path [] (B') edge (E');
\path [] (B') edge (F');
\end{tikzpicture}  \ \ \ \ \ \ \ \ \
\bfig
\morphism(-1000,0)|a|/@{>}@<0pt>/<700,0>[\cdot`\cdot;\rho\left(U_{<},V_{>} s(F)\right)]
\morphism(-1000,-400)|b|/@{>}@<0pt>/<700,0>[\cdot`\cdot;\rho\left(U_{<},V_{>} t(F)\right)]
\morphism(-300,0)|r|/@{>}@<0pt>/<0,-400>[\cdot`\cdot;\rho^{\ast}_{\tau}\left(\partial_{\tau}U_{<},\partial_{\tau}V_{>} F\right)]
\morphism(-1000,0)|l|/@{>}@<0pt>/<0,-400>[\cdot `\cdot;\rho^{\ast}_{\sigma}\left(\partial_{\sigma}U_{<},\partial_{\sigma}V_{>} F\right)]

\morphism(-600,-150)|l|/@{=>}@<0pt>/<-100,-100>[  ` ; ]
\efig
\end{equation}
Here, $\partial_{\epsilon}\circ \rho_{\epsilon}$ is the homogeneous-globular factorization of $\rho\circ \epsilon$ for $\epsilon=\sigma,\tau$, $F$ is the 2-cell that fills the 1-cylinder corresponding to the image via the functor $\cyl(\cdot)$ of the inclusion $D_1 \rightarrow A$ of the copy of $D_1$ over which we have added a new vertex, and $\partial_{\epsilon} W$ denotes, given a map $W\colon A\rightarrow X$, the precomposition of $W$ with $\partial_{\epsilon}\colon \partial A \rightarrow A$. Finally, $\rho_{\epsilon}^{\ast}$ is obtained as an extension of the following form:
\[
\bfig 
\morphism(0,0)|a|/@{>}@<0pt>/<700,0>[S^1`\partial A^+;(\partial_{\sigma}\circ \rho_{\epsilon},\partial_{\tau}\circ \rho_{\epsilon})]
\morphism(0,0)|b|/@{>}@<0pt>/<0,-400>[S^1`D_2;]
\morphism(0,-400)|b|/@{>}@<0pt>/<700,400>[D_2`\partial A^+;\rho_{\epsilon}^{\ast}]
\efig
\]where $\partial A^+$ is obtained from $\partial A$ by adjoining a new vertex in the same position as the one that was added to $A$ in order to get $B$.

If the new vertex $*_B$ is not the only one over the vertex $z$ to which the new edge has been added to, then we have to distinguish according to the order of the set of vertices over $z$. If $*_B$ is the maximal element, and it has been added to $\Sigma \alpha_q$, then the 1-cylinder we get has degenerate source, and can be depicted as follows:
\begin{equation}
\label{cyl 5}
\begin{tikzpicture}
\tikzstyle{every node}=[circle, draw,
inner sep=0pt, minimum width=2pt]
\node[] (a) at (6,-3) {};

\node[] (b) at (5.5,-2.5) {};
\node[] (c) at (6,-2) {};
\node[] (d) at (6.5,-2.5) {};

\node[] (e) at (5,-2) {};
\node[] (f) at (5.5,-2) {};

\path [] (a) edge (b);
\path [draw=red, very thick] (b) edge (c);
\path [] (a) edge (d);
\path [] (b) edge (e);
\path [] (b) edge (f);
\end{tikzpicture} \ \ \ \ \ \ \ \ \
\bfig
\morphism(-1000,0)|a|/@{>}@<0pt>/<700,0>[\cdot`\cdot;\rho\left(U_{<q},a U_q, V_{>q} \right)]
\morphism(-1000,-400)|b|/@{>}@<0pt>/<700,0>[\cdot`\cdot;\rho\left(U_{<q},\alpha U_q, V_{>q} \right)]
\morphism(-300,0)|r|/@{>}@<0pt>/<0,-400>[\cdot`\cdot;\rho^{\ast}_{\tau}\left(\partial_{\tau}U_{<q},\alpha,\partial_{\tau}V_{>q} \right)]
\morphism(-1000,0)|l|/@{=}@<0pt>/<0,-400>[\cdot `\cdot;degenerate ]

\morphism(-600,-150)|l|/@{=>}@<0pt>/<-100,-100>[  ` ; ]
\efig
\end{equation}
Here, we have denoted with $\alpha$ the 2-cell of the 1-cylinder $C_{\vert \Sigma \alpha_q}\circ \cyl(\partial_{\tau})$, and with $a$ its target 0-cylinder (viewed as a 1-cell).
Dually, if it is the minimal element, the 1-cylinder has degenerate target, and is of the following form:
\begin{equation}
\label{cyl 6}
\begin{tikzpicture}
\tikzstyle{every node}=[circle, draw,
inner sep=0pt, minimum width=2pt]
\node[] (a) at (14,-3) {};

\node[] (b) at (13.5,-2.5) {};
\node[] (c) at (14,-2) {};
\node[] (d) at (14.5,-2.5) {};

\node[] (e) at (13,-2) {};
\node[] (f) at (13.5,-2) {};

\path [] (a) edge (b);
\path [draw=red, very thick] (b) edge (e);
\path [] (a) edge (d);
\path [] (b) edge (f);
\path [] (b) edge (c);
\end{tikzpicture} \ \ \ \ \ \ \ \ \
\bfig
\morphism(-1000,0)|a|/@{>}@<0pt>/<700,0>[\cdot`\cdot;\rho\left(U_{<q},V_q \alpha, V_{>q} \right)]
\morphism(-1000,-400)|b|/@{>}@<0pt>/<700,0>[\cdot`\cdot;\rho\left(U_{<q}, V_q a, V_{>q} \right)]
\morphism(-300,0)|r|/@{=}@<0pt>/<0,-400>[\cdot`\cdot;degenerate]
\morphism(-1000,0)|l|/@{>}@<0pt>/<0,-400>[\cdot `\cdot; \rho^{\ast}_{\sigma}\left(\partial_{\sigma}U_{<q},\alpha,\partial_{\sigma}V_{>q} \right)]

\morphism(-600,-150)|l|/@{=>}@<0pt>/<-100,-100>[  ` ; ]
\efig
\end{equation}
Here, we have denoted with $\alpha$ the 2-cell of the 1-cylinder $C_{\vert \Sigma \alpha_q}\circ \cyl(\partial_{\sigma})$, and with $a$ its source 0-cylinder (viewed as a 1-cell).
Finally, if the new vertex has been added as the $r$-th element over the vertex the new edge is attached to, then we get sub-globular sum of $\Sigma \alpha_q\cong \Sigma D_1^{\otimes m}$ of the form $\Sigma D_1^{\otimes r}$ and $\Sigma D_1^{\otimes m-r}$. Corresponding to this subdivision we have a $\Sigma D_1^{\otimes r}$-shaped diagram in $X$ induced by $U$, that we denote with $U_q^{\leq r}$, and, similarly, a $\Sigma D_1^{\otimes m-r}$-shaped diagram induced by $V$, that we denote with $V_q^{\geq r}$.
The corresponding 1-cylinder is essentially a 2-cell in $X(x,y)$, since its source and target are degenerate, as depicted here below:
\begin{equation}
\label{cyl 7}
\begin{tikzpicture}
\tikzstyle{every node}=[circle, draw,
inner sep=0pt, minimum width=2pt]
\node[] (a) at (10,-3) {};

\node[] (b) at (9.5,-2.5) {};
\node[] (c) at (10,-2) {};
\node[] (d) at (10.5,-2.5) {};

\node[] (e) at (9,-2) {};
\node[] (f) at (9.5,-2) {};

\path [] (a) edge (b);
\path [draw=red, very thick] (b) edge (f);
\path [] (a) edge (d);
\path [] (b) edge (e);
\path [] (b) edge (c);
\end{tikzpicture} \ \ \ \ \ \ \ \ \
\bfig
\morphism(-1000,0)|a|/@{>}@<0pt>/<1000,0>[\cdot`\cdot;\rho\left(U_{<q}, {U_q}^{\leq r},{V_q}^{\geq r}\alpha, V_{>q} \right)]
\morphism(-1000,-400)|b|/@{>}@<0pt>/<1000,0>[\cdot`\cdot;\rho\left(U_{<q},\alpha {U_q}^{\leq r},{V_q}^{\geq r}, V_{>q} \right)]
\morphism(0,0)|r|/@{=}@<0pt>/<0,-400>[\cdot`\cdot;degenerate]
\morphism(-1000,0)|l|/@{=}@<0pt>/<0,-400>[\cdot `\cdot;\ \ \ \ \ \ \ \ degenerate]

\morphism(-450,-150)|l|/@{=>}@<0pt>/<-100,-100>[  ` ; ]
\efig 
\end{equation}
Here, we have denoted with $\alpha$ the 2-cell of the 1-cylinder given by the target of the $r$-th 2-cylinder in the image of $C_{\vert \Sigma \alpha_q}$.

The last case is that of a globular sum $B\in \mathscr{L}(A)$ in which the new vertex $*_B$ has been added to $A$ at height 3. Say the 2-cell the new edge has been attached to is the $r$-th in $\Sigma \alpha_q\cong \Sigma D_1^{\otimes m}$, then the associated 1-cylinder has degenerate source and target, and is of the following form:
\begin{equation}
\label{cyl 8}
\begin{tikzpicture}
\tikzstyle{every node}=[circle, draw,
inner sep=0pt, minimum width=2pt]
\node[] (a'') at (8,-3) {};

\node[] (b'') at (7.5,-2.5) {};
\node[] (c'') at (8,-2) {};
\node[] (d'') at (8.5,-2.5) {};
\node[] (f'') at (7,-2) {};
\node[] (e'') at (8,-1.5) {};

\path [] (a'') edge (b'');
\path [draw=red, very thick] (c'') edge (e'');
\path [] (a'') edge (d'');
\path [] (b'') edge (c'');
\path [] (b'') edge (f'');
\end{tikzpicture}
\bfig
\morphism(-1000,0)|a|/@{>}@<0pt>/<1000,0>[\cdot`\cdot;\rho\left(U_{<q}, {U_q}^{< r},s(F),{V_q}^{> r}a, V_{>q} \right)]
\morphism(-1000,-400)|b|/@{>}@<0pt>/<1000,0>[\cdot`\cdot;\rho\left(U_{<q}, {U_q}^{< r},t(F),{V_q}^{> r}a, V_{>q} \right)]
\morphism(0,0)|r|/@{=}@<0pt>/<0,-400>[\cdot`\cdot;degenerate]
\morphism(-1000,0)|l|/@{=}@<0pt>/<0,-400>[\cdot `\cdot;\ \ \ \ \ \ \ \ degenerate]

\morphism(-450,-150)|l|/@{=>}@<0pt>/<-100,-100>[  ` ; ]
\efig 
\end{equation}
Here, $F$ denotes the 3-cell of the 2-cylinder in $X$, whose 0-dimensional source we denoted by $a$, picked out by precomposing $C$ with $\cyl(D_2 \rightarrow A)$, where the copy of $D_2$ in question is the one that corresponds to the vertex in $A$ of height 2 over which $*_B$ has been added. 

So far, we have described a stack of $\vert \mathscr{L}(A)\vert$ vertically composable (possibly degenerate) 1-cylinders in $X(x,y)$. Its (vertical) composite is a 1-cylinder $C_t \rho(U)\curvearrowright \rho(V) C_s$ in $X(x,y)$, that transposes under the adjunction $\Sigma \dashv \Omega$ to give the desired 2-cylinder $C\circ\hat{\rho}\colon \rho(U) \curvearrowright \rho(V)$.
\subsection{The Division Lemma}
The proof of the crucial fact that $\mathbf{ev}_i=\pi_i \circ \mathbf{ev}$ is a trivial fibration for $i=0,1$ (where $\pi_i$ denotes the product projection onto the $i$-th factor) relies on the division lemma, i.e. the fact that given a pair of parallel $n$-cells $A,B$ and a $1$-cell $f$ in an $\infty$-groupoid $X$, any $(n+1)$-cell $H\colon fA \rightarrow fB$ (where juxtaposition denotes the choice of a whiskering operation) is homotopic to one of the form $f\overline{H}$ for some $\overline{H}\colon A \rightarrow B$ (this is essentially the content of Lemma 4.12 in \cite{AR2}). 

The proof of this lemma requires contractibility, and we were not able to generalize it to $\mathfrak{C}^{\mathbf{W}}$ (as defined in Definition \ref{D_W def}) in the case where $\mathfrak{C}$ is a coherator for $\infty$-categories. The three dimensional case can still be proven by hands, as follows. Note that, in the presence of both a left and a right inverse for every cell, any of them can be promoted to a two-sided inverse, therefore we will use the notation $f^{-1}$ with no reference to left or right.

If $n=1$ and we have a 2-cell in $X$ of the form:
\[\bfig 
\morphism(0,0)|a|/{@{>}@/^1em/}/<400,0>[a`c;fA]
\morphism(0,0)|r|/{@{>}@/^-1em/}/<400,0>[a`c;fB]

\morphism(200,100)|a|/@{=>}/<0,-150>[`;H]
\efig \] for $A,B\colon a \rightarrow b$ and $f\colon b\rightarrow c$, then we can define $\overline{H}$ as the following composite:
\[
\bfig
\morphism(0,0)|a|/@{>}@<0pt>/<400,0>[A`f^{-1}fA;\simeq]
\morphism(400,0)|a|/@{>}@<0pt>/<600,0>[f^{-1}fA`f^{-1}fB;f^{-1}H]
\morphism(1000,0)|a|/@{>}@<0pt>/<400,0>[f^{-1}fB`B;\simeq]
\efig 
\] where ``$\simeq$'' denotes coherence constraints that exist in $\mathfrak{C}^{\mathbf{W}}$. It is a routine exercise to check that $f\overline{H}$ is homotopic to $H$.

Turning to $n=2$, we assume we have a 3-cell $H\colon fA\rightarrow fB$. We define $\overline{H}\colon A \rightarrow B$ as the following composite of 3-cells:
\[
\bfig
\morphism(0,0)|a|/{@{>}@/^1em/}/<400,0>[a`b;s(A)]
\morphism(0,0)|r|/{@{>}@/^-1em/}/<400,0>[a`b;t(A)]

\morphism(200,100)|a|/@{=>}/<0,-150>[`;A]
\efig
\Rrightarrow
\bfig
\morphism(0,0)|a|/{@{>}@/^3.5em/}/<800,0>[a`b;s(A)]
\morphism(0,0)|r|/{@{>}@/^-3.5em/}/<800,0>[a`b;t(A)]

\morphism(0,0)|a|/{@{>}@/^1em/}/<350,0>[a`b;]
\morphism(0,0)|r|/{@{>}@/^-1em/}/<350,0>[a`b;]
\morphism(350,0)|a|/@{>}@<0pt>/<170,0>[b`c;f]
\morphism(520,0)|a|/@{>}@<0pt>/<280,0>[c`b;f^{-1}]
\morphism(400,300)|l|/@{=>}/<0,-150>[`;\simeq]
\morphism(400,-100)|l|/@{=>}/<0,-150>[`;\simeq]
\morphism(175,100)|l|/@{=>}/<0,-175>[`;A]
\efig 
\Rrightarrow
\bfig
\morphism(0,0)|a|/{@{>}@/^3.5em/}/<800,0>[a`b;s(B)]
\morphism(0,0)|r|/{@{>}@/^-3.5em/}/<800,0>[a`b;t(B)]

\morphism(0,0)|a|/{@{>}@/^1em/}/<350,0>[a`b;]
\morphism(0,0)|r|/{@{>}@/^-1em/}/<350,0>[a`b;]
\morphism(350,0)|a|/@{>}@<0pt>/<170,0>[b`c;f]
\morphism(520,0)|a|/@{>}@<0pt>/<280,0>[c`b;f^{-1}]
\morphism(400,300)|l|/@{=>}/<0,-150>[`;\simeq]
\morphism(400,-100)|l|/@{=>}/<0,-150>[`;\simeq]
\morphism(175,100)|l|/@{=>}/<0,-175>[`;B]
\efig 
\Rrightarrow
\bfig
\morphism(0,0)|a|/{@{>}@/^1em/}/<400,0>[a`b;s(B)]
\morphism(0,0)|r|/{@{>}@/^-1em/}/<400,0>[a`b;t(B)]

\morphism(200,100)|a|/@{=>}/<0,-150>[`;B]
\efig
\] Here, the 2-cells denoted with ``$\simeq$'' denote coherence constraints that exist in $\mathfrak{C}^{\mathbf{W}}$, and the first and last 3-cells are also composite of constraints, whereas the one in the middle is a whiskering of $H$ with the other cells depicted there. Again, it is a tedious but straightforward exercise to check that $f\overline{H}$ is homotopic (i.e. equal,for dimensionality reasons) to $\overline{H}$.

Finally, if $n=3$ we have to prove that $fA=fB$ implies $A=B$, which is entirely analogous to the arguments given so far.
\section{A path object in $\mathbf{Mod}({\mathfrak{C}_3}^{\mathbf{W}})$}
\label{S}
Given a coherator for 3-categories $\mathfrak{C}_3$, we are going to endow the globular set 
\[ (\p X)_k = \mathbf{Mod}({\mathfrak{C}_3}^{\mathbf{W}})\left(\cyl(D_k),X\right) \] with the structure of a $\mathfrak{C}_3$-model, which we can then extend to a ${\mathfrak{C}_3}^{\mathbf{W}}$-model thanks to the result of Section 6 in \cite{EL}, thus providing a proof of Proposition \ref{path object for Batanin 3-groupoids}. From now on we will drop the subscript and simply denote this coherator by $\mathfrak{C}$.

It follows from \ref{fact of maps of gpds} that in the cellularity condition for a coherator for $n$-categories $\mathfrak{C}$, we can assume that all the ``basic'' operations of dimension $k$ are added at the $k$-th step of the tower that defines $\mathfrak{C}$. More precisely, we can assume that $\mathfrak{C}_{k+1}=\mathfrak{C}_{k}[X_k]$ with $X_k=\{(h_1,h_2)\colon D_k \rightarrow A\}$ for every $0\leq k \leq n$. In particular, if $n=3$, we can assume without loss of generality that $\mathfrak{C}\cong\mathfrak{C}_4$ fits into the diagram displayed below:
\begin{equation}
\label{tower for D}
\bfig 
\morphism(0,0)|a|/@{>}@<0pt>/<600,0>[ \Theta^{\leq 3}_0` \mathfrak{C}_1\cong \Theta^{\leq 3}_0 [X_0] ;i_0]
\morphism(600,0)|a|/@{>}@<0pt>/<800,0>[ \mathfrak{C}_1\cong \Theta^{\leq 3}_0 [X_0] ` \mathfrak{C}_2\cong \mathfrak{C}_1 [X_1] ;i_1]
\morphism(1400,0)|a|/@{>}@<0pt>/<800,0>[\mathfrak{C}_2\cong \mathfrak{C}_1 [X_1]` \mathfrak{C}_3\cong \mathfrak{C}_2 [X_2];i_2]
\morphism(2200,0)|a|/@{>}@<0pt>/<800,0>[\mathfrak{C}_3\cong \mathfrak{C}_2 [X_2]` \mathfrak{C}_4\cong \mathfrak{C}_3 [X_3];i_3]
\efig
\end{equation} 
We can adapt the argument used in Proposition 11.4 of \cite{EL} to find a lifting of the functor $\p$ as depicted below
\[\bfig
\morphism(0,-400)|a|/@{-->}@<0pt>/<600,400>[ \mathbf{Mod}(\mathfrak{C}^{\mathbf{W}})`\mathbf{Mod}(\mathfrak{C}_2);\p]
\morphism(0,-400)|a|/@{>}@<0pt>/<600,0>[\mathbf{Mod}(\mathfrak{C}^{\mathbf{W}})`[\G_3^{op},\mathbf{Set}];\p]
\morphism(600,0)|r|/@{>}@<0pt>/<0,-400>[\mathbf{Mod}(\mathfrak{C}_2)`[\G_3^{op},\mathbf{Set}];\mathbf{U}]
\efig
\]
It turns out that extending along $i_3$ is automatic, thanks to the following result.
\begin{lemma}
Suppose given a pair of $n$-cylinders $(F,G)$ in an $\mathfrak{C}^{\mathbf{W}}$-model $X$, such that $s(F)=s(G)$, $t(F)=t(G)$ and $F_0=G_0=A$, $F_1=G_1=B$. Then $F=G$.
\begin{proof}
If $n=0$ the result is clear. Assume $n>0$, then we get $(n-1)$-cylinders $\overline{F},\overline{G}$ in $ X\left(s(f),t(g)\right)$ where $f=F_s=G_s$ and $g=F_t=G_t$. By definition, we have $\overline{F},\overline{G}\colon gA \curvearrowright Bf$ and $\epsilon(\overline{F})=\epsilon(\overline{G})$ for $\epsilon=s,t$. Therefore, by inductive assumption we get that $\overline{F}=\overline{G}$, which concludes the proof.
\end{proof}
\end{lemma}
We can now apply this lemma to the situation where we have a pair of parallel operations $\alpha,\beta \colon D_3 \rightarrow A$ in $X_3$, so that $\alpha=\beta$ in $\mathfrak{C}$, and interpretations $\cyl(\alpha),\cyl(\beta)\colon\cyl(D_3) \rightarrow \cyl(A)$ which are compatible with the map $\iota \colon\D_3 \coprod D_3 \rightarrow \cyl(D_3)$. We want to prove that $\cyl(\alpha)=\cyl(\beta)$, and we do so representably. Given $H\colon \cyl(A) \rightarrow Y$, with $Y \in \mathbf{Mod}(\mathfrak{C}^{\mathbf{W}})$, we see that \[(H\circ \cyl(\alpha))_{\epsilon}=H_{\epsilon}\circ \alpha=H_{\epsilon}\circ \beta=(H\circ \cyl(\beta))_{\epsilon}\] for $\epsilon=0,1$. Moreover, $s(H\circ \cyl(\alpha))=H \circ \cyl(\alpha\circ \sigma)=H \circ \cyl(\beta\circ \sigma)=s(H\circ \cyl(\beta))$, and similarly for the target. This implies $\cyl(\alpha)=\cyl(\beta)$.

We are now left with the problem of finding a lift of the form:
\begin{equation}
\label{extension to D3}
\bfig
\morphism(0,-400)|a|/@{-->}@<0pt>/<600,400>[ \mathbf{Mod}(\mathfrak{C}^{\mathbf{W}})`\mathbf{Mod}(\mathfrak{C}_3);\p]
\morphism(0,-400)|a|/@{>}@<0pt>/<600,0>[\mathbf{Mod}(\mathfrak{C}^{\mathbf{W}})`\mathbf{Mod}(\mathfrak{C}_2);\p]
\morphism(600,0)|r|/@{>}@<0pt>/<0,-400>[\mathbf{Mod}(\mathfrak{C}_3)`\mathbf{Mod}(\mathfrak{C}_2);\mathbf{U}]
\efig
\end{equation}
 and this, in turn, amounts to defining a map $\cyl(\rho)\colon \cyl(D_3)\rightarrow \cyl(A)$ for every $\rho \colon D_3 \rightarrow A$ added as a filler of a pair $(h_1,h_2) \in X_2$, in such a way that $\cyl(\rho)\circ \cyl(\sigma)=\cyl(h_1)$ and $\cyl(\rho)\circ \cyl(\tau)=\cyl(h_2)$. Note that these last 2 equations make sense, since $h_1,h_2 \in \mathfrak{C}_2$.

The strategy for constructing such maps will be the same as the one used to get the extension to $\mathfrak{C}_2$, namely to prove that we can endow every interpretation of a 2-dimensional operation $\cyl(\phi)\colon \cyl(D_2)\rightarrow \cyl(A)$ with a modification
\begin{equation}
\label{(1)}
\Theta_{\phi}\colon \widehat{\phi} \Rightarrow \cyl(\phi)
\end{equation}
in a way that is compatible with source and target (as will be explained in more detail later on), so that we can then use the following lemma to produce the map we are after.
\begin{lemma}
	\label{extension of modifications}
	Given a $\mathfrak{C}^{\mathbf{W}}$-model $X$, an $m$-cylinder $C\colon A \curvearrowright B$ in $X$, a pair of parallel $(m-1)$-cylinders $D_{s},D_t\colon\cyl_n(D_{m-1}) \rightarrow X$ and parallel modifications $\Theta_1 \colon s(C)\Rightarrow D_s, \ \Theta_2\colon t(C)\Rightarrow D_t$ there exists an $m$-cylinder $D\colon \cyl_n(D_m)\rightarrow X$ such that $s(D)=D_s, \ t(D)=D_t$ and a modification $\Theta \colon C \Rightarrow D$ such that $s(\Theta)=\Theta_1$ and $t(\Theta)=\Theta_2$.
\begin{proof}
The proof is just a word-by-word copy and paste of that of Lemma 10.5 in \cite{EL}, promoting left or right inverses to two-sided ones when necessary.
\end{proof}
\end{lemma}
In fact, we can apply this lemma to the situation depicted in the diagram below, thus getting the desired extension:
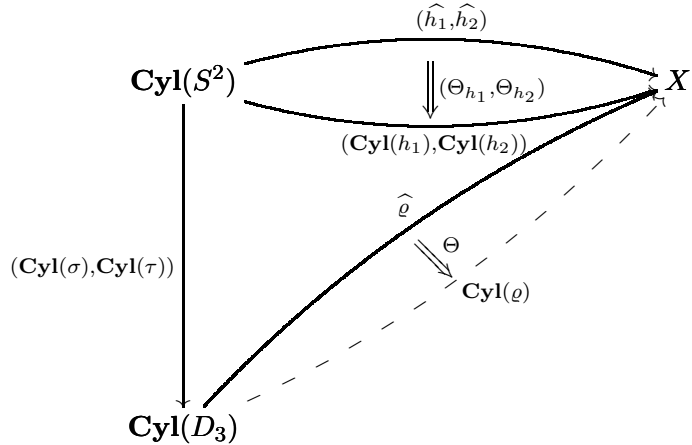
\[
\bfig
\morphism(0,0)|a|/{@{>}@/^1.5em/}/<1700,0>[\cyl(S^{2})`X;(\widehat{h_1},\widehat{h_2})]
\morphism(0,0)|b|/{@{>}@/^-1.5em/}/<1700,0>[\cyl(S^{2})`X;(\cyl(h_1), \cyl(h_2))\ \ \ \ ]
\morphism(850,125)|r|/@{=>}@<0pt>/<0,-250>[`; (\Theta_{h_1},\Theta_{h_2})]
\morphism(0,0)|a|/{@{>}@//}/<0,-1200>[\cyl(S^{2})`\cyl(D_3);(\cyl(\sigma),\cyl(\tau))]
\morphism(0,-1200)|a|/{@{>}@/^1.2em/}/<1700,1200>[\cyl(D_3)`X;\widehat{\rho}]
\morphism(0,-1200)|b|/{@{-->}@/^-1.2em/}/<1700,1200>[\cyl(D_3)`X;\cyl(\rho)]
\morphism(775,-525)|r|/@{=>}@<0pt>/<150,-150>[`; \Theta]
\efig
\]
We are going to prove several lemmas to obtain the modification in \eqref{(1)}. To simplify some arguments, we will sometimes assume (without loss of generality) that our computations happen in dimension $n=\infty$, thus replacing identities with appropriate cells. The result we are looking for can then be obtained by simply quotienting out these higher cells. Also, all the results that hold true in $\mathbf{Mod}(\mathfrak{C})$ will be proven in that context using the techniques illustrated in Section \ref{main constructions}. The case of $\mathfrak{C}^{\mathbf{W}}$-models then follows, as usual, by applying the free functor $F\colon \mathbf{Mod}(\mathfrak{C}) \rightarrow \mathbf{Mod}(\mathfrak{C}^{\mathbf{W}})$, that is easily seen to preserve cylinders, vertical composites of them and modifications. In the previous section we recalled the salient features of the construction of the map $\hat{\rho}\colon \cyl (D_k)\rightarrow \cyl(A)$ associated with a homogeneous map $\rho \colon D_k \rightarrow A$ in $\mathfrak{C}$. We refer the reader to Definition 9.15 of \cite{EL} for a fully detailed version. 

We start with a lemma that allows us to ``plug'' modifications of globular sums of cylinders into the elementary interpretation of a 2-dimensional operation.
\begin{lemma}
	\label{(3)}
Assume given a homogeneous operation $\rho \colon D_2 \rightarrow A$ in $\mathfrak{C}$, a $\mathfrak{C}$-model X and a pair of cylinders $C,D\colon\cyl(A)\rightarrow X$ that agree on the 0-cells of $A$ (i.e. each inclusion $\cyl\left( D_0\rightarrow A\right)$ equalizes these maps), with $C,D\colon U \curvearrowright V$. Given a modification $\Theta\colon C \Rightarrow D$ such that for each globular map $D_1\rightarrow A$, the induced modification $\bfig  \morphism(0,0)|b|/@{>}@<0pt>/<300,0>[\mathbf{M}_1` \mathbf{M}_A;]  \morphism(300,0)|a|/@{>}@<0pt>/<300,0>[\mathbf{M}_A` X;\Theta] \efig $ is essentially a 3-cell, we get an induced modification of the form $\Theta \circ \hat{\rho} \colon C\circ \hat{\rho} \Rightarrow D \circ \hat{\rho}$.
\begin{proof}
The proof is structured in the following manner: since both cylinders $C \circ \hat{\rho} $ and $ D \circ \hat{\rho}$ are built as the vertical composite of a stack of cylinders, we will construct compatible modifications from each of the cylinders that compose the stack associated with $C \circ \hat{\rho} $ towards the corresponding ones in the stack associated with $D \circ \hat{\rho} $. We will then conclude by using the bicategorical structure described in part B of the appendix to compose up these modifications thus getting the desired map $\Theta\circ \hat{\rho}$.
Let \[\begin{pmatrix}
i_1 &&i_2 & \ldots&i_{m-1} & &i_m\\
& i'_1 & &\ldots&& i'_{m-1}
\end{pmatrix}\] be the table of dimensions of $A$. Since $\rho$ is homogeneous, we have $\dim(A)\leq 2$, and therefore $i_k=1,2$ for every $1\leq k \leq m$. By precomposing with the appropriate colimit inclusions we thus get cylinders $C_k,D_k\colon \cyl(D_{i_k}) \rightarrow X$. 

The cylinders associated with case \eqref{cyl 1} to \eqref{cyl 3} in both stacks coincide thanks to the assumptions, thus we can use identity modifications in these cases.
We now consider case \eqref{cyl 4}: i.e globular sums $B\in \mathscr{L}(A)$ in which we added a new vertex $*_B$ to $A$ at height $\height(*_B)=2$, in such a way that this new vertex is the unique element of the fiber over the vertex of height one that is right below it (this is  in the explicit description of cylinders outlined in the previous section). Fix a globular sum $B$ in this family, such that the vertex $*_B$ has been added to $A$ over $D_{i_r}=D_1$, and consider the vertical stacks of 1-cylinders whose composites are the transpose of $C\circ \hat{\rho}$ and $ D \circ \hat{\rho}$ respectively. The 2-cell in $B$ corresponding to the vertex $*_B$ picks out the 2-cell associated with the 1-cylinder $C_r$ (resp.$D_r$) via the composites 
\[\
\bfig 
\morphism(0,0)|a|/@{>}/<400,0>[B`\cyl(A);i_B]
 \morphism(400,0)|a|/@{>}/<400,0>[\cyl(A)`X;C]
 
 \morphism(1000,0)|a|/@{>}/<400,0>[B`\cyl(A);i_B]
 \morphism(1400,0)|a|/@{>}/<400,0>[\cyl(A)`X;D]
\efig 
\]
We can use the components of $\Theta$ to construct the following boundary of a 1-modification in $B(x,y)$ (for $x=s^{i_1}(C_1)_0, \ y=t^{i_m}(C_m)_1$), where the 1-cylinders $\Gamma_B$ and $\Delta_B$ are the ones associated with the globular sum $B$ in the two stacks, as follows (using the notation established in the previous section):
\[
\bfig 
\morphism(-1000,0)|a|/@{>}@<0pt>/<2300,0>[a`b;\rho\left(U_{< r},V_{> r} s(C_r)\right)]
\morphism(-1000,-700)|b|/@{>}@<0pt>/<2300,0>[a'`b';\rho\left(U_{< r},V_{> r} t(C_r)\right)]
\morphism(1300,0)|r|/{@{>}@/^3em/}/<0,-700>[b`b';\rho^{\ast}_{\tau}\left(U^{\tau}_{<r},V^{\tau}_{>r} \overline{C_r}\right)]
\morphism(-1000,0)|l|/{@{>}@/^-3em/}/<0,-700>[a`a';\rho^{\ast}_{\sigma}\left(U^{\sigma}_{<r},V^{\sigma}_{>r} \overline{C_r}\right)]
\morphism(1300,0)|l|/{@{>}@/^-3em/}/<0,-700>[b`b';\rho^{\ast}_{\tau}\left(U^{\tau}_{<r},V^{\tau}_{>r} \overline{D_r}\right)]
\morphism(-1000,0)|r|/{@{>}@/^3em/}/<0,-700>[a`a';\rho^{\ast}_{\sigma}\left(U^{\sigma}_{< r},V^{\sigma}_{>r} \overline{D_r}\right)]
\morphism(-900,-350)|a|/@{=>}/<-200,0>[ ` ;U^{\sigma}_{<r},V^{\sigma}_{>r} \overline{\Theta_r}]
\morphism(1400,-350)|a|/@{=>}/<-200,0>[` ;U^{\tau}_{<r},V^{\tau}_{>r} \overline{\Theta_r}]

\morphism(250,-350)|a|/{@{=>}@/^-1em/}/<-200,0>[ ` ;\Gamma_B]
\morphism(250,-350)|b|/{@{=>}@/^1em/}/<-200,0>[` ;\Delta_B]
\efig 
\] Here, we have committed a minor abuse of language in denoting by $U^{\sigma}$ what we normally denote with $\partial_{\sigma} U$, and with $U^{\sigma}_{<r},V^{\sigma}_{>r} \overline{\Theta_r}$ the result of composing that pasting diagram with a chosen operation whose boundary is given by $\left(\partial_{\sigma}
,\partial_{\tau} \right)\circ \rho^{\ast}_{\sigma}$, and similarly for the analogues with $\tau$.
	We can now use the fact that a filler certainly exists in $\nCat{\omega}$ to extend this to a modification of 1-cylinders, and this concludes the construction for the first case.

Let us now address the case of globular sums $B\in \mathscr{L}(A)$ of case \eqref{cyl 5} to \eqref{cyl 8} that appear consecutively in $\mathcal{L}(A)$. We will build a modification involving the sub-stack associated with this subset of $\mathcal{L}(A)$, all at once rather than cylinder by cylinder.
Let $A\cong \Sigma \alpha_1\plus{D_0}\ldots \plus{D_0}\Sigma \alpha_p$ be the decomposition of $A$. We can consider the maximal sub-globular sum of $A$ of the form $D_2\plus{D_1}\ldots\plus{D_1}D_2\cong \Sigma D_1^{\otimes k}\cong \Sigma \alpha_q$ for some $q$ that contains the copy of $D_2$ to which the new edges have been adjoined. The globular inclusion $\Sigma D_1^{\otimes k}\rightarrow A$ picks out $k$ composable 2-cylinders $\Gamma_1,\ldots, \Gamma_k $ via $C$ and $\Delta_1,\ldots,\Delta_k$ via $D$. Notice that there exists an integer $r$ such that $\Gamma_i=C_{r+i}$, and the same holds if we replace $C$ and $\Gamma$ with $D$ and $\Delta$, with the same $r$. Consider the vertical stacks of 1-cylinders whose composites are the transpose of $C\circ \hat{\rho}$ and $ D \circ \hat{\rho}$ respectively. 
%
The sub-stack associated with the globular sums in the ordered set $\mathscr{L}(A)$ comprised between the one in which the new edge has been added at the far right of the corolla represented by $\Sigma D_1^{\otimes k}$ and the one in which it has been added at the far left is mapped in $X(x,y)$ under $C$ to a pasting diagram of the form:
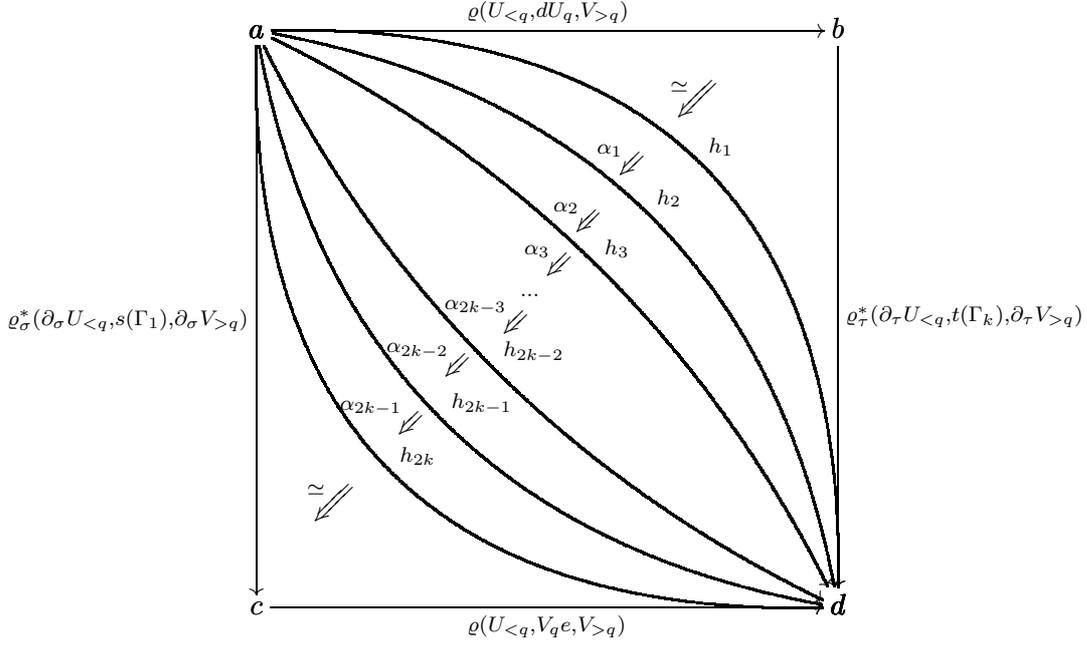
\[
\bfig 
\morphism(0,0)|a|/@{>}/<2000,0>[a`b;\rho\left(U_{<q},d U_q, V_{>q} \right)]	
\morphism(0,0)|a|/@{>}/<0,-2000>[a`c;\rho^{\ast}_{\sigma}\left(\partial_{\sigma}U_{<q},s(\Gamma_1),\partial_{\sigma}V_{>q} \right) ]
\morphism(0,-2000)|b|/@{>}/<2000,0>[c`d;\rho\left(U_{<q}, V_q e, V_{>q} \right)]
\morphism(2000,0)|r|/@{>}/<0,-2000>[b`d;\rho^{\ast}_{\tau}\left(\partial_{\tau}U_{<q},t(\Gamma_k),\partial_{\tau}V_{>q} \right)]

\morphism(0,0)|a|/{@{>}@/^7.5em/}/<2000,-2000>[a`d;h_1]
\morphism(0,0)|a|/{@{>}@/^5em/}/<2000,-2000>[a`d;h_2]
\morphism(0,0)|a|/{@{>}@/^2.5em/}/<2000,-2000>[a`d;h_3]
\morphism(0,0)|r|/{@{>}@/^-2.5em/}/<2000,-2000>[a`d;h_{2k-2}]
\morphism(0,0)|r|/{@{>}@/^-7.5em/}/<2000,-2000>[a`d;h_{2k}]
\morphism(0,0)|r|/{@{>}@/^-5em/}/<2000,-2000>[a`d; h_{2k-1}]
\morphism(1600,-150)|a|/@{=>}/<-150,-150>[ ` ;\simeq]
\morphism(1350,-400)|a|/@{=>}/<-100,-100>[ ` ;\alpha_1]
\morphism(1200,-600)|a|/@{=>}/<-100,-100>[ ` ;\alpha_2]
\morphism(1100,-750)|a|/@{=>}/<-100,-100>[ ` ;\alpha_3]
\morphism(1000,-950)|a|/@{}/<0,0>[ ` ;\ldots]
\morphism(950,-950)|a|/@{=>}/<-100,-100>[ ` ;\alpha_{2k-3}]
\morphism(750,-1100)|a|/@{=>}/<-100,-100>[ ` ;\alpha_{2k-2}]
\morphism(590,-1300)|a|/@{=>}/<-100,-100>[ ` ; \ \ \ \alpha_{2k-1}]
\morphism(350,-1550)|a|/@{=>}/<-150,-150>[ ` ;\simeq]
\efig 
\]
where $d=t^2(\Gamma_1)$, $e=s^2(\Gamma_1)$. Obviously we get a similar one replacing every occurrence of $C$ with $D$ and of $\Gamma$ with $\Delta$. Here, we have set: \[h_{2m+1}=\rho\left(U_{<q}, {U_q}^{< r},s(\underline{\Gamma_{k-m}}),{V_q}^{> r}a, V_{>q} \right)\]  \[h_{2m}=\rho\left(U_{<q}, {U_q}^{< r},t(\underline{\Gamma_{k-m}}),{V_q}^{> r}a, V_{>q} \right)\] where we have used $\underline{F}$ to denote the underlying 3-cell of a 2-cylinder $F$, and $t(\Gamma_0)$ is defined to be $s(\Gamma_1)$. The 2-cells in $X(x,y)$ labelled with $\alpha$'s represent 1-cylinders whose source and target are degenerate. In particular, each $\alpha_{2m+1}$ is a whiskering of the 3-cell in $\Gamma_{k-m}$ and each $\alpha_{2m}$ is an associativity constraint, for every $0\leq m \leq k-1$, as explained in detail at the end of the previous section.
We can use the components of $\Theta$ to find a modification between the vertical composites of these (degenerate) cylinders using the following lemma, which concludes the proof. 
\end{proof}
\end{lemma}
The following is a result that is needed in the proof of the previous lemma, but we only concern ourselves with a small simplification of it, leaving the (straightforward) proof of the generalization of the result to the interested reader. The simplification consists of restricting to the case $k=2$, following the notation established above. Nevertheless, the proof of the general case is entirely similar and has no more genuine content than the one we present.
\begin{lemma}
Assume given 2-cylinders $C,D\colon A_0 \curvearrowright B_0$ and $G,H\colon A_1 \curvearrowright B_1$ in a $\mathfrak{C}$-model $X$, with $t(C)=s(G), t(D)=s(H)$, together with 2-dimensional pasting diagrams $\epsilon\colon E\rightarrow X,\phi\colon F \rightarrow X$ with $s^2(\epsilon)=t^2(B_0), t^2(\phi)=s^2(A_0)$ and an operation $\rho \colon D_2 \rightarrow F\plus{D_0} D_2\plus{D_1}D_2\plus{ D_0}E$ in $\mathfrak{C}$. This implies that, in particular, $t(A_0)=s(A_1)$ and $t(B_0)=s(B_1)$. Also, assume given modifications $\Theta\colon C\Rightarrow D, \Theta'\colon G \Rightarrow H$, with $t(\Theta)=s(\Theta')$, whose sources and targets, denoted respectively with $S\colon s(C)\Rightarrow s(D),S'\colon s(G)\Rightarrow s(H)$ and $T\colon t(D) \Rightarrow t(C),T'\colon t(G) \Rightarrow t(H)$, are essentially represented by 3-cells (i.e. they have trivial 0-dimensional boundary). Then we get an induced modification $\epsilon\Theta'\Theta \phi$ between the vertical composite depicted below (where we have used $\underline{F}$ to denote the underlying 3-cell of a 2-cylinder $F$, $\alpha_2$ is simply an associativity constraint, and the 2-cells labelled with ``$\simeq$'' are also given by coherence constraints) and the one obtained by replacing each occurrence of $C$ with $D$ and of $G$ with $H$, with corresponding $\beta$'s in place of $\alpha$'s and $g$'s in place of $h$'s.
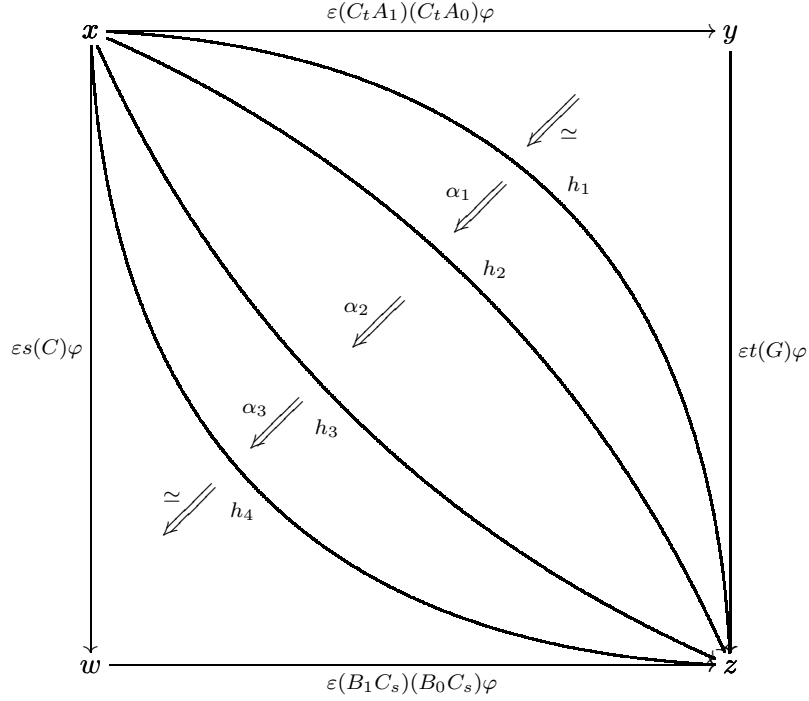
\[
\bfig 
\morphism(0,0)|a|/@{>}/<2200,0>[x`y;\epsilon(C_t A_1)(C_t A_0)\phi]	
\morphism(0,0)|a|/@{>}/<0,-2200>[x`w;\epsilon s(C)\phi]
\morphism(0,-2200)|b|/@{>}/<2200,0>[w`z;\epsilon(B_1 C_s)(B_0 C_s)\phi]
\morphism(2200,0)|r|/@{>}/<0,-2200>[y`z;\epsilon t(G) \phi]

\morphism(0,0)|a|/{@{>}@/^7em/}/<2200,-2200>[x`z;h_1]
\morphism(0,0)|l|/{@{>}@/^-7em/}/<2200,-2200>[x`z;h_4]
\morphism(0,0)|a|/{@{>}@/^3em/}/<2200,-2200>[x`z;h_2]
\morphism(0,0)|l|/{@{>}@/^-3em/}/<2200,-2200>[x`z;h_3]
\morphism(1700,-200)|r|/@{=>}/<-200,-200>[ ` ;\simeq]
\morphism(1450,-500)|a|/@{=>}/<-200,-200>[ ` ;\alpha_1]

\morphism(1100,-900)|a|/@{=>}/<-200,-200>[ ` ;\alpha_2]
\morphism(750,-1250)|a|/@{=>}/<-200,-200>[ ` ;\alpha_3]
\morphism(450,-1550)|a|/@{=>}/<-200,-200>[ ` ;\simeq]
\efig 
\] The notation is defined as follows: \begin{itemize}
\item $h_1=\epsilon(t(G)(C_t A_1))(C_t A_0)\phi$, $h_2=\epsilon((B_1C_s)s(G))(C_t A_0)\phi$, $\alpha_1= \epsilon\underline{G}(C_t A_0)\phi$
\item $h_3=\epsilon(B_1C_s)(s(G)(C_t A_0))\phi$, $h_4=\epsilon(B_1C_s)((B_0 C_s )s(C))\phi$, $\alpha_3=\epsilon(B_1 C_s)\underline{C}\phi$
\end{itemize}
where juxtaposition is either the result of composing using $\rho$ or using the composition operations that appear in the definition of cylinders, as should be clear from the context.
%
%
\begin{proof}
To begin with, we observe that the hypotheses imply $C_s=D_s$ and $C_t=D_t$, and we denote these 1-cells with $a$ and $b$ respectively. We consider the following pasting diagram in $\Omega_2(X,x,z)$ with $x,z$ being the appropriate 1-cells of $X$ depicted in the diagram above:
\[
\bfig 
\morphism(0,0)|a|/@{>}/<2000,0>[(\epsilon t(H)\phi)(\epsilon((bA_1)(bA_0))\phi)`(\epsilon t(G)\phi)(\epsilon((bA_1)(bA_0))\phi);(\epsilon T'\phi)(\epsilon((bA_1)(bA_0))\phi)]	
\morphism(0,0)|l|/@{>}/<0,-400>[(\epsilon t(H)\phi)(\epsilon((bA_1)(bA_0))\phi)`\epsilon((t(H)(bA_1))(bA_0))\phi;\simeq]	
\morphism(2000,0)|r|/@{>}/<0,-400>[(\epsilon t(G)\phi)(\epsilon((bA_1)(bA_0))\phi)`\epsilon((t(G)(bA_1))(bA_0))\phi;\simeq]	
\morphism(0,-400)|a|/@{>}/<2000,0>[\epsilon((t(H)(bA_1))(bA_0))\phi`\epsilon((t(G)(bA_1))(bA_0))\phi;\epsilon((T'(bA_1))(bA_0))\phi]
	
\morphism(2000,-400)|r|/@{>}/<0,-400>[\epsilon((t(G)(bA_1))(bA_0))\phi`\epsilon(((B_1 a)s(G))(bA_0))\phi;\epsilon(\underline{G}(bA_0))\phi]	
\morphism(2000,-800)|a|/@{>}/<-2000,0>[\epsilon(((B_1 a)s(G))(bA_0))\phi`\epsilon(((B_1 a)s(H))(bA_0))\phi;\epsilon(((B_1a)S')(bA_0))\phi]	
\morphism(0,-400)|a|/{@{>}@/^-2.5em/}/<0,-400>[\epsilon((t(H)(bA_1))(bA_0))\phi`\epsilon(((B_1 a)s(H))(bA_0))\phi;\epsilon(\underline{H}(bA_0))\phi]	
\morphism(0,-400)|r|/{@{>}@/^2.5em/}/<0,-400>[\epsilon((t(H)(bA_1))(bA_0))\phi`\epsilon(((B_1 a)s(H))(bA_0))\phi;\epsilon(s(\Theta')(bA_0))\phi]	
\morphism(0,-800)|a|/@{>}/<0,-400>[\epsilon(((B_1 a)s(H))(bA_0))\phi`\epsilon((B_1 a)(s(H)(bA_0)))\phi;\simeq]	
\morphism(2000,-800)|r|/@{>}/<0,-400>[\epsilon(((B_1 a)s(G))(bA_0))\phi`\epsilon((B_1 a)(s(G)(bA_0)))\phi; \simeq]	
\morphism(2000,-1200)|a|/@{>}/<-2000,0>[\epsilon((B_1 a)(s(G)(bA_0)))\phi`\epsilon((B_1 a)(s(H)(bA_0)))\phi; \epsilon((B_1 a)(S'(bA_0)))\phi]	
\morphism(2000,-2000)|a|/@{>}/<-2000,0>[(\epsilon((B_1 a)(B_0 a))\phi)(\epsilon s(C)\phi)`(\epsilon((B_1 a)(B_0 a))\phi)(\epsilon s(D)\phi);(\epsilon((B_1 a)(B_0 a))\phi)(\epsilon S \phi)]
\morphism(2000,-1600)|a|/@{>}/<-2000,0>[\epsilon((B_1 a)((B_0 a)s(C)))\phi`\epsilon((B_1 a)((B_0 a)s(D)))\phi; \epsilon((B_1 a)((B_0 a)S))\phi]
\morphism(2000,-1600)|r|/@{>}/<0,-400>[\epsilon((B_1 a)((B_0 a)s(C)))\phi`(\epsilon((B_1 a)(B_0 a))\phi)(\epsilon s(C)\phi); \simeq]
\morphism(0,-1600)|l|/@{>}/<0,-400>[\epsilon((B_1 a)((B_0 a)s(D)))\phi`(\epsilon((B_1 a)(B_0 a))\phi)(\epsilon s(D)\phi); \simeq]
\morphism(2000,-1200)|r|/@{>}/<0,-400>[\epsilon((B_1 a)(s(G)(bA_0)))\phi`\epsilon((B_1 a)((B_0 a)s(C)))\phi;\epsilon((B_1 a)\overline{C})\phi]	
\morphism(0,-1200)|a|/{@{>}@/^-2.5em/}/<0,-400>[\epsilon((B_1 a)(s(H)(bA_0)))\phi`\epsilon((B_1 a)((B_0 a)s(D)))\phi;\epsilon((B_1 a)(\underline{D}))\phi]	
\morphism(0,-1200)|r|/{@{>}@/^2.5em/}/<0,-400>[\epsilon((B_1 a)(s(H)(bA_0)))\phi`\epsilon((B_1 a)((B_0 a)s(D)))\phi;\epsilon((B_1 a) (s(\Theta')))\phi]	
\morphism(1000,-50)|a|/@{=>}/<0,-200>[ ` ; \simeq]
\morphism(1000,-450)|a|/@{=>}/<0,-200>[ ` ; \simeq]
\morphism(100,-650)|a|/@{=>}/<-200,0>[ ` ;\epsilon(\underline{\Theta'}(bA_0))\phi]
\morphism(100,-1450)|a|/@{=>}/<-200,0>[ ` ;\epsilon((B_1 a)\underline{\Theta})\phi]
\morphism(1000,-850)|a|/@{=>}/<0,-200>[ ` ; \simeq]
\morphism(1000,-1250)|a|/@{=>}/<0,-200>[ ` ; \simeq]
\morphism(1000,-1650)|a|/@{=>}/<0,-200>[ ` ; \simeq]
\efig 
\] in which all the cells labelled by ``$\simeq$'' are obtained by verifying their existence in $\nCat{\omega}$, since their boundaries factor through appropriate globular sums, and $\underline{\Theta},\underline{\Theta}'$ are the underlying 4-cell of the modifications. The composite of this pasting diagram is the modification $\epsilon\Theta\Theta ' \phi$ of the statement.
\end{proof}
\end{lemma}
Given an operation $\rho\colon D_2 \rightarrow A$ in $\mathfrak{C}$, we can consider the map $\tilde{\rho}\colon \cyl(D_2)\rightarrow \cyl(A)$ obtained by applying Lemma \ref{extension of modifications} to the following diagram:
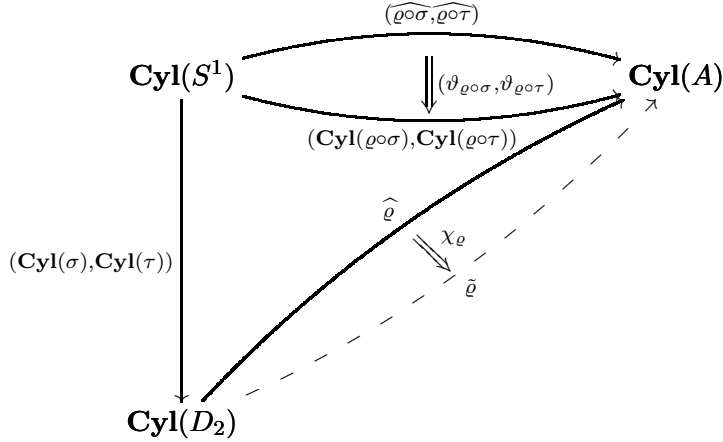
\[
\bfig
\morphism(0,0)|a|/{@{>}@/^1.5em/}/<1700,0>[\cyl(S^{1})`\cyl(A);(\widehat{\rho \circ \sigma},\widehat{\rho \circ \tau})]
\morphism(0,0)|b|/{@{>}@/^-1.5em/}/<1700,0>[\cyl(S^{1})`\cyl(A);(\cyl(\rho \circ \sigma), \cyl(\rho \circ \tau))\ \ \ \ ]
\morphism(850,125)|r|/@{=>}@<0pt>/<0,-250>[`; (\theta_{\rho \circ \sigma},\theta_{\rho \circ \tau})]
\morphism(0,0)|a|/{@{>}@//}/<0,-1200>[\cyl(S^{1})`\cyl(D_2);(\cyl(\sigma),\cyl(\tau))]
\morphism(0,-1200)|a|/{@{>}@/^1.2em/}/<1700,1200>[\cyl(D_2)`\cyl(A);\widehat{\rho}]
\morphism(0,-1200)|b|/{@{-->}@/^-1.2em/}/<1700,1200>[\cyl(D_2)`\cyl(A);\tilde{\rho}]
\morphism(775,-525)|r|/@{=>}@<0pt>/<150,-150>[`; \chi_{\rho}]
\efig
\] By construction, there is a modification $\chi_{\rho}\colon \hat{\rho} \Rightarrow \tilde{\rho}$. Also, note that $\cyl(\rho)$ and $\tilde{\rho}$, although potentially different, are parallel 2-cylinders.
\begin{lemma}
In the situation of Lemma \ref{(3)} and in the context of $\mathfrak{C}^{\mathbf{W}}$-models, we can replace $\hat{\rho}$ with $\tilde{\rho}$.
\begin{proof}
Consider the following diagram:
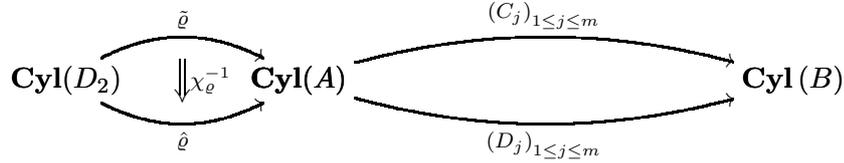
\[
\bfig 
\morphism(0,0)|a|/{@{>}@/^1.5em/}/<1700,0>[\cyl(A)`\cyl\left( B\right);\left(C_j \right)_{1\leq j \leq m}]
\morphism(0,0)|b|/{@{>}@/^-1.5em/}/<1700,0>[\cyl(A)`\cyl\left( B\right) ;\left(D_j \right)_{1\leq j \leq m}]
\morphism(-800,0)|a|/{@{>}@/^1.5em/}/<800,0>[\cyl(D_2)`\cyl(A);\tilde{\rho}]
\morphism(-800,0)|b|/{@{>}@/^-1.5em/}/<800,0>[\cyl(D_2)`\cyl(A);\hat{\rho}]
\morphism(-400,125)|r|/@{=>}@<0pt>/<0,-200>[`; \chi^{-1}_{\rho}]
\efig 
\] where we denote by $\chi^{-1}_{\rho}$ the modification obtained by applying Lemma \ref{inverse of modifications} to the modification $\chi_{\rho}$.
It induces a modification $\left(C_j \right)_{1\leq j \leq m} \circ {\chi_{\rho}}^{-1}\colon \left(C_j \right)_{1\leq j \leq m}\circ \tilde{\rho}\Rightarrow \left(C_j \right)_{1\leq j \leq m} \circ \hat{\rho}$, which we can compose with the modification $\Theta \circ \hat{\rho}\colon \left(C_j \right)_{1\leq j \leq m}\circ \hat{\rho} \Rightarrow \left(D_j \right)_{1\leq j \leq m}\circ \hat{\rho}$ obtained in Lemma \ref{extension of modifications}. Finally, we compose the result with $\left(D_j \right)_{1\leq j \leq m}\circ \chi_{\rho}\colon\left(D_j \right)_{1\leq j \leq m}\circ \hat{\rho} \Rightarrow \left(D_j \right)_{1\leq j \leq m}\circ \tilde{\rho}$ to get the desired modification.
\end{proof}
\end{lemma}
In what follows, we consider a homogeneous map $\phi\colon A \rightarrow B$ and we use the notation $\tilde{\phi}$ to denote the map $\cyl(A)\rightarrow\cyl(B)$ obtained by glueing the various maps $\tilde{\phi_j}\colon \cyl(D_{i_j})\rightarrow \cyl(B_j)$ for $1\leq j \leq m$, where $\phi_j$ is the homogeneous part of the composite $D_{i_j}\rightarrow A \rightarrow B$ for every globe $D_{i_j}$ in the globular decomposition of $A$, and $\colim_k B_k\cong B$.
\begin{lemma}
Assume given homogeneous operations $\rho\colon D_2\rightarrow A$ and $\phi\colon A \rightarrow B$. There is a modification \[\Lambda\colon\tilde{\phi}\circ \tilde{\rho} \Rightarrow \widetilde{\phi\circ \rho}\] in $\mathfrak{C}^{\mathbf{W}}$-models that is essentially given by a 4-cell (which is possible as both cylinders are parallel).
\end{lemma}
The idea of the proof is to consider the following diagram of composable modifications, where the solid ones have already been constructed:
\begin{equation}
\label{tilde modif}
\bfig
\morphism(0,0)|a|/@{=>}@<0pt>/<600,0>[\tilde{\phi}\circ \tilde{\rho}`\tilde{\phi}\circ \hat{\rho}; \tilde{\phi}\circ{\chi_{\rho}}]
\morphism(600,0)|a|/@{=>}@<0pt>/<600,0>[\tilde{\phi}\circ \hat{\rho}`\hat{\phi}\circ \hat{\rho}; {\chi_{\phi}}^{-1}\circ\hat{\rho}]
\morphism(1200,0)|a|/@{==>}@<0pt>/<600,0>[\hat{\phi}\circ \hat{\rho}`\widehat{\phi \circ \rho}; \eta]
\morphism(1800,0)|a|/@{=>}@<0pt>/<600,0>[\widehat{\phi \circ \rho}`\widetilde{\phi \circ \rho}; \chi_{\phi\circ \rho}]
\efig
\end{equation} 
 We then have to construct the modification denoted with $\eta$, and then prove that the resulting composite can be adjusted so as to consist of a 4-cell. This is accomplished by making use of the following lemma, once we have proven that its assumptions are satisfied in this case.
\begin{lemma}
	\label{reduction of modif's}
Assume given a pair of parallel $n$-cylinders $C,D\colon A \curvearrowright B$ in a $\mathfrak{C}^{\mathbf{W}}$-model $X$, together with a modification $\Theta\colon C\Rightarrow D$ between them. Assume further that $s(\Theta)$ and $t(\Theta)$ are essentially given by $(n+1)$-cells between the $(n-1)$-cylinders involved, and that there are $(n+2)$-cells $\eta_s\colon s(\Theta)\rightarrow 1_{\widehat{s(C)}}, \ \eta_t\colon t(\Theta)\rightarrow 1_{\widehat{t(C)}}$, where $\widehat{E}$ denotes the $n$-cell filling an $(n-1)$-cylinder. Then there exists a modification $\Theta'\colon C \Rightarrow D$ which essentially consists of an $(n+2)$-cell between $n$-cylinders.
\begin{proof}
We prove the statement by induction on $n>0$. Let $n=1$, and set $\Theta_{\epsilon}=\epsilon(\Theta)$ for $\epsilon=s,t$. Consider the following pasting diagram in $X\left(s(C_s),t(C_t)\right)$, where the unlabelled 2-cell comes from unitality of composition in $\mathfrak{C}^{\mathbf{W}}$:
\[
\bfig
\morphism(0,0)|a|/{@{>}@/^3.5em/}/<2000,0>[C_t A`B C_s;\hat{C}]
\morphism(0,0)|a|/{@{>}@/^1em/}/<800,0>[C_t A`C_t A;1]
\morphism(0,0)|b|/{@{>}@/^-1em/}/<800,0>[C_t A`C_t A;\Theta_s A]
\morphism(800,0)|a|/@{>}/<400,0>[C_t A`B C_s;\hat{C}]
\morphism(1200,0)|a|/{@{>}@/^1em/}/<800,0>[BC_s`BC_s;1]
\morphism(1200,0)|b|/{@{>}@/^-1em/}/<800,0>[BC_s`BC_s;B\Theta_t]
\morphism(0,0)|b|/{@{>}@/^-3.5em/}/<2000,0>[C_t A`B C_s;\hat{D}]

\morphism(1000,325)|a|/@{=>}@<0pt>/<0,-200>[ ` ; ]
\morphism(1000,-75)|a|/@{=>}@<0pt>/<0,-200>[ ` ;\Theta ]
\morphism(500,100)|a|/@{=>}@<0pt>/<0,-150>[ ` ;(\eta_s)^{-1} A ]
\morphism(1700,100)|a|/@{=>}@<0pt>/<0,-150>[ ` ;B (\eta_t)^{-1} ]
\efig 
\]
The composite of this pasting diagram is the modification $\Theta'$ we are looking for.

Now let $n>1$, we have a modification of $(n-1)$-cylinders $\overline{\Theta}\colon \overline{C}\Rightarrow \overline{D}$ in $X\left(s(C_s),t(C_t)\right)$. For $\epsilon=s,t$ we have $\epsilon(\overline{C})=\overline{\epsilon(C)}=\overline{\epsilon(D)}=\epsilon(\overline{D})$ and $\epsilon(\overline{\Theta})$ is an $n$-cell between $(n-2)$-cylinders. Also, we can view $\eta_s,\eta_t$ as $(n+1)$-cells in $X\left(s(C_s),t(C_t)\right)$, so that we can apply the inductive hypothesis to get a modification $\overline{\Theta'}\colon \overline{C}\Rightarrow\overline{D}$ which consists of an $(n+1)$-cell between $(n-1)$-cylinders in $X\left(s(C_s),t(C_t)\right)$. Its transpose $\Theta'\colon C \Rightarrow D$ is the modification we are looking for and this concludes the proof.
\end{proof}
\end{lemma}
We now recall Lemma 11.10 of \cite{EL}, since the modification we want to construct has to be compatible with the one obtained in that lemma in a sense that will be made precise in what follows.
\begin{lemma}
	Given compatible operations $\rho \colon D_1 \rightarrow D_1^{\otimes k}$, $\phi_j\colon D_i \rightarrow D_1^{\otimes q_j}$ for $1\leq j \leq k$ similarly to the previous lemma, there is an induced modification of $\mathfrak{C}$-models
	\begin{equation}
	\label{coherence mod}
	\bfig 
	\morphism(0,400)|a|/{@{>}@/^0em/}/<1700,-400>[\cyl(D_{1}^{\otimes k})`\cyl\left( D_{1}^{\otimes (\sum_j q_j) }\right);\plus{1\leq j \leq k}\widehat{\phi_j}]
	\morphism(-800,0)|a|/@{>}@<0pt>/<800,400>[\cyl(D_{1})`\cyl(D_{1}^{\otimes k});\hat{\rho}]
	\morphism(-800,0)|b|/@{>}@<0pt>/<2500,0>[\cyl(D_{1})`\cyl\left( D_{1}^{\otimes (\sum_j q_j) }\right);{ (\left(\phi_j\right)_{1\leq j \leq k}\circ \rho)}^{\wedge}]
	\morphism(150,250)|a|/{@{=>}@/^0em/}/<150,-150>[`;\delta_{\phi,\rho}]
	\efig 
	\end{equation}
\end{lemma} 
The proof of the following lemma is quite technical, but crucial to get the missing piece for this section.
\begin{lemma}
Assume given homogeneous operations $\rho\colon D_2\rightarrow A$ and $\phi\colon A \rightarrow B$. There is a modification  of $\mathfrak{C}$-models
\[\Delta\colon\widehat{\phi}\circ \widehat{\rho} \Rightarrow \widehat{\phi\circ \rho}\] with source and target given by the modifications $\cyl(j)\circ \delta_{(\phi \circ i)_{\epsilon},\rho_{\epsilon}}$ for $\epsilon=\sigma, \tau$, where $i\circ \rho_{\epsilon}$ is the homogeneous-globular factorization of $\rho \circ \epsilon$ and, similarly, $j\circ (\phi \circ i )_{\epsilon}$ is the homogeneous-globular factorization of $\phi \circ i$.
\[\bfig
\morphism(-1000,0)|a|/@{>}@<0pt>/<500,0>[D_2`A ;\rho]
\morphism(-1000,-500)|l|/@{>->}@<0pt>/<0,500>[D_1`D_2;\epsilon]
\morphism(-1000,-500)|a|/@{>>}@<0pt>/<500,0>[D_1`A';\rho_{\epsilon}]
\morphism(-500,-500)|r|/@{>->}@<0pt>/<0,500>[A'`A;i]
\morphism(-500,0)|a|/@{>}@<0pt>/<500,0>[A`B ;\phi]
\morphism(-500,-500)|a|/@{>>}@<0pt>/<500,0>[A'`B';(\phi\circ i)_{\epsilon}]
\morphism(0,-500)|a|/@{>->}@<0pt>/<0,500>[B'`B;j]
\efig 
\] Here, we have used the arrow $\twoheadrightarrow$ to denote homogeneous maps and $\rightarrowtail$ for globular ones.
\begin{proof}
The proof proceeds very similarly as to that of Lemma \ref{(3)}, i.e. we construct the modification $\Delta$ as the composite of modifications from substacks of the stack defining $\widehat{\phi}\circ \widehat{\rho}$ towards substacks of the one defining $\widehat{\phi \circ \rho}$, parametrized by the globular sums in $\mathcal{L}(A)$ in an exhaustive fashion.
We prove this representably, and we let \[\begin{pmatrix}
i_1 &&i_2 & \ldots&i_{m-1} & &i_m\\
& i'_1 & &\ldots&& i'_{m-1}
\end{pmatrix} \ \ \  \ \ \begin{pmatrix}
j_1 &&i_2 & \ldots&j_{q-1} & &j_q\\
& j'_1 & &\ldots&& j'_{q-1}
\end{pmatrix}\] be the table of dimensions of $A$ and $B$ respectively. This means that we are given a map $C\colon \cyl(B)\rightarrow X$, with $C\colon U\curvearrowright V$ and $X$ a 3-groupoid. From this, we get cylinders $C_1,\ldots,C_q$ in $X$, where $C_k$ is a $j_k$-cylinder. Notice that, by assumption on the homogeneity of $\rho$ and $\phi$, we have $i_k,j_r\leq 2$ for every $1\leq k \leq m$ and $1\leq r\leq q$. The two cylinders involved are both obtained as vertical composites of (different) stacks of 1-cylinders in $X(x,y)$ for $x=s^{i_1}(C_1)_0,y=s^{i_m}(C_m)_1$. Therefore, we need to provide a filler for this pair of composite 1-cells in the bicategory $\mathbf{hom}\left( D_1,X(x,y)\right) $, and we do so by decomposing it into some sub-composite, and we then explain how to find fillers for each such piece. As in Lemma \ref{(3)}, we firstly consider the cases \eqref{cyl 1} to \eqref{cyl 3}, where modifications can be constructed by using the fact that the corresponding cylinders factor through appropriate globular sums, and in $\nCat{\omega}$ the boundary data of these modifications admits fillers in a way that is compatible with the modifications we already have for the boundary. We now address case \eqref{cyl 4}, i.e. globular sums $D\in \mathscr{L}(A)$ which have been obtained by adding a new vertex $*_D$ to $A$ at height $\height(*_D)=2$, in such a way that this new vertex is the unique element of the fiber over $D_{i_k}=D_1$. Given such globular sum $D$, the 2-cell in $D$ represented by $*_D$ picks out the 2-cell associated with the 1-cylinder $F_k=\widehat{\phi_k}(C_{n_k},\ldots,C_{n_{k+1}-1})$, where $\phi=(\phi_i)_{1\leq i \leq m}$, according to the globular decomposition of $A$, and each of the $\phi_k$ has the sub-globular sum $G_k\subset B$ spanned by $D_{j_{n_k}},\ldots,D_{j_{n_{k+1}}-1}$ as codomain. Corresponding to such $D$, we have a cylinder in the stack associated with $\hat{\phi}\circ \hat{\rho}$ of the form:
\[
\bfig 
\morphism(-1000,0)|a|/@{>}@<0pt>/<1350,0>[a`b;\rho\left(\phi_{>k}\left(U_{\vert G_{>k}}\right),   s\left(\underline{F_k}\right),  \phi_{<k}\left(V_{\vert G_{<k}}\right)\right)]
\morphism(-1000,-500)|b|/@{>}@<0pt>/<1350,0>[a'`b';\rho\left(\phi_{>k}\left(U_{\vert G_{>k}}\right),   t\left(\underline{F_k}\right), \phi_{<k}\left(V_{\vert G_{<k}}\right)\right)]
\morphism(350,0)|r|/@{>}@<0pt>/<0,-500>[b`b';\rho^{\ast}_{\tau}\left({(\phi_{>k})}_{\tau}\left(U_{\vert G_{>k}}\right),   \underline{F_k},  {(\phi_{<k})}_{\tau}\left(V_{\vert G_{<k}}\right)\right)]
\morphism(-1000,0)|l|/@{>}@<0pt>/<0,-500>[a`a';\rho^{\ast}_{\tau}\left((\phi_{>k})_{\tau}\left(U_{\vert G_{>k}}\right),  \underline{F_k}, (\phi_{<k})_{\tau}\left(V_{\vert G_{<k}}\right)\right)]
\morphism(-300,-150)|a|/@{=>}/<0,-200>[ ` ; ]
\efig 
\]
We have used $\phi_{>k} \left(U_{\vert G_{>k}}\right)$ to denote the result of composing, using $(\phi_i)_{i\geq k}$, the restriction of $U$ to the union of the sub-globular sum of $B$ given by $G_i$ for $i>k$. Similarly for the other piece of notation involving the indices smaller than $k$.
 We want to produce a modification having this cylinder as source, and having as target a sub-composite of the vertical stack of 1-cylinders whose composite is the transpose of $\widehat{\phi\circ \rho}$ under the adjunction $\Sigma \dashv \Omega$. This sub-stack is the one parametrized by the family of globular sums of the form \[\{D_{j_1}\plus{D_{j_1'}}\ldots\plus{D_{j'_{n_k-1}}}E\plus{D_{j_{n_{k+1}}-1}}\ldots\plus{D_{j'_{q-1}}}D_{j_q}\}_{E \in \mathscr{L}(G_k)}\subset \mathscr{L}(B)\] 
Notice that the respective boundaries of these cylinders are of the same form as the ones appearing in the proof of Lemma 11.10 in \cite{EL}, and therefore we can use the modifications we produced there to compare the boundaries. These constitute the boundary of the modification we want to construct, whose existence follows, finally, from the fact that this boundary factors trough the globular sum \[D_{j_1}\plus{D_{j_1'}}\ldots\plus{D_{j'_{n_k-1}}}\plus{D_{j'_{n_k-1}}} \Sigma(D_1^{\otimes \vert \mathscr{L}(D_1^{\otimes n_{k+1}-n_k}) \vert})\plus{D_{j_{n_{k+1}}-1}} \ldots\plus{D_{j'_{q-1}}}D_{j_q}\] and a filler for it certainly exists in $\nCat{\omega}$.
We now turn to the case of globular sums $C\in \mathscr{L}(A)$ corresponding to case \eqref{cyl 5} to \eqref{cyl 8}. We can thus consider the maximal sub-globular sum of $A$ of the form $D_2\plus{D_1}\ldots\plus{D_1}D_2\cong \Sigma D_1^{\otimes k}$ that contains the copy of $D_2$ to which the new edges have been adjoined. The globular inclusion $\Sigma D_1^{\otimes k}\rightarrow A$ picks out $k$ composable 2-cylinders $\Gamma_1,\ldots, \Gamma_k $ in $X$, where we have $\Gamma_i=\phi_{r+i}(C_{n_{r+i}},\ldots,C_{n_{r+i+1}-1})$ for a unique integer $r$.
We have to construct a modification whose source is given by a stack of (collapsed) cylinders of the form:
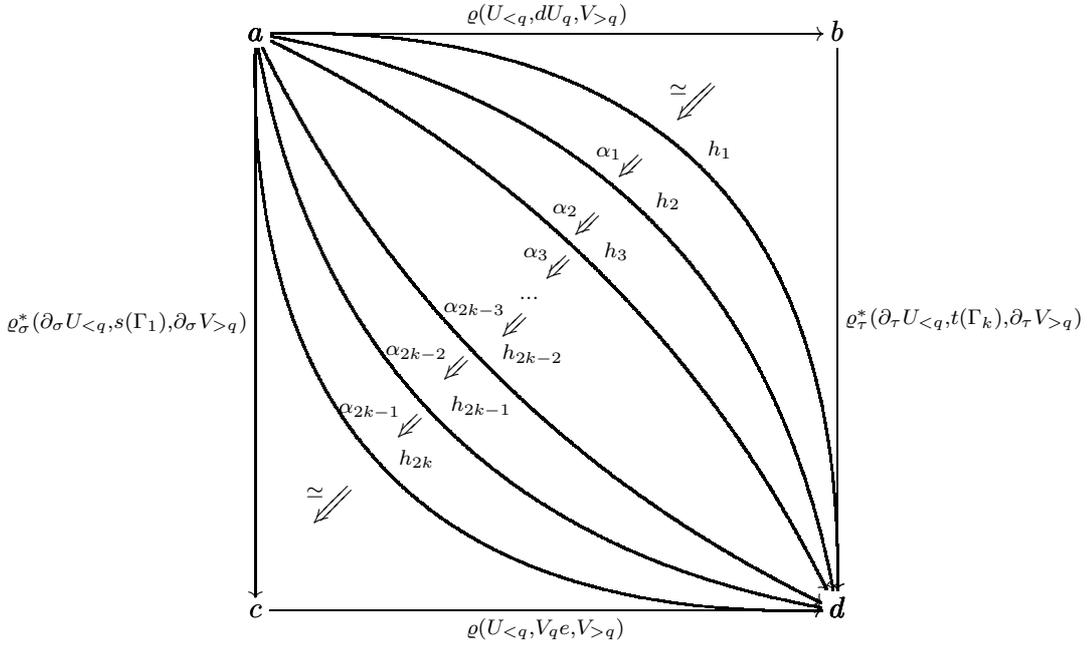
\[
\bfig 
\morphism(0,0)|a|/@{>}/<2000,0>[a`b;\rho\left(U_{<q},d U_q, V_{>q} \right)]	
\morphism(0,0)|a|/@{>}/<0,-2000>[a`c;\rho^{\ast}_{\sigma}\left(\partial_{\sigma}U_{<q},s(\Gamma_1),\partial_{\sigma}V_{>q} \right)  ]
\morphism(0,-2000)|b|/@{>}/<2000,0>[c`d;\rho\left(U_{<q}, V_q e, V_{>q} \right)]
\morphism(2000,0)|r|/@{>}/<0,-2000>[b`d;\rho^{\ast}_{\tau}\left(\partial_{\tau}U_{<q},t(\Gamma_k),\partial_{\tau}V_{>q} \right)]

\morphism(0,0)|a|/{@{>}@/^7.5em/}/<2000,-2000>[a`d;h_1]
\morphism(0,0)|a|/{@{>}@/^5em/}/<2000,-2000>[a`d;h_2]
\morphism(0,0)|a|/{@{>}@/^2.5em/}/<2000,-2000>[a`d;h_3]
\morphism(0,0)|r|/{@{>}@/^-2.5em/}/<2000,-2000>[a`d;h_{2k-2}]
\morphism(0,0)|r|/{@{>}@/^-7.5em/}/<2000,-2000>[a`d;h_{2k}]
\morphism(0,0)|r|/{@{>}@/^-5em/}/<2000,-2000>[a`d; h_{2k-1}]
\morphism(1600,-150)|a|/@{=>}/<-150,-150>[ ` ;\simeq]
\morphism(1350,-400)|a|/@{=>}/<-100,-100>[ ` ;\alpha_1]
\morphism(1200,-600)|a|/@{=>}/<-100,-100>[ ` ;\alpha_2]
\morphism(1100,-750)|a|/@{=>}/<-100,-100>[ ` ;\alpha_3]
\morphism(1000,-950)|a|/@{}/<0,0>[ ` ;\ldots]
\morphism(950,-950)|a|/@{=>}/<-100,-100>[ ` ;\alpha_{2k-3}]
\morphism(750,-1100)|a|/@{=>}/<-100,-100>[ ` ;\alpha_{2k-2}]
\morphism(590,-1300)|a|/@{=>}/<-100,-100>[ ` ; \ \ \ \alpha_{2k-1}]
\morphism(350,-1550)|a|/@{=>}/<-150,-150>[ ` ;\simeq]
\efig 
\]
%
%
Here, we set $d=t^2(\Gamma_1)$, $e=s^2(\Gamma_1)$, and: \[h_{2m+1}=\rho\left(U_{<q}, {U_q}^{< r},s(\underline{\Gamma_{k-m}}),{V_q}^{> r}a, V_{>q} \right)\]  \[h_{2m}=\rho\left(U_{<q}, {U_q}^{< r},t(\underline{\Gamma_{k-m}}),{V_q}^{> r}a, V_{>q} \right)\] where we have used $\underline{F}$ to denote the underlying 3-cell of a 2-cylinder $F$, and $t(\Gamma_0)$ is defined to be $s(\Gamma_1)$. The 2-cells in $X(x,y)$ labelled with $\alpha$'s represent 1-cylinders whose source and target are degenerate. In particular, $\alpha_{2m+1}$ is a whiskering of the 3-cell in $\Gamma_{k-m}$ and $\alpha_{2m}$ is an associativity constraint.
The target of the modification we want to construct is a composite of a sub-stack of the one associated with $\widehat{\phi \circ \rho}$, parametrized by the family of globular sums given by 
\[\{D_{j_1}\plus{D_{j_1'}}\ldots\plus{D_{j'_{n_k-1}}}E\plus{D_{j_{n_{k+1}}-1}}\ldots\plus{D_{j'_{q-1}}}D_{j_q}\}_{E \in \mathscr{L}(G_p)}\subset \mathscr{L}(B)\]  
To finish this construction, we introduce an intermediate step in this modification by taking into consideration Lemma \ref{pseudo nat transf} in the Appendix, and we focus on the square \eqref{pseudonat square} that originates from it. By applying this to each of the (possibly degenerate) 1-cylinders in the sub-stack we are considering, we obtain a new stack of the same shape where all the new 1-cylinders are whiskering of the previous ones in the appropriate sense. The respective boundaries of these stacks we are comparing are of the same form as the ones appearing in the proof of Lemma 11.10 in \cite{EL}, and therefore we can use the modifications we produced there to compare the boundaries. In the same way, the boundary of this new ``whiskered'' stack and the one of the source of the modification we want to build can also be compared using the modification of Lemma 11.10 in \cite{EL} (which was indeed the composite of two such). After having adjusted the boundaries, filling in the rest of the modification follows from a straightforward application of the classical result of coherence for pseudofunctors and bicategories. 
\end{proof}
\end{lemma}
 By construction, it is clear that the boundary of the composite modification in \ref{tilde modif} satisfies the assumptions of Lemma \ref{reduction of modif's}. Hence, when we quotient out the 4-cells, we find that the following identity holds:
 \[\tilde{\phi}\circ \tilde{\rho}=\widetilde{\phi \circ \rho}\] Therefore, we see that an extension to $\mathfrak{C}_2$ is equivalently obtained by setting $\cyl(\rho)=\tilde{\rho}$ for all homogeneous operations $\rho\colon D_2 \rightarrow A$. Finally, we recall that, by definition, $\mathfrak{C}_3\cong \mathfrak{C}_2[X_2]$ and so we can obtain the desired extension depicted in \ref{extension to D3} by defining $\cyl(\Phi)$, for every $\Phi\colon D_3 \rightarrow A$ added as a filler of $(\phi_0,\phi_1)\in X_2$, in the following manner:
 \[
 \bfig 
 \morphism(0,0)|a|/@{>}@<3pt>/<500,0>[D_2`A;\phi_0]
 \morphism(0,0)|b|/@{>}@<-3pt>/<500,0>[D_2`A;\phi_1]
 
 \morphism(0,0)|r|/@{>}@<3pt>/<0,-400>[D_2`D_{3};\tau_2]
 \morphism(0,0)|l|/@{>}@<-3pt>/<0,-400>[D_2`D_{3};\sigma_2]
 
 \morphism(0,-400)|r|/@{>}@<0pt>/<500,400>[D_{3}`A;\Phi]
 \efig  \ \ \ \ \
 \rightsquigarrow \ \ \ \ \
 \bfig
 \morphism(0,0)|a|/{@{>}@/^1.5em/}/<1700,0>[\cyl(S^{2})`\cyl(A);(\widehat{\phi_0},\widehat{\phi_1})]
 \morphism(0,0)|b|/{@{>}@/^-1.5em/}/<1700,0>[\cyl(S^{2})`\cyl(A);(\widetilde{\phi_0},\widetilde{\phi_1})\ \ \ \ ]
 \morphism(850,125)|r|/@{=>}@<0pt>/<0,-250>[`; (\chi_{\phi_0},\chi_{\phi_1})]
 \morphism(0,0)|a|/{@{>}@//}/<0,-1000>[\cyl(S^{2})`\cyl(D_3);(\cyl(\sigma),\cyl(\tau))]
 \morphism(0,-1000)|a|/{@{>}@/^1.2em/}/<1700,1000>[\cyl(D_3)`\cyl(A);\widehat{\Phi}]
 \morphism(0,-1000)|b|/{@{-->}@/^-1.2em/}/<1700,1000>[\cyl(D_3)`\cyl(A);\cyl(\Phi)]
 \morphism(725,-425)|r|/@{=>}@<0pt>/<150,-150>[`; ]
 \efig
 \]
  This concludes the construction of the path 3-category of a 3-groupoid, as we record here below.
 \begin{thm}
 	\label{3 cat of paths}
Let $\mathfrak{C}$ be a 3-coherator for categories, and denote with $\nCat{3}$ the category of weak 3-categories modeled by it. Then there is a lift of the form:
\[\bfig
\morphism(0,-400)|a|/@{-->}@<0pt>/<600,400>[ \mathbf{Mod}(\mathfrak{C}^{\mathbf{W}})`\mathbf{Mod}(\mathfrak{C})\cong \nCat{3};\p]
\morphism(0,-400)|a|/@{>}@<0pt>/<600,0>[\mathbf{Mod}(\mathfrak{C}^{\mathbf{W}})`[\G_3^{op},\mathbf{Set}];\p]
\morphism(600,0)|r|/@{>}@<0pt>/<0,-400>[\mathbf{Mod}(\mathfrak{C})\cong \nCat{3}`[\G_3^{op},\mathbf{Set}];\mathbf{U}]
\efig
\] where we set $\p X=\mathbf{Mod}(\mathfrak{C}^{\mathbf{W}})\left(\cyl(D_{\bullet}),X\right) $.
 \end{thm}
Thanks to Theorem 6.5 of \cite{EL}, it is possible to extend the codomain a bit further, thus landing in a category whose objects posses a richer algebraic structure. We recall here the definition that is needed to formulate this extension result. In what follows, an operation $\phi\colon D_{n+1}\rightarrow A$ in an $n$-globular theory means an equality of the form $\phi\circ \sigma =\phi \circ \tau$.
\begin{cor}
	\label{lift to C_W}
Let $\mathfrak{C}$ be a 3-coherator for categories, then there exists a lift of the form:
\[\bfig
\morphism(0,-400)|a|/@{-->}@<0pt>/<600,400>[ \mathbf{Mod}(\mathfrak{C}^{\mathbf{W}})`\mathbf{Mod}(\mathfrak{C}^{\mathbf{W}});\p]
\morphism(0,-400)|a|/@{>}@<0pt>/<600,0>[\mathbf{Mod}(\mathfrak{C}^{\mathbf{W}})`[\G_3^{op},\mathbf{Set}];\p]
\morphism(600,0)|r|/@{>}@<0pt>/<0,-400>[\mathbf{Mod}(\mathfrak{C}^{\mathbf{W}})`[\G_3^{op},\mathbf{Set}];\mathbf{U}]
\efig
\]
\begin{proof}
Such an extension amounts to endow the $\mathfrak{C}$-model $\p X$ obtained in the previous theorem with a system of inverses with respect to the chosen systems of compositions and identities. This was done in Theorem 6.5 of \cite{EL} in the case of two-sided inverses: therefore, we can interpret both the left inverse operation and the right inverse one by means of that, or simply adapt the proof given there to produce a left and a right inverse.
\end{proof}
\end{cor}
This proves Theorems \ref{path object for Batanin 3-groupoids} and \ref{model structure} and concludes this section.
\newpage
\appendix
\section{Pseudofunctoriality of whiskerings}
In this section of the appendix we want to record some results and constructions that involve the mapping 2-groupoid of morphisms $X(x,y)$, where $X$ is a 3-groupoid and $x,y$ are 0-cells in $X$.
Having in mind that 2-groupoids essentially corresponds to unbiased bicategories with weak inverses, we will treat them as such.

Let's consider the following situation, we are given 1-dimensional globular pasting diagrams in $X$ of the form $\alpha\colon A \rightarrow X, \beta\colon B \rightarrow X$, with $\partial_{\sigma}(\alpha)\overset{def}{=}\alpha\circ \partial_{\sigma}=w,\ \partial_{\tau}(\alpha)=x,\ \partial_{\sigma}(\beta)=y,\ \partial_{\tau}(\beta)=z$. Moreover, we are given a homogeneous operation $\rho\colon D_1 \rightarrow A\plus{D_0}D_1\plus{D_0}B$. We then have the following result:
\begin{lemma}
	The data above extend to a pseudofunctor of bicategories of the form \[\left( \alpha,-,\beta\right)\circ \rho\colon X(x,y)\rightarrow X(w,z)\] in $\ngpd{2}$.
	\begin{proof}
		Choose operations $\rho^2,\rho^3$ as depicted in the following diagrams:
		\[
		\bfig 
		\morphism(0,0)|a|/@{>}@<3pt>/<800,0>[D_1`A\plus{D_0}D_2\plus{D_0}B;(1,\sigma,1)\circ \rho]
		\morphism(0,0)|b|/@{>}@<-3pt>/<800,0>[D_1`A\plus{D_0}D_2\plus{D_0}B;(1,\tau,1)\circ \rho]
		
		\morphism(0,0)|r|/@{>}@<3pt>/<0,-400>[D_1`D_{2};\tau]
		\morphism(0,0)|l|/@{>}@<-3pt>/<0,-400>[D_1`D_{2};\sigma]
		
		\morphism(0,-400)|r|/@{>}@<0pt>/<800,400>[D_{2}`A\plus{D_0}D_2\plus{D_0}B;\rho^2]
		\efig 
		\bfig 
		\morphism(200,0)|a|/@{>}@<3pt>/<800,0>[D_2`A\plus{D_0}D_3\plus{D_0}B;(1,\sigma,1)\circ \rho^2]
		\morphism(200,0)|b|/@{>}@<-3pt>/<800,0>[D_2`A\plus{D_0}D_3\plus{D_0}B;(1,\tau,1)\circ \rho^2]
		
		\morphism(200,0)|r|/@{>}@<3pt>/<0,-400>[D_2`D_{3};\tau]
		\morphism(200,0)|l|/@{>}@<-3pt>/<0,-400>[D_2`D_{3};\sigma]
		
		\morphism(200,-400)|r|/@{>}@<0pt>/<800,400>[D_{3}`A\plus{D_0}D_3\plus{D_0}B;\rho^3]
		\efig 
		\] Next, define the underlying map of globular sets to be given by $\left( \alpha,-,\beta\right)\circ \rho^{k+1}\colon X(x,y)_k\rightarrow X(w,z)_k$ on $k$-cells, where we implicitly use the isomorphism of sets $ X(a,b)_k\cong \{f\in X_{k+1}\vert s^k(f)=a,t^k(f)=b\}$. The fact that this extends to a pseudofunctor is a simple exercise using contractibility of globular sums, and is therefore left to the interested reader.
	\end{proof}
\end{lemma}
If we go one dimension up, we can consider the following situation: we are given globular sums $A,B$ with $\max\{\dim(A),\dim(B)\}=2$, and maps $\alpha\colon A \rightarrow X, \beta\colon B \rightarrow X$, with  $\partial^{\dim(A)}_{\sigma}(\alpha)\overset{def}{=}\alpha\circ \partial^{\dim(A)}_{\sigma}=w,\ \partial^{\dim(A)}_{\tau}(\alpha)=x,\ \partial^{\dim(B)}_{\sigma}(\beta)=y,\ \partial^{\dim(B)}_{\tau}(\beta)=z$. Furthermore, assume given a homogeneous operation of the form $\rho\colon D_2 \rightarrow A+D_2+B$, where $+$ denotes either $\plus{D_0}$ or $\plus{D_1}$. We then have the following result, whose proof is analogous to that of the previous one.
\begin{lemma}
	\label{pseudo nat transf}
	The previous data determine a pseudo-natural transformation of the form:
	\[
	\bfig 
	\morphism(0,0)|a|/{@{>}@/^2em/}/<1200,0>[X(x,y)`X(w,z);\left(\partial^{\dim(A)-1}_{\sigma}(\alpha),-,\partial^{\dim(B)-1}_{\sigma}(\beta)\right)\circ \rho_\sigma]
	\morphism(0,0)|r|/{@{>}@/^-2em/}/<1200,0>[X(x,y)`X(w,z);\left(\partial^{\dim(A)-1}_{\tau}(\alpha),-,\partial^{\dim(B)-1}_{\tau}(\beta)\right)\circ \rho_\tau]
	
	\morphism(600,100)|a|/@{=>}/<0,-200>[`;\left( \alpha,-,\beta\right)\circ \rho]
	\efig 
	\] where $\rho_{\epsilon}$ denotes the homogeneous part of the composite $\rho \circ \epsilon$ for $\epsilon=\sigma,\tau$.
\end{lemma}
Finally, we observe that given bicategories $\mathcal{K},\mathscr{L}$, a square in $\mathcal{K}$ of the form:
\[\bfig 
\morphism(0,0)|a|/@{>}@<0pt>/<400,0>[A` B;h]
\morphism(0,0)|a|/@{>}@<0pt>/<0,-400>[A`C;f]
\morphism(0,-400)|b|/@{>}@<0pt>/<400,0>[C` D;k]
\morphism(400,0)|r|/@{>}@<0pt>/<0,-400>[B`D;g]

\morphism(200,-100)|a|/@{=>}/<0,-200>[`;\Theta]
\efig \] and a pseudo-natural transformation
\[\bfig 
\morphism(0,0)|a|/{@{>}@/^1.5em/}/<600,0>[\mathcal{K}`\mathscr{L};F]
\morphism(0,0)|r|/{@{>}@/^-1.5em/}/<600,0>[\mathcal{K}`\mathscr{L};G]

\morphism(300,100)|a|/@{=>}/<0,-200>[`;\alpha]
\efig \] we get a filler for the square
\begin{equation}
\label{pseudonat square}
\bfig 
\morphism(0,0)|a|/@{>}@<0pt>/<600,0>[FA` GB;G(h)\alpha_A]
\morphism(0,0)|a|/@{>}@<0pt>/<0,-400>[FA`FC;F(f)]
\morphism(0,-400)|b|/@{>}@<0pt>/<600,0>[FC` GD;F(k)\alpha_D]
\morphism(600,0)|r|/@{>}@<0pt>/<0,-400>[GB`GD;G(g)]

\morphism(300,-100)|a|/@{==>}/<0,-200>[`;\exists\Phi]
\efig
\end{equation}
Indeed, it is enough to consider the following composite:
\[
\bfig
\morphism(0,0)|a|/@{>}@<0pt>/<700,0>[G(g)(G(h)\alpha_A)` G(gh)\alpha_A;\cong]
\morphism(700,0)|a|/@{>}@<0pt>/<800,0>[G(gh)\alpha_A`G(kf)\alpha_A;G(\Theta)\alpha_A]
\morphism(1500,0)|a|/@{>}@<0pt>/<600,0>[G(kf)\alpha_A`\alpha_D F(kf);\alpha_{kf}]
\morphism(2100,0)|a|/@{>}@<0pt>/<700,0>[\alpha_D F(kf)`(\alpha_D F(k)) F(f);\cong]
\efig 
\] It is clear that an analogous statement holds if we replace squares with commutative triangles or even just a 2-cell, thus covering the case of all possible degenerate 1-cylinders. 
\section{A bicategory of cylinders and modifications}
Given an $\infty$-groupoid $X$ and an integer $n\geq 0$, we want to organize the collection of $n$-cells, $n$-cylinders and modifications between $n$-cylinders into an algebraic structure that allows us to perform calculations with them and encode the low dimensional structure of a yet to be defined internal hom between $\infty$-groupoids.
This is going to be a truncation of an $\infty$-groupoid that results from the existence of a Gray tensor product $\bigotimes\colon\wgpd \times \wgpd\rightarrow \wgpd$.

To simplify things, the bicategory that we define will be of a special kind, as defined below.
\begin{rmk}
Everything that follows can be proven to hold true also in $\mathbf{Mod}(\mathfrak{C})$ for any given coherator for $\infty$-categories $\mathfrak{C}$. Indeed, all the fillers obtained using contractibility can be obtained using the methods we described in Section 5, once we observe that the latching map of $\Xi\colon\cyl(\bullet)\ast \cyl(\bullet)\rightarrow \mathbf{M}_{\bullet}$ is a pushout of the boundary inclusion $S^{n+1}\rightarrow D_{n+2}$, as proven in Lemma 10.3 of \cite{EL}.
\end{rmk}
\begin{defi}
	A bicategory $\mathcal{C}$ is called locally posetal if the hom-category $\mathcal{C}(A,B)$ is a poset for every pair of objects $(A,B)$ in $\mathcal{C}$.
\end{defi}
Suppose given a $2$-truncated globular set $X\colon \G_{\leq 2}^{op} \rightarrow \mathbf{Set}$, we want to get a locally posetal bicategory $\chi(X)$ from it by setting $X_0$ as its set of objects and defining the underlying graph of each of its hom-categories to be
\[\chi(X)(a,b)_k=\begin{cases}
X_1(a,b)& k=1\\
\coprod_{(f,g)}\{*\}&X_2(f,g)\neq \emptyset
\end{cases}\]
In words, we are saying that there is a $2$-cell $\alpha\colon f \rightarrow g$ if and only if the set $X_2(f,g)$ is non-empty.

What extra structure do we need to define, and what conditions should it satisfy in order to get a locally posetal bicategory? The properties not encoded by the structures, i.e. the axioms for a bicategory, all concern equality between $2$-cells, and therefore are trivially satisfied.
Thus we only need to define the following operations:
\begin{enumerate}
	\item composition of $1$-cells $\chi(X)_1 \times_{\chi(X)_0}\chi(X)_1 \rightarrow \chi(X)_1$;
	\item \label{vert comp of modif} vertical composition of $2$-cells $\chi(X)_2 \times_{\chi(X)_1}\chi(X)_2 \rightarrow \chi(X)_2$;
	\item whiskerings $\chi(X)_2 \times_{\chi(X)_0}\chi(X)_1 \rightarrow \chi(X)_2$ and $\chi(X)_1 \times_{\chi(X)_0}\chi(X)_2 \rightarrow \chi(X)_2$;
	\item identity $1$-cells $1_a \in \chi(X)(a,a)$ for every $a \in \ob(\chi(X))$;
	\item identity $2$-cells $1_f\colon f \Rightarrow f$ for every $f\in X_1$;
	\item unit constraints, which amount to check that $X_2(f \circ 1_{s(f)},f)\neq \emptyset \neq X_2(1_{t(f)}\circ f,f)$ and $X_2(f,f \circ 1_{s(f)})\neq \emptyset \neq X_2(f,1_{t(f)}\circ f)$;
	\item associators, which amount to check that \[X_2((h\circ g)\circ f,h\circ (g \circ f)\neq \emptyset \neq X_2(h\circ (g \circ f),(h\circ g)\circ f)\]
\end{enumerate}
Given an $\infty$ groupoid $X$, we define a $2$-truncated globular set out of it, for each $n\geq 0$, as follows:
\[ \mathbf{hom}(D_n,X)_k \begin{cases}
\wgpd(D_n,X)=X_n & k=0\\
\wgpd (\cyl(D_n),X)& k=1\\
\wgpd(\M_n,X)&k=2
\end{cases} \]
where, the globular structure is induced by precomposition with the structural maps \[\iota=(\iota_0,\iota_1)\colon D_n\coprod D_n \rightarrow \cyl(D_n) \ \text{and} \ \Xi=(\Xi_0,\Xi_1)\colon \cyl (D_n)+\cyl (D_n)\rightarrow \M_n\]
All the proof and construction that follow, can be adapted to the more general case of (possibly) degenerate cylinders as 1-cells and (possibly) degenerate modifications as 2-cells. The latter can be defined in a straightforward way by mimicking the changes made in going from normal cylinders to degenerate ones.

We already have some of the operations required to get a locally posetal bicategory out of it: composition of $1$-cells is given by vertical composition of cylinders, and the identity $1$-cell on an $n$-cell $A\in X_n$ is the trivial cilinder $\mathbf{C}_A$ defined as the composite $\bfig 	\morphism(0,0)|a|/@{>}/<500,0>[\cyl(D_n)`D_n;\mathbf{C}_n] 	\morphism(500,0)|a|/@{>}/<300,0>[D_n`X;A]\efig$.

The existence of the rest of the structure in the case $n=0$ is straightforward, and follows directly from the contractibility of the coherator $\mathfrak{C}$.
In what follows, we fix an integer $n > 0$ and we assume as inductive hypothesis that $\mathbf{hom}(D_k,X)$ is a locally posetal bicategory for each $k< n$.

Let us now address point \eqref{vert comp of modif}, i.e. vertical composition of modifications.
From here onwards, until the end of this section, whenever a $1$-cell is labelled with $\Theta, \Psi$ or $\Phi$, that refers to the coherence cylinders considered in Definition \ref{coherence cyls}.
\begin{lemma}
	Given a pair of composable modifications $\Theta\colon F \Rightarrow G, \ \Psi \colon G \Rightarrow H$ between $n$-cylinders $F,G,H\colon A \curvearrowright B$ in $X$, there exists a composite modification \[\Psi \circ \Theta\colon F \Rightarrow H \]
	\begin{proof}
		Define the $2$-cells $(\Psi \circ \Theta)_{\epsilon}=\Psi_{\epsilon}\circ \Theta_{\epsilon}$ for $\epsilon=s,t$, using the same operation representing vertical composition of $2$-cells in $X $ that has been chosen for $\mathbf{hom}(D_0,X)$ (e.g. the one used in the definition of cylinders).
		Consider the following $2$-dimensional pasting diagram in the bicategory $\mathbf{hom}(D_{n-1},X(x,y)) $, where $x=s^2(\Theta_s)$ and $y=t^2(\Theta_t)$:
		\[
		\bfig
		\morphism(0,0)|l|/@{>}/<-800,-400>[t^n(H)A`t^n(F)A;\Gamma(A,(\Psi \circ \Theta)_t)] 
		\morphism(-800,-400)|l|/@{>}/<0,-400>[t^n(F)A`Bs^n(F);\bar{F}] 
		\morphism(-800,-800)|b|/@{>}/<800,-400>[Bs^n(F)`Bs^n(H);\Upsilon((\Psi \circ \Theta)_s,B) \ \ \ ] 
		\morphism(0,0)|r|/@{>}/<0,-400>[t^n(H)A`t^n(G)A;\Gamma(A, \Psi_t)]
		\morphism(0,-400)|a|/@{>}/<-800,0>[t^n(G)A`t^n(F)A;\Gamma(A, \Theta_t)]
		\morphism(0,-400)|r|/@{>}/<0,-400>[t^n(G)A`Bs^n(G);\bar{G}]
		\morphism(-800,-800)|a|/@{>}/<800,0>[Bs^n(F)`Bs^n(G);\Upsilon( \Theta_s,B)]
		\morphism(0,-800)|r|/@{>}/<0,-400>[Bs^n(G)`Bs^n(H);\Upsilon( \Psi_s,B)]
		
		\morphism(0,0)|r|/{@{>}@/^7.5em/}/<0,-1200>[t^n(H)A`Bs^n(H);\bar{H}] 
		
		\morphism(-300,-250)|a|/@{=>}/<200,0>[`;\alpha]
		\morphism(-500,-600)|a|/@{=>}/<200,0>[`;\bar{\Theta}]
		\morphism(-300,-950)|r|/@{=>}/<200,0>[`;\beta]
		\morphism(300,-600)|r|/@{=>}/<200,0>[`;\bar{\Psi}]
		\efig
		\]
		Here, the existence of $\alpha$ (resp. $\beta$) follows by an application of 	 Lemma \ref{modifications in contractible groupoids} to the contractible $\infty$-groupoid $D_n\coprod_{ D_0} D_2\coprod_{D_1}D_2$ (resp. $D_2\coprod_{D_1}D_2\coprod_{ D_0}D_n$).
		The composite of this pasting diagram defines the modification claimed in the statement, thus concluding the proof.
	\end{proof}
\end{lemma}
Let us now address the problem of constructing identity $2$-cells in $\mathbf{hom}(D_n,X)$.
\begin{lemma}
	Given an $n$-cylinder $F\colon A \curvearrowright B$ in $X$, there exists a modification of $n$-cylinders in $X$ of the form $1_F\colon F \Rightarrow F$.
	\begin{proof}
		Define a pair of $2$-cells $(1_F)_{\epsilon}=1_{\epsilon^n(F)}$ for $\epsilon=s,t$, where $1_f$ denotes the choice of an identity $2$-cell on $f$, when $f$ is a $1$-cell of $X$.
		Consider the following $2$-dimensional pasting diagram in $\mathbf{hom}\left( D_{n-1}, X(x,y)\right)$, with $x=s^n(A)$ and $y=t^n(B)$:
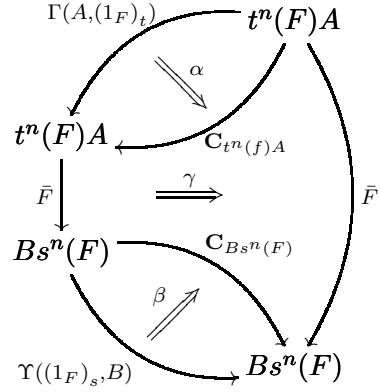
		\[
		\bfig
		\morphism(0,0)|l|/{@{>}@/^-2em/}/<-800,-400>[t^n(F)A`t^n(F)A;\Gamma(A,{(1_F)}_t)]
		\morphism(0,0)|b|/{@{>}@/^2em/}/<-800,-400>[t^n(F)A`t^n(F)A;\mathbf{C}_{t^n(f)A}] 
		\morphism(-800,-400)|l|/@{>}/<0,-400>[t^n(F)A`Bs^n(F);\bar{F}] 
		\morphism(-800,-800)|b|/{@{>}@/^-2em/}/<800,-400>[Bs^n(F)`Bs^n(F);\Upsilon({(1_F)}_s,B) \ \ \ ] 
		\morphism(-800,-800)|a|/{@{>}@/^2em/}/<800,-400>[Bs^n(F)`Bs^n(F);\mathbf{C}_{Bs^n(F)}]
		\morphism(0,0)|r|/{@{>}@/^2em/}/<0,-1200>[t^n(F)A`Bs^n(F);\bar{F}]
		
		\morphism(-500,-100)|a|/@{=>}/<200,-200>[`;\alpha]
		\morphism(-500,-1100)|a|/@{=>}/<200,200>[`;\beta]
		\morphism(-500,-600)|a|/@{=>}/<280,0>[`;\gamma]
		\efig
		\]
		Here, $\alpha$ (resp. $\beta$) is obtained by applying Lemma \ref{modifications in contractible groupoids} to the contractible $\infty$-groupoid $D_n\coprod_{ D_0}D_1$ (resp. $D_1 \coprod_{ D_0} D_n$), and $\gamma$ is a pasting of unit constraints in the bicategory $\mathbf{hom}(D_{n-1},X(x,y))$.
		The composite of this pasting diagram provides the modification we are looking for, and thus we conclude the proof.
	\end{proof}
\end{lemma}
We prove the next two lemmas by a simultaneous induction on $n$.
\begin{lemma}
	\label{modif vert comp+whiskering 1-cell}
	Let $F\colon A \curvearrowright B$, $G\colon B \curvearrowright C$ be $n$-cylinders in $\Omega_m (X,\phi_1,\phi_2)$ where $(\phi_1,\phi_2)=\phi\colon S^{m-1} \rightarrow X$.
	Given a $1$-cell $h\colon a \rightarrow s^{n+m}(A)=s^{n+m}(B)=s^{n+m}(C)$ in $X$, we get a modification 
	\[\Theta^{g,f,h}\colon (Gh)\circ( Fh)\Rightarrow (G\circ F)h\colon Ah \curvearrowright Ch \]
	where $(\bullet)h$ denotes the operation of whiskering defined in Section \ref{subsect vert comp of cyls} and $\circ$ is the vertical composition of cylinders.
	\begin{proof}
		We prove the statement by induction, the case $n=0$ being straightforward.
		
		To begin with, we have to define a pair of $2$-cells $\Theta^{g,f,h}_s\colon (s^n(G)h) (s^n(F)h) \rightarrow (s^n(G)s^n(F))h$ and $\Theta^{g,f,h}_t\colon (t^n(G)t^n(F))h \rightarrow (t^n(G)h) (t^n(F)h) $ in $\Omega_m  (X,\phi_1 h, \phi_2 h)$. These are easily obtained from the contractibility of the globular sum $D_1\plus{ D_{0}} D_{m+1}\plus{D_{m}}D_{m+1}$. Indeed, one has the following string of equalities: \[s^{n+m}(A)=s^m(s^n(A))=s^m(t^n(A))=s^m(s(t^n(F)))=s^m(t(t^{n}(F)))=s^m(s(t^n(G)))\] which implies that there is a map $(h,t^n(F),t^n(G))\colon D_1\coprod_{ D_0} D_{m+1}\coprod_{D_m}D_{m+1} \rightarrow X$.
		
		For sake of simplicity, we denote by $f_{\epsilon}$ the $1$-cell $\epsilon^n(F)$ in $\Omega_m  (X,\phi_1,\phi_2)$ for $\epsilon=\sigma, \tau$, and similarly for $G$.
		
		We have the following diagram in the bicategory $\mathbf{hom}(D_{n-1}, \Omega_{m+1}(X,s^n(Ah),t^n(Ch)))$
		\[
		\bfig
		\morphism(0,0)|l|/@{>}/<-800,-400>[((g_t f_t)h)(Ah)`((g_t h) (f_t h)) (Ah) ;\Gamma(Ah,\Theta^{g,f,h}_t) \ \ ] 
		\morphism(-800,-400)|l|/@{>}/<0,-400>[((g_t h) (f_t h)) (Ah)`(g_t h) ((f_t h) (Ah)) ;\Psi] 
		\morphism(-800,-800)|l|/@{>}/<-800,-400>[(g_t h) ((f_t h) (Ah))`(g_t h)((Bh)(f_s h)) ;(g_t h)\overline{Fh}]
		\morphism(-800,-800)|r|/@{>}/<800,-400>[(g_t h) ((f_t h) (Ah))`(g_t h)((f_t A)h) ;(g_t h)\Phi]
		\morphism(0,-1200)|r|/@{>}/<-800,-400>[(g_t h)((f_t A)h)`(g_t h)((B f_s)h) ;(g_t h)(\overline{F}h)]
		\morphism(-800,-1600)|l|/@{>}/<-800,400>[ (g_t h)((B f_s)h)`(g_t h)((Bh)(f_s h));(g_t h)\Theta]
		\morphism(800,-800)|l|/@{>}/<-800,-400>[ ((g_t)(f_t A))h`(g_t h)((f_t A)h);]
		\morphism(0,0)|r|/@{>}/<800,-400>[((g_t f_t)h)(Ah)`((g_t f_t) A)h ;\Phi]
		\morphism(800,-400)|r|/@{>}/<0,-400>[ ((g_t f_t) A)h`((g_t)(f_t A))h;\Psi h]
		\morphism(800,-800)|r|/@{>}/<0,-800>[ ((g_t)(f_t A))h`(g_t(B f_s))h;(g_t \overline{F})h]
		\morphism(-800,-1600)|r|/@{>}/<1600,0>[ (g_t h)((B f_s)h)`(g_t(B f_s))h;]
		\morphism(800,-1600)|r|/@{>}/<0,-400>[(g_t(B f_s))h`((g_t B) f_s)h;\Psi h]
		\morphism(-1600,-1200)|l|/@{>}/<0,-800>[(g_t h)((Bh)(f_s h))`((g_t h)(B h))(f_s h);\Psi]
		\morphism(-1600,-2000)|a|/@{>}/<1600,0>[((g_t h)(B h))(f_s h)`((g_t B)h)(f_s h);\Phi (f_s h)]
		\morphism(800,-2000)|a|/@{>}/<-800,0>[((g_t B)f_s)h`((g_t B)h)(f_s h);]
		\morphism(-1600,-2000)|a|/@{>}/<0,-400>[((g_t h)(B h))(f_s h)`((Ch)(g_s h))(f_s h);(\overline{Gh}(f_s h))]
		\morphism(0,-2000)|l|/@{>}/<0,-400>[((g_t B)h)(f_s h)`((C g_s)h)(f_s h);((\overline{G}h)(f_s h))]
		\morphism(0,-2400)|l|/@{>}/<-1600,0>[((C g_s)h)(f_s h) `((Ch)(g_s h))(f_s h);\Theta (g_s h)]
		\morphism(800,-2000)|r|/@{>}/<0,-400>[((g_t B)f_s) h `((C g_s)f_s) h;(\overline{G}f_s)h]
		\morphism(0,-2400)|r|/@{>}/<800,0>[((C g_s)h)(f_s h) `((C g_s)f_s) h;]
		\morphism(-1600,-2400)|l|/{@{>}@/^0em/}/<800,-400>[((Ch)(g_s h))(f_s h)`(Ch)((g_s h)(f_s h));\Psi]
		\morphism(-800,-2800)|l|/@{>}/<1600,0>[(Ch)((g_s h)(f_s h))`(C(g_s f_s))h;]
		\morphism(800,-2400)|r|/@{>}/<0,-400>[((C g_s)f_s)h`(C(g_s f_s))h;\Psi h]
		\morphism(-800,-2800)|l|/@{>}/<800,-400>[(Ch)((g_s h)(f_s h))`(Ch) ((g_s f_s) h);\Upsilon(\Theta^{g,f,h}_s,Ch)]
		\morphism(800,-2800)|r|/@{>}/<-800,-400>[(C(g_s f_s))h`(Ch) ((g_s f_s) h);\Theta]
		\morphism(800,-400)|r|/{@{>}@/^6em/}/<0,-2400>[ ((g_t f_t) A)h`(C(g_s f_s))h;(\overline{G\circ F})h]
		
		\morphism(-100,-600)|a|/@{=>}/<200,0>[`;]
		\morphism(-900,-1200)|a|/@{=>}/<200,0>[`;\text{ind.hyp.}]
		\morphism(200,-1400)|a|/@{=>}/<200,0>[`;(1)]
		\morphism(-700,-1800)|a|/@{=>}/<200,0>[`;]
		\morphism(-1100,-2200)|a|/@{=>}/<200,0>[`;\text{ind.hyp.}]
		\morphism(300,-2200)|a|/@{=>}/<200,0>[`;(2)]
		\morphism(-100,-3000)|a|/@{=>}/<200,0>[`;]
		\morphism(-100,-2600)|a|/@{=>}/<200,0>[`;]
		\morphism(1100,-1800)|a|/@{=>}/<200,0>[`;\text{ind.hyp.}]
		\efig
		\]
		The $2$-cells filling this diagram either come from the inductive hypothesis oh this lemma and of the following one (when specified), from contractibility of appropriate globular sums (the unlabeled $2$-cells) or are of the form $(1)$ and $(2)$. The construction of $(2) $is similar to that of$ (1)$, which is the content of Lemma \ref{coherence modif 1}. The composite of this pasting diagram provides the $2$-cell we are looking for, the left-hand side (resp. right-hand side) composite being (isomorphic to) $\Upsilon(\Theta^{g,f,h}_s,Ch)\circ \overline{(Gh)\circ (Fh)}\circ \Gamma(Ah,\Theta^{g,f,h}_t)$ (resp.$\overline{(G\circ F)h}$).
	\end{proof}
\end{lemma}

\begin{lemma}
	\label{whiskering modif with 1-cell}
	Given a pair of $n$-cylinders $F,G\colon A \curvearrowright B$ in $\Omega_m (X ,\phi)$, a modification $\Lambda\colon F \Rightarrow G$ and a $1$-cell $c\colon b= t^{n+m}(B)\rightarrow b'$, we get an induced modification $c\Lambda\colon cF \Rightarrow cG$ between the $n$-cylinders $cF,cG\colon cA \curvearrowright cB$ in $\Omega_m (X, c\phi)$.
	\begin{proof}
		For sake of simplicity, we denote by $f_{\epsilon}$ the $1$-cell $\epsilon^n(F)$ in $\Omega_m (X,\phi_1,\phi_2)$ for $\epsilon=\sigma, \tau$, and similarly for $G$.
		
		Consider the bicategory $\mathbf{hom}(D_{n-1},\Omega_{m+1}(X,s^n(cA),t^n(cB))$, inside which we define the following $2$-dimensional pasting diagram:
		\[
		\bfig
		\morphism(0,0)|l|/@{>}/<-800,-400>[(c g_t)(cA)` (c f_t)(cA);\Gamma(cA,c\Lambda_t) ] 
		\morphism(-800,-400)|l|/@{>}/<0,-400>[(c f_t)(cA)` c (f_t A);\phi ] 
		\morphism(0,0)|r|/@{>}/<0,-400>[(c g_t)(cA)`c(g_t A);\phi ]
		\morphism(0,-400)|a|/@{>}/<-800,-400>[c(g_t A)`c(f_t A);c\Gamma(A,\Lambda_t) ]
		\morphism(-800,-800)|l|/@{>}/<0,-400>[c (f_t A)` c (B f_s);c\overline{F} ] 
		
		\morphism(0,-400)|r|/{@{>}@/^2em/}/<0,-800>[c(g_t A)`c(B g_s);c\overline{G} ]
		\morphism(0,-400)|l|/{@{>}@/^-2em/}/<0,-800>[c(g_t A)`c(B g_s);c(\Upsilon\circ \overline{F} \circ \Gamma) ]
		\morphism(-800,-1200)|a|/@{>}/<800,0>[ c (B f_s)` c (B g_s);c\Upsilon(\Lambda_s, B) ] 
		\morphism(-800,-1200)|a|/@{>}/<0,-400>[ c (B f_s)`(cB)(c f_s);\Theta] 
		\morphism(-800,-1600)|a|/@{>}/<800,0>[(cB)(c f_s)`(cB)(c g_s);\Upsilon(c\Lambda_s, cB) ] 
		\morphism(0,-1200)|r|/@{>}/<0,-400>[c (B g_s)`(cB)(c g_s);\Theta] 
		
		\morphism(-500,-400)|a|/@{=>}/<200,0>[`;]
		\morphism(-500,-1000)|a|/@{=>}/<200,0>[`;\text{ind.hyp.}]
		\morphism(-100,-800)|a|/@{=>}/<200,0>[`;\text{ind.hyp.}]
		\morphism(-500,-1400)|a|/@{=>}/<200,0>[`;]
		\efig
		\]
		The $2$-cells that fill the diagram either come from the inductive hypothesis of this lemma or the previous one, or by contractibility of suitable globular sums when unlabeled. The composite of this pasting diagram is the $2$-cell we are looking for, and so this concludes the proof.	
	\end{proof}
\end{lemma}
\begin{lemma}
	\label{coherence modif 1}
	Given an $n$-cylinder $F\colon A \curvearrowright B$ in $\Omega_{m}(X,\phi)$, a $1$-cell $g$ in $\Omega_{m}(X,\phi)$ and a 1-cell $h\colon a \rightarrow s^{n+m}(A)$ in $X$, such that $s^2(g)=t^{n+1}(A)=t^{n+1}(B)$, there is a modification $\chi$ as displayed below, where the cylinders denoted by $\lambda_1,\lambda_2$ are obtained by contractibility of the appropriate globular sum.
	\[
	\bfig
	\morphism(0,0)|l|/@{>}/<-600,0>[(gA)h`(gh)(Ah);\lambda_1]
	\morphism(-600,0)|l|/@{>}/<0,-400>[(gh)(Ah)`(gh)(Bh);(gh)(Fh)]
	\morphism(-600,-400)|l|/@{>}/<600,0>[(gh)(Bh)`(gB)h;\lambda_2]
	\morphism(0,0)|r|/{@{>}@/^0em/}/<0,-400>[(gA)h`(gB)h;(gF)h]
	\morphism(-400,-200)|a|/@{=>}/<200,0>[`;\chi]
	\efig
	\] 
	\begin{proof}
		Firstly, notice that the existence of such modification does not depend on the choice of $\lambda_1,\lambda_2$.
		By definition, given $\epsilon=s,t$, we have that $\epsilon^n(\lambda_2 \circ (gh)(Fh) \circ \lambda_1)$ is given by a composite 
		\[
		\bfig
		\morphism(0,0)|a|/@{>}/<800,0>[(g\epsilon^n(A))h`(gh)(\epsilon^n(A)h);\simeq]
		\morphism(800,0)|a|/@{>}/<1200,0>[(gh)(\epsilon^n(A)h)`(gh)(\epsilon^n(B)h);(gh)(\epsilon^n(F)h)]
		\morphism(2000,0)|a|/@{>}/<800,0>[(gh)(\epsilon^n(B)h)`(g \epsilon^n(B))h;\simeq]
		\efig
		\]
		where the first and the third map arise from contractibility of suitable globular sums.
		
		On the other hand, $\epsilon^n((gF)h)$ is given by $(g\epsilon^n(F))h\colon (g\epsilon^n(A))h \rightarrow (g\epsilon^n(B))h$. From these observations it is clear that we can find a pair of two cells $\chi_s,\chi_t$ as required in the definition of a modification.
		The rest of the proof follows analogously to that of the previous results, so it will be omitted.
	\end{proof}
\end{lemma}
The next lemma address the problem of contructing the whiskering operations. The other half that is required follows from a duality-kind argument.
\begin{lemma}
	\label{Lemma B8}
	Assume given $n$-cylinders $F\colon A \curvearrowright B, \ G\colon B \curvearrowright C$ together with a modification $\Theta\colon F \Rightarrow F'$ in $X$. Then there is an induced modification \[G*\Theta\colon G\circ F \rightarrow G\circ F' \]
	\begin{proof}
		The cases $n=0,1$ are pretty straightforward. To prove the inductive step, consider the following $2$-dimensional pasting diagram in the bicategory $\mathbf{hom}(D_{n-1}, X(s^n(A),t^n(C)))$:
		\[
		\bfig
		\morphism(0,0)|l|/@{>}/<-800,0>[(g_t f'_t) A` (g_t f_t)A;\Gamma(A,g_t\Theta_t) ]
		\morphism(-800,0)|l|/@{>}/<0,-400>[(g_t f_t)A` g_t (f_t A);\Psi ]
		\morphism(-800,-400)|l|/@{>}/<0,-400>[g_t (f_t A)` g_t (B f_s);g_t \overline{F} ]
		\morphism(-800,-800)|l|/@{>}/<0,-400>[g_t (B f_s)` (g_t B) f_s;\Psi ]
		\morphism(-800,-1200)|l|/@{>}/<0,-400>[(g_t B) f_s` (C g_s) f_s;\overline{G} f_s]
		\morphism(-800,-1600)|l|/@{>}/<0,-400>[(C g_s) f_s` C (g_s f_s);\Psi]
		\morphism(-800,-2000)|l|/@{>}/<800,0>[C (g_s f_s)` C (g_s f'_s);\Upsilon(g_s \Theta_s,C	)]
		
		\morphism(0,0)|r|/@{>}/<0,-400>[(g_t f'_t) A` g_t (f'_t A);\Psi ]
		\morphism(0,-400)|r|/@{>}/<0,-400>[g_t (f'_t A)` g_t (B f'_s);g_t \overline{F'} ]
		\morphism(0,-800)|r|/@{>}/<0,-400>[g_t (B f'_s)` (g_t B) f'_s;\Psi ]
		\morphism(0,-1200)|r|/@{>}/<0,-400>[(g_t B) f'_s` (C g_s) f'_s;\overline{G} f'_s]
		\morphism(0,-1600)|r|/@{>}/<0,-400>[(C g_s) f'_s` C (g_s f'_s);\Psi]
		\morphism(0,-400)|a|/@{>}/<-800,0>[g_t (f'_t A)` g_t (f_t A); g_t \Gamma(A, \Theta_t)]
		\morphism(-800,-800)|a|/@{>}/<800,0>[g_t (B f_s)` g_t (B f'_s);g_t \Upsilon(\Theta_s,B) ]
		\morphism(0,-1200)|r|/@{>}/<-800,0>[(g_t B) f'_s` (g_t B) f_s;]
		\morphism(-800,-1600)|l|/@{>}/<800,0>[(C g_s) f_s` (C g_s) f'_s;]
		
		\morphism(-500,-200)|a|/@{=>}/<200,0>[`;]
		\morphism(-500,-600)|a|/@{=>}/<200,0>[`;(1)]
		\morphism(-500,-1000)|a|/@{=>}/<200,0>[`;]
		\morphism(-500,-1400)|a|/@{=>}/<200,0>[`;(2)]
		\morphism(-500,-1800)|a|/@{=>}/<200,0>[`;]
		\efig
		\]
		The unlabeled cells come from the contractibility of the appropriate globular sums, while $(1)$ is provided by Lemmas \ref{modif vert comp+whiskering 1-cell} and \ref{whiskering modif with 1-cell}. Finally, the $2$-cell labeled with $(2)$ is constructed in the following lemma.
	\end{proof}
\end{lemma}
\begin{lemma}
	Given an $n$-cylinder $G\colon A \curvearrowright B$ in $\Omega_{m}(X,\phi)$ and a $2$-cell in $X$ \[\xymatrix{
		**[l] a \rtwocell_{f'}^{f}{\alpha} & **[r]s^{n+m}(A)\\}\] 
	we get an induced modification 
	\[
	\bfig
	\morphism(0,0)|a|/@{>}/<400,0>[Af'` Af;\Lambda_1]
	\morphism(0,0)|l|/@{>}/<0,-400>[Af'` Bf';Gf']
	\morphism(400,0)|r|/@{>}/<0,-400>[Af` Bf;Gf]
	\morphism(400,-400)|a|/@{>}/<-400,0>[Bf` Bf';\Lambda_2]
	
	\morphism(100,-200)|a|/@{=>}/<200,0>[`;\Delta]
	\efig
	\]
	Here, $\Lambda_1$ and $\Lambda_2$ are obtained by contractibility of appropriate globular sums, and the existence of $\Delta$ does not depend on the choice of these.
	\begin{proof}
		The pair of $2$-cells $\Delta_s,\Delta_t$ is obtained by contractibility of suitable globular sums, and the modification $\Delta$ is given by the composite of the following $2$-dimensional pasting diagram in the bicategory $\mathbf{hom}(D_{n-1}, \Omega_m (X,\phi)(s^n(Af'),t^n(Bf'))$:
		\[
		\bfig
		\morphism(0,0)|a|/@{>}/<-1100,0>[(g_t f')(Af')`( (1 (g_t f)) 1) (A f');\Gamma(Af',\Delta_t)]
		\morphism(-1100,0)|l|/@{>}/<0,-400>[( (1 (g_t f)) 1) (A f')`(1 (g_t f)) (1 (A f'));\Psi]
		\morphism(-1100,-400)|l|/@{>}/<0,-400>[(1 (g_t f)) (1 (A f'))`(1 (g_t f)) ((Af) 1);(1 (g_t f)) \overline{\Lambda}_1]
		\morphism(-1100,-800)|l|/@{>}/<0,-400>[(1 (g_t f)) ((Af) 1)`(1 ((g_t f)(Af))) 1;\Psi]
		\morphism(-1100,-1200)|l|/@{>}/<400,-400>[(1 ((g_t f)(Af))) 1`(1 ((Bf)(g_s f))) 1;(1\overline{Gf} )1]
		\morphism(-700,-1600)|l|/@{>}/<0,-400>[(1 ((Bf)(g_s f))) 1`(1 (Bf))((g_s f)1);\Psi]
		\morphism(-700,-2000)|a|/@{>}/<1300,0>[(1 (Bf))((g_s f)1)`((Bf') 1)((g_s f)1);\Lambda_2((g_s f)1)]
		\morphism(600,-2000)|l|/@{>}/<400,-400>[((Bf') 1)((g_s f)1)`(Bf') (1 ((g_s f)1));\Psi]
		\morphism(1000,-2400)|r|/@{>}/<400,400>[(Bf') (1 ((g_s f)1))`(Bf') (g_s f');\Upsilon(\Delta_s,Bf')]
		
		\morphism(0,0)|r|/@{>}/<0,-400>[(g_t f')(Af')`(g_t A)f';\Phi]
		\morphism(0,-400)|r|/@{>}/<0,-400>[(g_t A)f'`(g_t A)f;\eta_1]
		\morphism(0,-800)|r|/@{>}/<0,-400>[(g_t A)f`(1((g_t A)f))1;]
		\morphism(-1100,-1200)|l|/@{>}/<1100,0>[(1 ((g_t f)(Af))) 1`(1((g_t A)f))1;]
		\morphism(0,-1200)|l|/@{>}/<600,-400>[(1((g_t A)f))1`(1((B g_s)f))1;(1 (\overline{G}f))1]
		\morphism(600,-1600)|r|/@{>}/<-1300,0>[(1((B g_s)f))1`(1 ((Bf)(g_s f))) 1;]
		\morphism(600,-1600)|r|/@{>}/<0,400>[(1((B g_s)f))1`(B g_s) f;]
		\morphism(0,-800)|a|/{@{>}@/^0em/}/<600,-400>[(g_t A) f`(B g_s)f;\overline{G}f]
		\morphism(600,-1200)|r|/@{>}/<800,0>[(B g_s) f`(B g_s) f';\eta_2]
		\morphism(0,-400)|r|/{@{>}@/^0em/}/<1400,-800>[(g_t A)f'`(B g_s)f';\overline{G}f']
		\morphism(1400,-1200)|r|/@{>}/<0,-800>[(B g_s) f'`(B f') (g_s f');\Theta]

		\morphism(-650,-600)|a|/@{=>}/<200,0>[`;]				\morphism(-650,-1400)|a|/@{=>}/<200,0>[`;(1)]
		\morphism(600,-1800)|a|/@{=>}/<200,0>[`;]	
		\morphism(250,-1300)|a|/@{=>}/<200,0>[`;(2)]	
		\morphism(500,-1000)|a|/@{=>}/<200,0>[`;ind.hyp.]	
		\efig
		\]
		The unlabelled cells are obtained by contractibility of suitable globular sums. The existence of the $2$-cell denoted by $(1)$ is ensured by Lemmas \ref{modif vert comp+whiskering 1-cell} and \ref{whiskering modif with 1-cell}, $(2)$ is constructed in the next lemma and the remaining $2$-cell exists by inductive hypothesis, as indicated.
	\end{proof}
\end{lemma}
\begin{lemma}
	Assume given an $n$-cylinder $C\colon A \curvearrowright B$ in $\Omega_{m} (X,\phi)$ and a choice of an identity $1$-cell $1_a\colon a \rightarrow a$ in $X$, where $a=s^{n+m}(A)$. We then get a modification of the following form:
	\[
	\bfig
	\morphism(0,0)|l|/@{>}/<-400,0>[A`A1_a;\Lambda_1]
	\morphism(-400,0)|l|/@{>}/<0,-400>[A1_a`B1_a;C1_a]
	\morphism(-400,-400)|l|/@{>}/<400,0>[B1_a`B;\Lambda_2]
	\morphism(0,0)|r|/@{>}/<0,-400>[A`B;C]
	\morphism(-300,-200)|a|/@{=>}/<200,0>[`;\beta]
	\efig
	\]
	Again, $\Lambda_1$ and $\Lambda_2$ are obtained by contractibility of the appropriate globular sums and the existence of $\beta$ does not depend on a choice of such.
	\begin{proof}
		The $1$-cells $\epsilon^n(\Lambda_i)$ in $\Omega_{m} (X,\phi)$, for $\epsilon=s,t$ and $i=1,2$, are obtained by contractibility of $D_n$ and are therefore identity cells (having the same surce and target). For this reason, we denote all of them by $1$, as no confusion should arise.
		
		The pair of $2$-cells $\beta_s,\beta_t$ is obtained by contractibility, and we choose $\beta$ to be induced by the composite of the following $2$-dimensional pasting diagram in $\mathbf{hom}(D_{n-1}, \Omega_m (X,\phi)(s^n(A),t^n(B)))$:
		\[
		\bfig
		\morphism(0,0)|l|/@{>}/<-1000,0>[c_t A`(1(c_t1_a)1)A;\Gamma(A,\beta_t)]
		\morphism(-1000,0)|l|/@{>}/<0,-400>[(1(c_t1_a)1)A`(1(c_t1_a))(1A);\Psi]
		\morphism(-1000,-400)|l|/@{>}/<0,-400>[(1(c_t1_a))(1A)`(1(c_t1_a))((A1_a)1);(1(c_t 1_a))\overline{\Lambda_1}]
		\morphism(-1000,-800)|a|/@{>}/<1000,0>[(1(c_t1_a))((A1_a)1)`(1 ((c_t1_a)(A1_a)))1;\Psi]
		\morphism(0,-800)|a|/@{>}/<1500,0>[(1 ((c_t1_a)(A1_a)))1`(1 ((B1_a)(c_s 1_a)))1;(1 (\overline{C}1_a))1]
		\morphism(1500,-800)|a|/@{>}/<0,400>[(1 ((B1_a)(c_s 1_a)))1`(1 (B1_a))((c_s 1_a)1);\Psi]
		\morphism(1500,-400)|r|/@{>}/<0,400>[(1 (B1_a))((c_s 1_a)1)`B((c_s 1_a)1);\overline{\Lambda_2}((c_s 1_a)1)]
		\morphism(1500,0)|l|/@{>}/<-900,0>[B((c_s 1_a)1)`Ba;\Upsilon(\beta_s,B)]
		
		\morphism(0,0)|r|/{@{>}@/^0em/}/<0,-800>[c_t A`(1 ((c_t1_a)(A1_a)))1;\gamma_1]
		\morphism(1500,-800)|l|/{@{>}@/^0em/}/<-900,800>[(1 ((B1_a)(c_s 1_a)))1`Ba;\gamma_2]
		
		\morphism(0,0)|a|/{@{>}@/^0em/}/<600,0>[c_t A`Ba;\overline{C}]
		\morphism(-600,-400)|a|/@{=>}/<200,0>[`;]
		\morphism(1200,-200)|a|/@{=>}/<-200,0>[`;]
		\morphism(300,-400)|a|/@{=>}/<0,200>[`;ind.hyp.]
		\efig
		\]
		where the unlabelled $2$-cells arise from contractibility of the appropriate globular sums, and the remaining one comes from the inductive hypothesis.
	\end{proof}
\end{lemma}
The next lemma provides the unit constraint for the bicategory structure on $\mathbf{hom}(D_n,X)$. We only prove one side of the unit constraint, the other one being analogous.
\begin{lemma}
	\label{unitality of vert comp}
	Given an $n$-cylinder $C\colon A \curvearrowright B$ there exists a modification $\upsilon \colon C\circ \mathbf{C}_A \Rightarrow C$.
	\begin{proof}
		The existence of the pair of $2$-cells $\upsilon_s,\upsilon_t$ is straightforward.
		
		Consider the following pasting diagram in the bicategory $\mathbf{hom}(D_{n-1},X(s^n(A),t^n(B)))$, where $a=s^n(A),b=t^n(B)$:
		\[
		\bfig
		\morphism(0,0)|l|/@{>}/<-1200,0>[c_t A`(c_t 1_a)A;\Gamma(A,\upsilon_t)]
		\morphism(-1200,0)|l|/@{>}/<0,-400>[(c_t 1_b)A`c_t (1_b A);\Psi]
		\morphism(-1200,-400)|l|/@{>}/<0,-400>[c_t (1_b A)`c_t (A 1_a);c_t \overline{\mathbf{C}_A}]
		\morphism(-1200,-800)|l|/@{>}/<0,-400>[c_t (A 1_a)`(c_t A) 1_a;\Psi]
		\morphism(-1200,-1200)|l|/@{>}/<600,0>[(c_t A) 1_a`(B c_s) 1_a;\overline{C} 1_a]
		\morphism(-600,-1200)|l|/@{>}/<600,0>[(B c_s) 1_a`B(c_s 1_a);\Psi]
		\morphism(0,-1200)|r|/@{>}/<0,600>[B (c_s 1_a)`B c_s;\Upsilon(\upsilon_t,A)]
		
		\morphism(0,0)|l|/@{>}/<-1200,-1200>[c_t A`(c_t A)1_a;\lambda_1]
		\morphism(-600,-1200)|l|/@{>}/<600,600>[(B c_s) 1_a`B c_s ;\lambda_2]
		\morphism(0,0)|r|/{@{>}@/^0em/}/<0,-600>[c_t A`B c_s;\overline{C}]
		
		\morphism(-800,-400)|a|/@{=>}/<200,0>[`;]
		\morphism(-400,-700)|a|/@{=>}/<150,150>[`;(1)]
		\morphism(-200,-1100)|a|/@{=>}/<0,200>[`;]
		\efig
		\]
		The unlabeled $2$-cells come from contractibility of appropriate globular sums, as well as $\lambda_1$ and $\lambda_2$, and the $2$-cell labeled with $(1)$ is provided by the previous lemma.
	\end{proof}
\end{lemma} 
We now turn to the final construction, that of the associator for the bicategory $\mathbf{hom}(D_n,X)$. We start with a preliminary lemma
\begin{lemma}
	\label{associativity whiskering cylinders}
	Given an $n$-cylinder $F\colon A \curvearrowright B$ in $\Omega_m(X,\phi)$, and a pair of composable $1$-cells $h\colon t^{n+m}(A) \rightarrow b, \ g\colon b \rightarrow c$, there is a modification
	\[
	\bfig
	\morphism(0,0)|l|/@{>}/<-600,0>[h(gA)`(hg)A;\lambda_1]
	\morphism(-600,0)|l|/@{>}/<0,-400>[(hg)A`(hg)B;(hg)F]
	\morphism(-600,-400)|l|/@{>}/<600,0>[(hg)B`h(gB);\lambda_2]
	\morphism(0,0)|r|/@{>}/<0,-400>[h(gA)`h(gB);h(gF)]
	\morphism(-400,-200)|a|/@{=>}/<200,0>[`;\zeta]
	\efig
	\]
	Here, $\lambda_1,\lambda_2$ come from the contractibility of $D_n \coprod_{ D_0} D_1  \coprod_{ D_0} D_1$, and the existence of $\zeta$ does not depend on the choice of such cylinders.
	\begin{proof}
		We denote the $1$-cells $\epsilon^n(\lambda_i)$, for $\epsilon=s,t$ and $i=1,2$ with $a$, being an instance of an associativity constraint.
		
		The $2$-cells $\zeta_s,\zeta_t$ arise from contractibility of appropriate globular sums, and the modification we are looking for is given by the composite of the following $2$-dimensional pasting diagram in the bicategory $\mathbf{hom}(D_{n-1},\Omega_m(X,\phi)(s^n(h(gA)),t^n(h(gB))))$:
		\[
		\bfig
		\morphism(0,0)|a|/@{>}/<-800,-400>[(h(gf_t))(h(gA))`(a(((hg)f_t)a))(h(gA));\Gamma(h(gA),\zeta_t)]
		\morphism(-800,-400)|l|/@{>}/<0,-400>[(a(((hg)f_t)a))(h(gA))`(a((hg) f_t))(a(h(gA)));\Psi]
		\morphism(-800,-800)|l|/@{>}/<0,-400>[(a((hg) f_t))(a(h(gA)))`(a((hg) f_t))(((hg)A)a);(a((hg) f_t))\overline{\lambda_1}]
		\morphism(-800,-1200)|l|/@{>}/<0,-400>[(a((hg) f_t))(((hg)A)a)`(a(((hg)f_t)((hg)A)))a;\Psi]
		\morphism(-800,-1600)|l|/@{>}/<0,-400>[(a(((hg)f_t)((hg)A)))a`(a(((hg)B)((hg)f_s)))a;(a(\overline{(hg)F}))a]
		\morphism(-800,-2000)|a|/@{>}/<1200,-200>[(a(((hg)B)((hg)f_s)))a`(a((hg)B))(((hg) f_s) a);\Psi]
		\morphism(400,-2200)|r|/@{>}/<1400,400>[(a((hg)B))(((hg) f_s) a)`(h((gB)a))(((hg) f_s) a); \ \  \ \overline{\lambda_2}(((hg) f_s) a)]
		\morphism(1800,-1800)|r|/@{>}/<0,400>[(h((gB)a))(((hg) f_s) a)`(h(gB))(h(gf_s));\Upsilon(\zeta_s,h(gB)) \ \ ]
		
		\morphism(-800,-1600)|a|/@{>}/<1400,400>[(a(((hg)f_t)((hg)A)))a`(a((hg)(f_t A)))a;(a\Phi)a]
		\morphism(600,-1200)|a|/@{>}/<0,-400>[(a((hg)(f_t A)))a`(a((hg)(B f_s)))a;(a((hg)\overline{F}))a]
		\morphism(600,-1600)|a|/@{>}/<-1400,-400>[(a((hg)(B f_s)))a`(a(((hg)B)((hg)f_s)))a;(a\Theta)a]

		\morphism(0,0)|r|/@{>}/<600,-400>[(h(gf_t))(h(gA))`h((g f_t)A);\Phi]
		\morphism(600,-400)|a|/@{>}/<0,-400>[h((g f_t)A)`h(g(f_t A));\Phi]
		\morphism(600,-800)|r|/@{>}/<600,-200>[h(g(f_t A))`h(g(B f_s));h(g(\overline{F}))]
		\morphism(600,-800)|l|/{@{>}@/^0em/}/<0,-400>[h(g(f_t A))`(a((hg)(f_t A)))a;\delta_1]
		\morphism(600,-1600)|r|/{@{>}@/^-3em/}/<600,600>[(a((hg)(B f_s)))a`h(g(B f_s));\delta_2]
		
		\morphism(1200,-1000)|r|/@{>}/<200,600>[h(g(B f_s))`h((gB)(g f_s));\Theta]
		\morphism(600,-400)|a|/{@{>}@/^0em/}/<800,0>[h((g f_t)A)`h((gB)(g f_s));h(\overline{gF})]
		\morphism(1400,-400)|r|/{@{>}@/^0em/}/<400,-1000>[h((gB)(g f_s))`(h(gB))(h(gf_s));\Theta]
		
		\morphism(-100,-800)|a|/@{=>}/<200,0>[`;]
		\morphism(-100,-1600)|a|/@{=>}/<200,0>[`;(1)]
		\morphism(800,-1400)|a|/@{=>}/<200,0>[`;(1)]
		\morphism(900,-600)|a|/@{=>}/<200,0>[`;ind.hyp.]
		\morphism(400,-1900)|a|/@{=>}/<200,0>[`;]
		\efig
		\]
		Here, $\delta_1,\delta_2$ and the unlabelled $2$-cells come from contractibility of suitable globular sums. On the other hand, the $2$-cells labelled with $(1)$ come from  Lemmas \ref{modif vert comp+whiskering 1-cell} and \ref{whiskering modif with 1-cell} and the remaining one comes from the inductive hypothesis, as indicated.
	\end{proof}
\end{lemma}
Finally, here is the construction of the modification representing the associativity constraint in the bicategory $\mathbf{hom}(D_n,X)$.
\begin{lemma}
	Given a composable triple of $n$-cylinders $F\colon A \curvearrowright B, \ G\colon B \curvearrowright C$ and $H\colon C \curvearrowright D$ in $X$, there exists a modification 
	\[\alpha\colon ( H\circ G) \circ F \Rightarrow H\circ (G \circ F) \]
	\begin{proof}
		The required $2$-cells $\alpha_s,\alpha_t$ are simply instances of associativity constraints for composition of $1$-cells in the coherator $\mathfrak{C}$.
		The modification $\alpha$ is induced by composing the following $2$-dimensional pasting diagram in the bicategory $\mathbf{hom}(D_{n-1},X(s^n(A),t^n(C)))$:
		\[
		\bfig
		\morphism(0,0)|l|/@{>}/<-800,-200>[(h_t( g_t f_t)) A`((h_t g_t)f_t)A;\Gamma(A,\alpha_t)]
		\morphism(-800,-200)|l|/@{>}/<0,-400>[((h_t g_t)f_t)A`(h_t g_t)(f_t A);\Psi]
		\morphism(-800,-600)|l|/@{>}/<0,-400>[(h_t g_t)(f_t A)`(h_t g_t)(B f_s);(h_t g_t)\overline{F}]
		\morphism(-800,-1000)|l|/@{>}/<0,-800>[(h_t g_t)(B f_s)`((h_t g_t)B)f_s;\Psi]
		\morphism(-800,-1800)|l|/@{>}/<0,-400>[((h_t g_t)B)f_s`(D(h_s g_s))f_s;\overline{H\circ G}f_s]
		\morphism(-800,-2200)|l|/@{>}/<0,-400>[(D(h_s g_s))f_s`D((h_s g_s) f_s);\Psi]
		\morphism(-800,-2600)|l|/@{>}/<400,-400>[D((h_s g_s) f_s)`D(h_s(g_s f_s));\Upsilon(\alpha_s,D) \ \ ]
		
		\morphism(0,0)|r|/@{>}/<1000,-200>[(h_t( g_t f_t)) A`h_t ((g_t f_t)A);\Psi]
		\morphism(1000,-200)|a|/@{>}/<0,-600>[h_t ((g_t f_t)A)`h_t (C (g_s f_s));h_t \overline{G\circ F}]
		\morphism(1000,-800)|r|/@{>}/<800,-800>[h_t (C (g_s f_s))`(h_t C)(g_s f_s);\Psi]
		\morphism(1800,-1600)|r|/{@{>}@/^5em/}/<-1800,-1000>[(h_t C)(g_s f_s)`(D h_s)(g_s f_s);\overline{H}(g_s f_s)]
		\morphism(0,-2600)|r|/{@{>}@/^0em/}/<-400,-400>[(D h_s)(g_s f_s)`D(h_s(g_s f_s));\Psi]
		\morphism(1000,-200)|a|/@{>}/<-1000,-400>[h_t ((g_t f_t)A)`h_t (g_t(f_t A));h_t \Psi]
		\morphism(0,-600)|r|/@{>}/<0,-400>[h_t (g_t(f_t A))`h_t (g_t(B f_s));h_t (g_t \overline{F})]
		\morphism(0,-1000)|r|/@{>}/<0,-400>[h_t (g_t(B f_s))`h_t ((g_t B) f_s);h_t\Psi]
		\morphism(0,-1400)|a|/@{>}/<1000,0>[h_t ((g_t B) f_s)`h_t ((C g_s) f_s);h_t(\overline{G}f_s)]
		\morphism(1000,-1400)|l|/@{>}/<0,600>[h_t ((C g_s) f_s)`h_t (C (g_s f_s));h_t\Psi]
		
		\morphism(0,-600)|a|/@{>}/<-800,0>[h_t (g_t(f_t A))`(h_t g_t)(f_t A);\eta_1]
		\morphism(-800,-1000)|l|/@{>}/<800,0>[(h_t g_t)(B f_s)`h_t(g_t(B f_s));\eta_2]
		
		\morphism(-800,-1800)|a|/@{>}/<800,0>[((h_t g_t)B)f_s`(h_t (g_t B))f_s;\Psi f_s]
		\morphism(0,-1800)|a|/@{>}/<1000,0>[(h_t (g_t B))f_s`(h_t (C g_s))f_s;(h_t \overline{G})f_s]
		\morphism(1000,-1800)|a|/@{>}/<0,-400>[(h_t (C g_s))f_s`((h_t C) g_s)f_s;\Psi f_s]
		\morphism(1000,-2200)|a|/@{>}/<-1000,0>[((h_t C) g_s)f_s`((D h_s) g_s)f_s; (\overline{H}g_s) f_s]
		\morphism(0,-2200)|a|/@{>}/<-800,0>[((D h_s) g_s)f_s`(D (h_s g_s))f_s; \Psi f_s]
		\morphism(0,-1400)|l|/@{>}/<0,-400>[h_t ((g_t B) f_s)`(h_t(g_t B))f_s;\mu_1]
		\morphism(1000,-1800)|a|/@{>}/<0,400>[(h_t (C g_s))f_s`h_t((C g_s)f_s);\mu_2]
		\morphism(1800,-1600)|r|/@{>}/<-800,-600>[(h_t C)(g_s f_s)`((h_t C)g_s)f_s;\nu_1]
		\morphism(0,-2200)|a|/@{>}/<0,-400>[((D h_s) g_s)f_s`(D h_s) (g_s f_s); \nu_2]
		
		\morphism(-100,-200)|a|/@{=>}/<200,0>[`;]
		\morphism(-500,-800)|a|/@{=>}/<200,0>[`;(1)]
		\morphism(400,-800)|a|/@{=>}/<200,0>[`;(0)]
		\morphism(-500,-1400)|a|/@{=>}/<200,0>[`;]
		\morphism(400,-1600)|a|/@{=>}/<200,0>[`;(2)]
		\morphism(1200,-1600)|a|/@{=>}/<200,0>[`;]
		\morphism(-100,-2000)|a|/@{=>}/<200,0>[`;(0)]
		\morphism(500,-2400)|a|/@{=>}/<200,0>[`;(3)]
		\morphism(-500,-2400)|a|/@{=>}/<200,0>[`;]
		\efig
		\]
		Here, the unlabelled $2$-cells and the $1$-cells $\eta_i,\mu_i$ and $\nu_i$ for $i=1,2$ all come from contractibility of suitable globular sums.
		The $2$-cells labelled with $(0)$ have been constructed in Lemmas \ref{modif vert comp+whiskering 1-cell} and \ref{whiskering modif with 1-cell}. Finally, $(1)$ is constructed in Lemma \ref{associativity whiskering cylinders}, and $(2)$ and $(3)$ are built up in an analogous way.
	\end{proof}
\end{lemma}
We conclude this section of the Appendix with the following results, which requires the existence of inverses and does not hold true in $\mathbf{Mod}(\mathfrak{C})$.
\begin{lemma}
	\label{inverse of modifications}
Given a pair of $n$-cylinders $F,G\colon \cyl(D_n)\rightarrow X$ in $\mathbf{Mod}(\mathfrak{C}^{\mathbf{W}})$ (see Definition \eqref{D_W def}) and a modification $\Theta\colon F\rightarrow G$ there exists a modification $\Theta' \colon G \rightarrow F$
\begin{proof}
We denote by $f^{-1}$ the result of promoting either a left or a right inverse for $f$ to a two-sided inverse.
If $n=0$ then $\Theta'$ is obtained by inverting the 2-cell $\Theta$. Let $n>0$, we define $\Theta'_s=(\Theta_s)^{-1}$ and $\Theta'_t=(\Theta_t)^{-1}$. By definition, $\Theta$ induces a modification of $(n-1)$-cylinders of the form $\overline{\Theta} \colon\Upsilon(C_0,\Theta_t)\otimes \bar{C}\otimes\Gamma(\Theta_s,C_1)\Rightarrow \bar{D}$ (where $\otimes$ denotes the vertical composition operation). By inductive hypothesis this can be inverted, to give us $\overline{\Theta}' \colon \bar{D}\Rightarrow \Upsilon(C_0,\Theta_t)\otimes \bar{C}\otimes\Gamma(\Theta_s,C_1)$. Lemma \ref{Lemma B8} implies that we get a modification:
\[
\bfig 
		\morphism(0,0)|a|/@{=>}/<0,-400>[\Upsilon(C_0,(\Theta_t)^{-1})\otimes \bar{D} \otimes \Gamma((\Theta_s)^{-1},C_1)`\Upsilon(C_0,(\Theta_t)^{-1})\otimes \Upsilon(C_0,\Theta_t)\otimes \bar{C}\otimes\Gamma(\Theta_s,C_1) \otimes \Gamma((\Theta_s)^{-1},C_1);\Upsilon(C_0,(\Theta_t)^{-1})\overline{\Theta}'\Gamma((\Theta_s)^{-1}]
\efig 
\]by whiskering.

Now, the existence of 3-cells $\Theta_s \circ (\Theta_s)^{-1}\rightarrow 1_{t(\Theta_s)}$ and $(\Theta_t)^{-1}\circ (\Theta_t)\rightarrow 1_{s(\Theta_t)} $ implies that there is an induced modification (using the usual methods to produce such modification in $\mathbf{Mod}(\mathfrak{C})$, applied to the globular sums $D_3\plus{D_0}D_n$ and $D_n \plus{D_0}D_3$):
\[\Upsilon(C_0,(\Theta_t)^{-1})\otimes \Upsilon(C_0,\Theta_t)\Rightarrow \Upsilon(C_0,1_{t(\Theta_s)})\] and 
\[\Gamma(\Theta_s,C_1) \otimes \Gamma((\Theta_s)^{-1},C_1) \Rightarrow \Gamma(1_{s(\Theta_t)},C_1)\]
The usual methods can also be employed to construct modifications $\Gamma(1_{s(\Theta_t)},C_1)\Rightarrow \mathbf{C}_{n-1}(C_1 C_s)$ and $\Upsilon(C_0,1_{t(\Theta_s)})\Rightarrow \mathbf{C}_{n-1}(C_t C_0)$, so that upon composing the $(n-1)$-modification defined so far we get one of the form:
\[
\bfig 
\morphism(0,0)|a|/@{=>}/<0,-400>[\Upsilon(C_0,(\Theta_t)^{-1})\otimes \bar{D} \otimes \Gamma((\Theta_s)^{-1},C_1)`\mathbf{C}_{n-1}(C_t C_0)\otimes \bar{C}\otimes \mathbf{C}_{n-1}(C_1 C_s);]
\efig 
\]
 We can now finish the construction of $\overline{\Theta'}$ by invoking Lemma \ref{unitality of vert comp}.
\end{proof}
\end{lemma}

\end{document}